\documentclass[11pt,a4paper]{article}
\RequirePackage[hmargin=1.25in,vmargin=1in]{geometry} %页面样式
\RequirePackage{graphicx}		%插图
\RequirePackage{xcolor}			%颜色
\RequirePackage{listings}		%关键字高亮的代码环境lstlisting
\RequirePackage{amsmath}		%AMS数学宏集
\RequirePackage{amsfonts}		%AMS字体
\RequirePackage{amssymb}		%数学公式
\RequirePackage{amsthm}			%定理环境
\RequirePackage{bm}				%公式加粗
\RequirePackage{diagbox}		%表格
\RequirePackage{multirow}		%表格
\RequirePackage{booktabs}		%三线表
\usepackage{float}
\RequirePackage{hyperref} 
\newtheorem{thm}{Theorem}[section]
\newtheorem{lem}{Lemma}[section]
\newtheorem{dfn}{Definition}[section]
\newtheorem{prp}{Proposition}[section]

\newtheorem{Ass}{Assumption}[section]
\newtheorem{remark}{Remark}[section]
\title{General Leader-Follower Linear-Quadratic Stochastic Graphon Games with Indefinite Control Weights\thanks{This work is supported by National Key R\&D Program of China (2022YFA1006104), National Natural Science Foundations of China (12471419, 12271304), and Shandong Provincial Natural Science Foundation (ZR2024ZD35).}}
\author{\normalsize
	Weijia Chen\thanks{\textit{School of Mathematics, Shandong University, Jinan 250100, P.R. China, E-mail: 202520301@mail.sdu.edu.cn}},
	\ Jingtao Shi\thanks{\textit{Corresponding author, School of Mathematics, Shandong University, Jinan 250100, P.R. China, E-mail: shijingtao@sdu.edu.cn}}\\ [8pt]
{\it \small Dedicated to Professor Jiongmin Yong on the Occasion of His 70th Birthday}}
\date{}

\begin{document}

\maketitle

\begin{abstract}
	This paper studies general leader-follower linear-quadratic stochastic graphon games with indefinite control weights. The model consists of one leader and a continuum of followers, where the followers interact through graphon aggregate terms and are coupled with the leader through their dynamics and cost functionals. Compared with existing related models, the state equations considered here allow more general diffusion terms, including diffusion coefficients depending on the states, controls, leader variables, and graphon aggregate terms. Using the stochastic maximum principle, we characterize the followers' Nash response and the leader's optimal control through coupled forward-backward stochastic systems with graphon aggregate terms. Under suitable Riccati solvability and convexity conditions, we construct a Stackelberg-Nash equilibrium for the limiting graphon game. We further design decentralized strategies for the associated finite-player network game and prove that they form an approximate Stackelberg-Nash equilibrium, which yields a propagation of chaos result. The framework covers, as special cases, indefinite linear-quadratic stochastic graphon games, Stackelberg mean field games, and stochastic graphon games with random effects.
\end{abstract} 


	\vspace{2mm}
	

	\tableofcontents
	
\section{Introduction}

{\it Mean-field games} (MFGs) were first proposed by Lasry and Lions \cite{Lasry-Lions-2007} and Huang et al. \cite{Huang-Mallahme-2006}, providing an effective limiting analysis framework for large-scale noncooperative stochastic games. The core idea is to approximate the complex interactions among finitely many players by a mean field. When the players satisfy homogeneity and anonymity conditions, the limiting equilibrium can typically be used in reverse to construct an approximate Nash equilibrium for the finite-player game, and the convergence of the finite system to the mean-field limit is characterized by a propagation of chaos result. For the theory of MFGs, see \cite{Bensoussan-Frehse-Yam-2013,Carmona-Delarue-2018}.

However, in many networked systems, the influence among players is not a globally uniform average, but is determined by the underlying connection structure. When the interactions are described by a graph or a random graph, the homogeneity and anonymity assumptions in classical MFGs are no longer natural. Early works studied MFGs on finite graphs and Erd\H{o}s-R\'enyi random graphs \cite{Gueant-2015,Delarue-2017}. A graphon can be regarded as the limit of a sequence of dense graphs \cite{Lovasz-2012}, and provides a natural tool for describing weighted adjacency structures in large-scale heterogeneous networks. Gao and Caines introduced graphon structures into large-scale network control \cite{Gao-Caines-2017,Gao-Caines-2018,Gao-Caines-2019a,Gao-Caines-2019b,Gao-Caines-2020}; Parise and Ozdaglar proposed static graphon games \cite{Parise-Ozdaglar-2019,Parise-Ozdaglar-2023}; Carmona et al. further established a rigorous mathematical framework for static stochastic graphon games \cite{Carmona-Coopey-Graves-Lauriere-2022}.

In the dynamic game setting, Caines and Huang proposed the theory of {\it graphon mean field games} (GMFGs), also referred to as stochastic graphon games, for the analysis and control of large-scale noncooperative dynamic game systems on networks \cite{Caines-Huang-2018,Caines-Huang-2019,Caines-Huang-2021}. Aurell et al. studied {\it linear-quadratic} (LQ) stochastic graphon games, used the rich Fubini extension to handle measurability issues arising from a continuum of players and essentially pairwise independent noises, established Nash equilibria for LQ stochastic graphon games, and proved the corresponding propagation of chaos result \cite{Aurell-Carmona-Lauriere-2022}. Bayraktar et al. studied graphon mean field systems as well as their stability and laws of large numbers \cite{Bayraktar-Chakraborty-Wu-2023}, and further analyzed propagation of chaos for {\it forward-backward stochastic differential equations} (FBSDEs) with graphon interactions \cite{Bayraktar-Wu-Zhang-2023}. Coppini et al. extended the related convergence theory to nonlinear graphon mean field systems \cite{Coppni-DeCrescenzo-Pham-2025}. In recent years, GMFGs have also continued to develop in directions such as common noise, risk sensitivity, infinite horizon problems, jump processes, reinforcement learning, and applied modeling \cite{Xu-Gou-Huang-Gao-2025,Chen-Huang-2026,FoguenTchuendom-Gao-Caines-Huang-2024,Amini-Cao-Sulem-2026,Plank-Zhang-2025,Tangpi-Zhou-2024,Zhang-Tan-Wnag-Yang-2024,Aurell-Carmona-Dayanikli-Lauriere-2022,Liu-Dayanikli-2025}.

On the other hand, Stackelberg games, also known as leader-follower games, were first proposed by von Stackelberg \cite{Stackelberg-1934} to describe decision-making problems with a hierarchical order: the leader acts first, the followers make optimal responses after observing the leader's strategy, and the leader optimizes its own objective in anticipation of the followers' responses. Yong \cite{Yong-2002} studied general stochastic LQ Stackelberg games, and various Stackelberg MFG models were subsequently developed
\cite{Moon-Basar-2018,Wang-2024,Bensoussan-Chau-Yam-2015,Bensoussan-Chau-Lai-Yam-2017,Yang-Huang-2021,Wang-Zhang-2020,Si-Wu-2021,Si-Shi-2025}. These works reveal the two-level equilibrium mechanism generated by combining a leader-follower hierarchical structure with mean-field interactions: for a given leader control, the follower population forms a Nash equilibrium; the leader then substitutes this Nash response mapping back into its own problem, forming a Stackelberg optimal control. However, existing Stackelberg mean field models are still usually based on global mean interactions, and it is difficult for them to characterize the network heterogeneity induced by graphon weighted adjacency structures.

This paper studies a general leader-follower LQ stochastic graphon game, extending our earlier work \cite{Chen-Shi2026}, which first introduced leader-follower stochastic graphon games. The model consists of one leader and a continuum of followers labeled by $I=[0,1]$, where the followers interact through graphon aggregation and are coupled with the leader through both state equations and cost functionals. Compared with \cite{Chen-Shi2026}, in which the leader's influence on the followers was mainly restricted to the cost functionals, the control weight matrices were positive definite, and no propagation of chaos result was obtained, the present paper considers a more general setting with direct leader-follower state coupling, indefinite control weight matrices, and finite-player approximation. In particular, the diffusion terms are allowed to depend on the states, controls, graphon aggregate terms, and the weighted average of the follower states in a general form.

The contributions of this paper include the following aspects.

(1) By introducing the {\it conditional exact law of large numbers} (CELLN) under a rich Fubini extension, this paper provides a rigorous mathematical formulation for the leader-follower stochastic graphon game with direct state-level coupling. It proves the well-posedness of the controlled system equations under admissible controls and constructs a Stackelberg-Nash equilibrium for the limiting problem.

(2) This paper establishes the connection between the limiting graphon game and finite-player network games. Specifically, when the number of followers is sufficiently large, finite-player strategies can be induced by the limiting Stackelberg-Nash equilibrium, and these strategies are proved to satisfy the corresponding approximate optimality in the finite network game, thereby yielding a propagation of chaos result in the sense of this paper.

(3) The LQ structure considered in this paper is fairly general. In particular, the diffusion coefficients allow a general dependence on the relevant state, control, graphon aggregate, and weighted-average terms, while the cost functionals of both the leader and the followers allow indefinite control weight matrices.

(4) The results of this paper cover multiple existing or naturally degenerate models. When the leader-follower structure degenerates, the model contains indefinite LQ stochastic graphon games; when the graphon is constant and the weights degenerate into a uniform average, the model degenerates into a Stackelberg mean field game; when the leader control enters the follower system as an external random influence, the model can also characterize LQ stochastic graphon games with random effects.

(5) The consistency condition of the follower problem can be reduced to a class of FBSDEs with graphon aggregate terms. By using the eigenvalue decomposition of the graphon operator and the method of continuation, this paper gives sufficient conditions for the existence and uniqueness of a solution to this consistency system. Compared with related results on LQ-GMFGs by Gao et al. \cite{Gao-Tchuendom-Caines-2021} and Xu et al. \cite{Xu-Gou-Huang-Gao-2025}, this paper does not rely on the finite-rank graphon operator assumption and allows a more general structure in the diffusion terms.

The remainder of this paper is organized as follows. Section \ref{Sec:02} introduces preliminaries including graphons, the rich Fubini extension, and the CELLN. Section \ref{Sec:03} gives the mathematical formulation of the problem and proves the existence, uniqueness, and estimates of solutions to the system equations. Section \ref{Sec:04} solves the followers' problem and the leader's problem, respectively, and constructs a Stackelberg-Nash equilibrium for the limiting system. Section \ref{Sec:05} discusses the finite-player game and proves the corresponding propagation of chaos result. Section \ref{Sec:06} gives several special cases. Section \ref{Sec:07} provides a numerical simulation example for {\it unmanned aerial vehicle} (UAV) formation navigation. Section \ref{Sec:08} concludes the paper. 

\section{Preliminaries}\label{Sec:02}

\subsection{Graphon theory}

For the theory of graphons, we refer to the monograph \cite{Lovasz-2012}. In what follows, let $I=[0,1]$, and let $(I,\mathcal{I},\lambda)$ be a probability space. A graphon is a real-valued bounded symmetric measurable function on $I\times I$. The graphons considered below take values in $[0,1]$, and the collection of all such graphons is denoted by $\mathcal{W}_0$. Let $L_{\mathcal{I}}^1(\mathbb{R}^n)$ denote the set of all $\mathbb{R}^n$-valued measurable functions $X$ on $I$ satisfying $\int_I |X^u|\lambda(\mathrm{d}u)<\infty$, let $L_{\mathcal{I}}^2(\mathbb{R}^n)$ denote the set of all $\mathbb{R}^n$-valued measurable functions $X$ on $I$ satisfying $\int_I |X^u|^2\lambda(\mathrm{d}u)<\infty$, and let $L_\mathcal{I}^\infty(\mathbb{R}^n)$ denote the set of all bounded $\mathbb{R}^n$-valued functions on $I$.

\begin{dfn}
	For any $G\in\mathcal{W}_0$, define
	
	(1) the $L^\infty$ norm and the degree of the graphon at point $u$:
	\begin{equation*}
		\|G\|_\infty=\sup_{u,v\in I}G(u,v),\quad \|G(u,\cdot)\|_1=\int_I G(u,v)\lambda(\mathrm{d}v),
	\end{equation*}
	
	(2) the bounded linear operator generated by $G$ from $L^1_{\mathcal{I}}(\mathbb{R}^n)$ to $L_\mathcal{I}^\infty(\mathbb{R}^n)$:
	\begin{equation*}
		GX^u=\int_I G(u,v)X^v\lambda(\mathrm{d}v),\quad u\in I,\quad X\in L^1_{\mathcal{I}}(\mathbb{R}^n).
	\end{equation*}
\end{dfn}

\begin{remark}
	A graphon can be understood as a weighted operator that measures the mutual influence among continuum players.
\end{remark}

The bounded linear operator generated by a graphon, referred to as the graphon operator, has the following properties.

\begin{prp}\label{Prp:2.1}
	For any $G\in\mathcal{W}_0$, the graphon operator $G$ defined above has the following properties:
	
	(1) for any $X\in L_{\mathcal{I}}^1(\mathbb{R}^n)$, we have $GX\in L_{\mathcal{I}}^\infty(\mathbb{R}^n)$, and
	\begin{equation*}
		\sup_{u\in I}|GX^u|\le \|G\|_\infty\int_I |X^u|\lambda(\mathrm{d}u)\le \int_I |X^u|\lambda(\mathrm{d}u),
	\end{equation*}
	
	(2) for any $X\in L_{\mathcal{I}}^\infty(\mathbb{R}^n)$ and any $u\in I$, we have
	\begin{equation*}
		|GX^u|\le \|G(u,\cdot)\|_1\cdot\sup_{u\in I}|X^u|\le \sup_{u\in I}|X^u|,
	\end{equation*}
	
	(3) if the domain of the graphon operator $G$ is restricted to $L_{\mathcal{I}}^2(\mathbb{R}^n)$, then $G$ is a compact self-adjoint operator on $L_{\mathcal{I}}^2(\mathbb{R}^n)$. There exists a sequence of eigenvectors $\{\phi_i\}_{i=1}^{\infty}$ of $G$ that forms an orthonormal basis of $\overline{\mathrm{ran}\;G}$. Let $\{\lambda_i\}_{i=1}^\infty$ be the corresponding eigenvalues. Then the eigenvalue decomposition is given by
	\begin{equation*}
		GX=\sum_{k=1}^\infty \lambda_k\langle X,\phi_k \rangle_I\phi_k,\quad \forall X\in L^2_{\mathcal{I}}(\mathbb{R}^n),
	\end{equation*}
	where
	\begin{equation*}
		\langle X,\phi_k \rangle_I=\int_I X^{u\top}\phi_k^u\lambda(\mathrm{d}u).
	\end{equation*}
\end{prp}

\begin{proof}
	The estimates follow from the definition of the graphon operator. The spectral decomposition theorem for graphon operators can be found in \cite{Lovasz-2012,Aurell-Carmona-Dayanikli-Lauriere-2022}, and the spectral decomposition properties of compact self-adjoint operators can be found in \cite{Conway-1990}.
\end{proof}

\subsection{Rich Fubini extension and CELLN}

For studies on the rich Fubini extension, the {\it exact law of large numbers} (ELLN), and the CELLN, see Sun \cite{Sun-2006} and Qiao et al. \cite{Qiao-Sun-Zhang-2016}. These works provide effective mathematical tools for studying games with a continuum of players. We list below the mathematical foundations needed in this paper. The proofs of the propositions are omitted, and the reader may refer to \cite{Aurell-Carmona-Dayanikli-Lauriere-2022,Sun-Zhang-2009,Podczeck-2010,Carmona-Delarue-2018}.

\begin{dfn}
	Let $(\varOmega,\mathcal{F},P)$ be a probability space, and let $\mathcal{C}$ be a countably generated sub-$\sigma$-algebra of $\mathcal{F}$. The $\mathbb{R}^n$-valued random variables $X,Y$ are said to be conditionally independent given $\mathcal{C}$ if, for any Borel subsets $B_1,B_2$ of $\mathbb{R}^n$, the conditional probabilities satisfy
	\begin{equation*}
		P\big( X^{-1}(B_1)\cap Y^{-1}(B_2)\big|\mathcal{C} \big)=P\big(X^{-1}(B_1)\big|\mathcal{C}\big)P\big( Y^{-1}(B_2) \big|\mathcal{C} \big).
	\end{equation*}
\end{dfn}

\begin{dfn}
	Let $(\varOmega,\mathcal{F},P),(I,\mathcal{I},\lambda)$ be two probability spaces, and let $\mathcal{C}$ be a countably generated sub-$\sigma$-algebra of $\mathcal{F}$. A mapping $X:\varOmega\times I\to\mathbb{R}^n$ is said to be essentially pairwise independent if, for $\lambda\text{-}a.e.\;u\in I$ and $\lambda\text{-}a.e.\;v\in I$, the random variables $X^u,X^v$ are independent. A mapping $X:\varOmega\times I\to\mathbb{R}^n$ is said to be essentially pairwise conditionally independent given $\mathcal{C}$ if, for $\lambda\text{-}a.e.\;u\in I$ and $\lambda\text{-}a.e.\;v\in I$, the random variables $X^u,X^v$ are conditionally independent given $\mathcal{C}$.
\end{dfn}

\begin{dfn}
	Let $(\varOmega,\mathcal{F},P),(I,\mathcal{I},\lambda)$ be two probability spaces, and let $(\varOmega\times I,\mathcal{F}\otimes\mathcal{I},P\otimes\lambda)$ be the usual product space. Suppose that the probability space $(\varOmega\times I,\mathcal{W},Q)$ is an extension of $(\varOmega\times I,\mathcal{F}\otimes\mathcal{I},P\otimes\lambda)$. If the classical Fubini theorem holds for any real-valued integrable function $f(\omega,u)$ on $(\varOmega\times I,\mathcal{W},Q)$, then $(\varOmega\times I,\mathcal{W},Q)$ is called a Fubini extension of $(\varOmega\times I,\mathcal{F}\otimes\mathcal{I},P\otimes\lambda)$.
\end{dfn}

\begin{prp}\label{prp03}
	Let $I=[0,1]$, let $\mathcal{B}_I$ be the Borel $\sigma$-algebra on $[0,1]$, let $\lambda_I$ be the Lebesgue measure on $[0,1]$, and let $E$ be a Polish space. Then there exist an extension $(I,\mathcal{I},\lambda)$ of $(I,\mathcal{B}_I,\lambda_I)$, a probability space $(\varOmega,\mathcal{F},P)$, and a Fubini extension $(\varOmega\times I,\mathcal{F}\boxtimes\mathcal{I},P\boxtimes\lambda)$ of $(\varOmega\times I,\mathcal{F}\otimes\mathcal{I},P\otimes\lambda)$ such that, for any measurable mapping $\varphi:(I,\mathcal{I},\lambda)\to\mathcal{P}(E)$, there exists an $\mathcal{F}\boxtimes\mathcal{I}$-measurable process $f:\varOmega\times I\to E$ such that the family of random variables $\{f^u\}_{u\in I}$ is essentially pairwise independent, and for any $u\in I$, $\mathcal{L}(f^u)=\varphi(u)$.
\end{prp}

The following proposition gives the existence of a family of essentially pairwise independent Brownian motions.

\begin{prp}
	Let $(\varOmega\times I,\mathcal{F}\boxtimes\mathcal{I},P\boxtimes\lambda)$ be the rich Fubini extension in Proposition \ref{prp03}. Then there exists $W=(W^u)_{u\in I}$ such that $W$ is essentially pairwise independent, and for any $u\in I$, $W^u$ is a $1$-dimensional standard Brownian motion on $(\varOmega,\mathcal{F},P)$.
\end{prp}

The following result is called the CELLN; see \cite{Qiao-Sun-Zhang-2016}.

\begin{thm}\label{Thm:2.1}
	Let $X$ be a real-valued integrable function on $(\varOmega\times I,\mathcal{F}\boxtimes\mathcal{I},P\boxtimes\lambda)$, and let $\mathcal{C}$ be a countably generated sub-$\sigma$-algebra of $\mathcal{F}$. If $X$ is essentially pairwise conditionally independent given $\mathcal{C}$, then
	\begin{equation*}
		\int_I X^u\lambda(\mathrm{d}u)=\int_I \mathbb{E}[X^u|\mathcal{C}]\lambda(\mathrm{d}u),\quad P\text{-}a.s..
	\end{equation*}
\end{thm}

\begin{proof}
	The proof of the theorem can be found in Corollary 2 and Lemma 3 of \cite{Qiao-Sun-Zhang-2016}.
\end{proof}

\section{Problem formulation}\label{Sec:03}

\subsection{Notation}

We use $A^\top$ to denote the transpose of a matrix $A$, and use $\mathcal{S}^n$ to denote the set of all $n$-dimensional real symmetric matrices. For an $\mathcal{S}^n$-valued matrix $Q$ and a constant $K\ge 0$, if $x^\top Qx\ge K|x|^2$ for any $x\in\mathbb{R}^n$, we write $Q\ge K$. For an $\mathbb{S}^n$-valued process $Q=(Q_t)_{t\in[0,T]}$ and a constant $K\ge 0$, if $x^\top Q_tx\ge K|x|^2$ for any $t\in[0,T]$ and $x\in\mathbb{R}^n$, we write $Q\gg K$. For $Q\in\mathcal{S}^n$ and an $n$-dimensional vector $x$, define $\|x\|_{Q}=x^\top Qx$.

Let $(\varOmega\times I,\mathcal{F}\boxtimes\mathcal{I},P\boxtimes\lambda)$ be the rich Fubini extension constructed in Proposition \ref{prp03}, let $(\varOmega,\mathcal{F},P)$ be the sample space, and let $(I,\mathcal{I},\lambda)$ be the follower index space. We use $\mathbb{E},\mathbb{E}^\boxtimes$ to denote the expectations on $(\varOmega,\mathcal{F},P)$ and $(\varOmega\times I,\mathcal{F}\boxtimes\mathcal{I},P\boxtimes\lambda)$, respectively. Let $W^f=(W^{f,u})_{u\in I}$ be essentially pairwise independent Brownian motions, and let $W^l$ be a Brownian motion independent of $W^{f}$.

Define the filtrations $\mathcal{F}^{f,u},\;u\in I$, and $\mathcal{F}^l$ as the completed natural filtrations generated by the Brownian motions $W^{f,u},W^l$, namely,
\begin{equation*}
	\mathcal{F}_t^{f,u}=\sigma(W_s^{f,u},0\le s\le t)\lor\mathcal{N},\quad u\in I,\quad \mathcal{F}_t^l=\sigma(W_s^l,0\le s\le t)\lor\mathcal{N},
\end{equation*}
where $\mathcal{N}$ is the collection of all $P$-null sets.

Define $L_{\boxtimes^l}^2(0,T;\mathbb{R}^n)$ as the set of all $\mathbb{R}^n$-valued $\mathcal{F}\boxtimes\mathcal{I}$-measurable processes $X$ on $[0,T]$ such that, for any $u\in I$, $X^u$ is an $\mathcal{F}^{f,u}\lor\mathcal{F}^l$-adapted process and $\mathbb{E}^\boxtimes\big[ \int_0^T|X_t|^2\mathrm{d}t \big]<\infty$; $L_{\boxtimes^l}^{\infty,2}(0,T;\mathbb{R}^n)$ as the set of all processes $X$ in $L_{\boxtimes^l}^2(0,T;\mathbb{R}^n)$ satisfying $\sup\limits_{u\in I}\mathbb{E}\big[\int_0^T|X_t^u|^2\mathrm{d}t\big]$ $<\infty$; $L_{\boxtimes^l}^2(\varOmega\times I;C(0,T;\mathbb{R}^n))$ as the set of all processes $X$ in $L_{\boxtimes^l}^2(0,T;\mathbb{R}^n)$ satisfying $\mathbb{E}^\boxtimes\big[\sup\limits_{0\le t\le T}|X_t|^2 \big]<\infty$; $L_{\boxtimes^l}^{\infty,2}(\varOmega\times I;C(0,T;\mathbb{R}^n))$ as the set of all processes $X$ in $L_{\boxtimes^l}^{\infty,2}(0,T;\mathbb{R}^n)$ satisfying $\sup\limits_{u\in I}\mathbb{E}\big[\sup\limits_{0\le t\le T}|X_t^u|^2 \big]<\infty$; $\bar{L}_{\boxtimes^l}^{2}(0,T;\mathbb{R}^n)$ as the set of all processes $X$ in $L_{\boxtimes^l}^2(0,T;\mathbb{R}^n)$ such that, for any $u\in I$, $X^u$ is an $\mathcal{F}^l$-adapted process; $\bar{L}_{\boxtimes^l}^2(\varOmega\times I;C(0,T;\mathbb{R}^n))$ as the set of all processes $X$ in $L_{\boxtimes^l}^2(\varOmega\times I;C(0,T;\mathbb{R}^n))$ such that, for any $u\in I$, $X^u$ is an $\mathcal{F}^l$-adapted process; $L_{\mathcal{F}^l}^2(0,T;\mathbb{R}^n)$ as the set of all $\mathbb{R}^n$-valued $\mathcal{F}^l$-measurable processes $X$ on $[0,T]$ satisfying $\mathbb{E}\big[\int_0^T |X_t|^2\mathrm{d}t \big]<\infty$; $L_{\mathcal{F}^l}^2(\varOmega;C(0,T;\mathbb{R}^n))$ as the set of all processes $X$ in $L_{\mathcal{F}^l}^2(0,T;\mathbb{R}^n)$ satisfying $\mathbb{E}\big[\sup\limits_{0\le t\le T}|X_t|^2\big]<\infty$; $L_{\mathcal{G}}^2(\mathbb{R}^n)$ as the set of all $\mathcal{G}$-measurable $\mathbb{R}^n$-valued random variables; $L_\boxtimes^2(\mathbb{R}^n)$ as the set of all $\mathcal{F}\boxtimes\mathcal{I}$-measurable $\mathbb{R}^n$-valued square-integrable random variables; $L_{\boxtimes^l_T}^2(\mathbb{R}^n)$ as the set of all processes $X$ in $L_\boxtimes^2(\mathbb{R}^n)$ such that, for any $u\in I$, $X^u\in L_{\mathcal{F}_T\lor\mathcal{F}_T^l}^2(\mathbb{R}^n)$; and $L_{\mathcal{F}\boxtimes\mathcal{I}}^{\infty,2}(\mathbb{R}^n)$ as the set of all random variables on $\mathcal{F}\boxtimes\mathcal{I}$ satisfying $\sup\limits_{u\in I}\mathbb{E}[|X^u|^2]<\infty$.

In the estimates for the equations, for notational brevity, write $\mathbb{H}_{\boxtimes^l}^2(\mathbb{R}^n)=L_{\boxtimes^l}^2(0,T;\mathbb{R}^n)$, $\mathbb{H}_{\boxtimes^l}^{\infty,2}(\mathbb{R}^n)=L_{\boxtimes^l}^{\infty,2}(0,T;\mathbb{R}^n)$, $\mathbb{S}_{\boxtimes^l}^2(\mathbb{R}^n)=L_{\boxtimes^l}^{2}(\varOmega\times I;C(0,T;\mathbb{R}^n))$, $\mathbb{S}_{\boxtimes^l}^{\infty,2}(\mathbb{R}^n)=L_{\boxtimes^l}^{\infty,2}(\varOmega\times I;C(0,T;\mathbb{R}^n))$, $\mathbb{H}_{\mathcal{F}^l}^2(\mathbb{R}^n)=L_{\mathcal{F}^l}^2(0,T;\mathbb{R}^n)$, $\mathbb{S}_{\mathcal{F}^l}^2(\mathbb{R}^n)=L_{\mathcal{F}^l}^2(\varOmega;C(0,T;\mathbb{R}^n))$.

\subsection{Problem formulation}

In this paper, we use the superscript $f$ to denote the states and controls of the followers, and use the superscript $l$ to denote the state and control of the leader. Assume that the state of each follower is $n_1$-dimensional, the state of the leader is $n_2$-dimensional, the control of each follower is $m_1$-dimensional, the control of the leader is $m_2$-dimensional, and the Brownian motions are $1$-dimensional. The system equations of the followers and the leader are

\begin{equation}\label{state}
	\begin{cases}
		\mathrm{d}X_t^{f,u}=b_t^{f,u}(X_t^{f,u},\alpha_t^{f,u},GX_t^{f,u},X_t^l,\alpha_t^l)\mathrm{d}t\\
		\qquad\qquad+\sigma_t^{f,u}(X_t^{f,u},\alpha_t^{f,u},GX_t^{f,u},X_t^l,\alpha_t^l)\mathrm{d}W_t^{f,u},\\
		\mathrm{d}X_t^l=b_t^l(X_t^l,\alpha_t^l,M_t^f)\mathrm{d}t+\sigma_t^l(X_t^l,\alpha_t^l,M_t^f)\mathrm{d}W_t^l,\\
		X_0^{f,u}=x_0^{f,u},\quad u\in I,\quad X_0^l=x_0^l,
	\end{cases}
\end{equation}
where
\begin{equation*}
	GX_t^{f,u}=\int_I G(u,v)X_t^{f,v}\lambda(\mathrm{d}v),\quad M_t^f=\int_I w^u X_t^{f,u}\lambda(\mathrm{d}u),
\end{equation*}
$w\in L_\mathcal{I}^\infty(\mathbb{R})$, for any $u\in I$, $w^u\ge 0$, and $\int_I w^u\lambda(\mathrm{d}u)=1$.

The admissible control sets of the followers and the leader are
\begin{equation*}
	\mathcal{U}^{f}=L^2_{\boxtimes^l}(0,T;\mathbb{R}^{m_1}),\quad \mathcal{U}^l=L_{\mathcal{F}^l}^2(0,T;\mathbb{R}^{m_2}),
\end{equation*}
and, for the followers' problem, we also consider a stronger class of admissible controls
\begin{equation*}
	\mathcal{U}_s^f=L_{\boxtimes^l}^{\infty,2}(0,T;\mathbb{R}^{m_1}).
\end{equation*}
Under admissible controls $\alpha^f\in\mathcal{U}^f$, $\alpha^l\in\mathcal{U}^l$, the cost functional of follower $u$ is
\begin{multline}\label{cost followers}
	J^{f,u}(\alpha^{f,u};\alpha^{f,-u},\alpha^l)\\
	=\mathbb{E}\biggl[ \int_0^T h_t^{f,u}(X_t^{f,u},\alpha_t^{f,u},GX_t^{f,u},X_t^l,\alpha_t^l)\mathrm{d}t+g^{f,u}(X_T^{f,u},GX_T^{f,u},X_T^l) \biggr],
\end{multline}
and the cost functional of the leader is
\begin{equation}\label{cost leader}
	J^l(\alpha^l;\alpha^f)=\mathbb{E}\biggl[ \int_0^T h_t^l(X_t^l,\alpha_t^l,M_t^f)\mathrm{d}t+g^l(X_T^l,M_T^f) \biggr].
\end{equation}

Each follower and the leader seek to optimize their own cost functionals. We consider a hierarchical decision structure: the leader first gives the admissible control set $\mathcal{U}^l$; for each control $\alpha^l\in\mathcal{U}^l$ of the leader, the followers compete and reach a Nash equilibrium $\hat{\alpha}^f(\alpha^l)\in\mathcal{U}^f$; the leader can anticipate the Nash equilibrium that the followers will reach, chooses $\hat{\alpha}^l\in\mathcal{U}^l$, and optimizes the cost functional $J^l(\alpha^l,\hat{\alpha}^f(\alpha^l))$.

\begin{Ass}\label{Ass:05}
	The coefficients
	\begin{gather*}
		b^f(x^f,\alpha^f,Gx^f,x^l,\alpha^l):[0,T]\times I\times\mathbb{R}^{n_1}\times\mathbb{R}^{m_1}\times\mathbb{R}^{n_1}\times\mathbb{R}^{n_2}\times\mathbb{R}^{m_2}\to\mathbb{R}^{n_1},\\
		\sigma^f(x^f,\alpha^f,Gx^f,x^l,\alpha^l):[0,T]\times I\times\mathbb{R}^{n_1}\times\mathbb{R}^{m_1}\times\mathbb{R}^{n_1}\times\mathbb{R}^{n_2}\times\mathbb{R}^{m_2}\to\mathbb{R}^{n_1},\\
		b^l(x^l,\alpha^l,m^f):[0,T]\times\mathbb{R}^{n_2}\times\mathbb{R}^{m_2}\times\mathbb{R}^{n_1}\to\mathbb{R}^{n_2},\\
		\sigma^l(x^l,\alpha^l,m^f):[0,T]\times\mathbb{R}^{n_2}\times\mathbb{R}^{m_2}\times\mathbb{R}^{n_1}\to\mathbb{R}^{n_2},\\
		h^f(x^f,\alpha^f,Gx^f,x^l,\alpha^l):[0,T]\times I\times\mathbb{R}^{n_1}\times\mathbb{R}^{m_1}\times\mathbb{R}^{n_1}\times\mathbb{R}^{n_2}\times\mathbb{R}^{m_2}\to\mathbb{R},\\
		g^f(x^f,Gx^f,x^l):I\times\mathbb{R}^{n_1}\times\mathbb{R}^{n_1}\times\mathbb{R}^{n_2}\to\mathbb{R},\\
		h^l(x^l,\alpha^l,m^f):[0,T]\times\mathbb{R}^{n_2}\times\mathbb{R}^{m_2}\times\mathbb{R}^{n_1}\to\mathbb{R},\\
		g^l(x^l,m^f):\mathbb{R}^{n_2}\times\mathbb{R}^{n_1}\to\mathbb{R},
	\end{gather*}
	are jointly measurable.
\end{Ass}

\begin{Ass}[The LQ case]\label{Ass:LQ}
	Assume that there exist deterministic coefficient matrices satisfying
	\begin{gather*}
		b_t^{f,u}(x^f,\alpha^f,Gx^f,x^l,\alpha^l)=b_t^{f,u,0}+b_t^{f,u,1}x^f+b_t^{f,u,2}\alpha^f+b_t^{f,u,3}Gx^f+b_t^{f,u,4}x^l+b_t^{f,u,5}\alpha^l,\\
		\sigma_t^{f,u}(x^f,\alpha^f,Gx^f,x^l,\alpha^l)=\sigma_t^{f,u,0}+\sigma_t^{f,u,1}x^f+\sigma_t^{f,u,2}\alpha^f+\sigma_t^{f,u,3}Gx^f+\sigma_t^{f,u,4}x^l+\sigma_t^{f,u,5}\alpha^l,\\
		b_t^l(x^l,\alpha^l,m^f)=b_t^{l,0}+b_t^{l,1}x^l+b_t^{l,2}\alpha^l+b_t^{l,3}m^f,\\
		\sigma_t^l(x^l,\alpha^l,m^f)=\sigma_t^{l,0}+\sigma_t^{l,1}x^l+\sigma_t^{l,2}\alpha^l+\sigma_t^{l,3}m^f,\\
		h_t^{f,u}(x^f,\alpha^f,Gx^f,x^l,\alpha^l)=1/2\|x^f-h_t^{f,u,1}Gx^f-h_t^{f,u,2}x^l\|_{Q^{f,u}_t}+1/2\|\alpha^f-h_t^{f,u,3}\alpha^l\|_{R^{f,u}_t},\\
		g^{f,u}(x^f,Gx^f,x^l)=1/2\|x^f-g^{f,u,1}Gx^f-g^{f,u,2}x^l\|_{G^{f,u}},\\
		h_t^l(x^l,\alpha^l,m^f)=1/2\|x^l-h_t^{l,1}m^f\|_{Q_t^l}+1/2\|\alpha^l\|_{R_t^{l}},\\
		g^l(x^l,m^f)=1/2\|x^l-g^{l,1}m^f\|_{G^l},
	\end{gather*}
	where all coefficients are $\mathcal{I}$-measurable with respect to $u$ and uniformly bounded with respect to $(t,u)$.
\end{Ass}

We will solve the Stackelberg-Nash equilibrium of the leader and the followers in Section \ref{Sec:04}. The remainder of this section explains the well-posedness of the problem formulation.

\subsection{Existence and uniqueness of system solutions}

This subsection proves that, for any admissible controls, the system equations of the followers and the leader admit a unique solution, and that the graphon aggregate and weighted mean-field term of the follower states depend only on the Brownian motion driving the leader's equation.

The following result extends the domain of the graphon operator to $L_{\boxtimes}^2(\mathbb{R}^n)$.

\begin{prp}\label{Prp:3.1}
	Let $G\in\mathcal{W}_0$, and define the operator generated by the graphon as
	\begin{equation*}
		GX^u=\int_I G(u,v)X^v\lambda(\mathrm{d}v),\quad X\in L_{\boxtimes}^2(\mathbb{R}^n),
	\end{equation*}
	then $G$ is a bounded linear operator from $L_\boxtimes^2(\mathbb{R}^n)$ to $L_\boxtimes^{\infty,2}(\mathbb{R}^n)$, and the estimate
	\begin{equation}\label{eq03}
		\sup_{u\in I}\mathbb{E}\biggl[\int_0^T |GX_t^u|^2\mathrm{d}t\biggr]\le \|G\|_\infty^2\cdot\mathbb{E}^\boxtimes\biggl[ \int_0^T |X_t|^2\mathrm{d}t \biggr]<\infty
	\end{equation}
	holds.
\end{prp}

\begin{proof}
	It has been proved in \cite{Aurell-Carmona-Lauriere-2022} that $G$ is a bounded linear operator from $L_\boxtimes^2(\mathbb{R}^n)$ to $L_\boxtimes^2(\mathbb{R}^n)$. For any $u\in I$, by the Cauchy-Schwarz inequality and by exchanging the order of integration,
	\begin{equation*}
		\int_0^T |GX_t^u|^2\mathrm{d}t\le \int_0^T\|G\|_\infty^2\int_I |X_t^u|^2\lambda(\mathrm{d}u)\mathrm{d}t=\|G\|_\infty^2\int_I\int_0^T |X_t^u|^2\mathrm{d}t\lambda(\mathrm{d}u),
	\end{equation*}
	and taking expectations and then the supremum on both sides yields estimate (\ref{eq03}).
\end{proof}

\begin{Ass}\label{Ass:01}
	$b^f,\sigma^f$ are uniformly Lipschitz continuous with respect to $(x^f,\alpha^f,Gx^f$, $x^l,\alpha^l)$, and $b^l,\sigma^l$ are uniformly Lipschitz continuous with respect to $(x^l,\alpha^l,m^f)$. Moreover, $b_t^{f,u}(0,0,0,0,0),\sigma_t^{f,u}(0,0,0,0,0)$ are uniformly bounded with respect to $(t,u)$, $x_0^f$ is uniformly bounded with respect to $u$, and $b_t^l(0,0,0),\sigma_t^l(0,0,0)$ are uniformly bounded with respect to $t$.
\end{Ass}

For convenience in the subsequent estimates of the equation solutions, we write
\begin{gather*}
	I_0^{f,u}:=\mathbb{E}\biggl[ |x_0^{f,u}|^2+\int_0^T\big\{ |b_t^{f,u}(0,0,0,0,0)|^2+|\sigma_t^{f,u}(0,0,0,0,0)|^2 \big\}\mathrm{d}t \biggr],\quad u\in I,\\
	I_0^{l}:=\mathbb{E}\biggl[ |x_0^l|^2+\int_0^T\big\{ |b_t^l(0,0,0)|^2+|\sigma_t^l(0,0,0)|^2 \big\}\mathrm{d}t \biggr].
\end{gather*}

\begin{prp}\label{prp07}
	For any $(z^1,z^2)\in\mathbb{H}^2_{\boxtimes^l}(\mathbb{R}^{n_1})\times\mathbb{H}^2_{\mathcal{F}^l}(\mathbb{R}^{n_1})$ and $(\alpha^f,\alpha^l)\in\mathcal{U}^f\times\mathcal{U}^l$, the equation
	\begin{equation}\label{eq01}
		\begin{cases}
			\mathrm{d}X_t^{f,u}=b_t^{f,u}(X_t^{f,u},\alpha_t^{f,u},z_t^{1,u},X_t^l,\alpha_t^l)\mathrm{d}t\\
			\qquad\qquad+\sigma_t^{f,u}(X_t^{f,u},\alpha_t^{f,u},z_t^{1,u},X_t^l,\alpha_t^l)\mathrm{d}W_t^{f,u},\\
			\mathrm{d}X_t^l=b_t^l(X_t^l,\alpha_t^l,z_t^2)\mathrm{d}t+\sigma_t^l(X_t^l,\alpha_t^l,z_t^2)\mathrm{d}W_t^l,\\
			X_0^{f,u}=x_0^{f,u},\quad u\in I,\quad X_0^l=x_0^l,
		\end{cases}
	\end{equation}
	admits a unique solution $(X^f,X^l)\in\mathbb{S}^2_{\boxtimes^l}(\mathbb{R}^{n_1})\times\mathbb{S}_{\mathcal{F}^l}^2(\mathbb{R}^{n_2})$, and the estimate
	\begin{multline}\label{eq18}
		\mathbb{E}^\boxtimes\biggl[ \sup_{0\le t\le T}|X_t^f|^2 \biggr]+\mathbb{E}\biggl[ \sup_{0\le t\le T}|X_t^l|^2 \biggr]\le K\int_I I_0^{f,u}\lambda(\mathrm{d}u)+KI_0^l\\
		+K\biggl\{\mathbb{E}^\boxtimes\biggl[ \int_0^T \big\{|z_t^1|^2+|\alpha_t^f|^2\big\}\mathrm{d}t \biggr]+\mathbb{E}\biggl[ \int_0^T\big\{|z_t^2|^2+|\alpha_t^l|^2\big\}\mathrm{d}t \biggr] \biggr\}
	\end{multline}
	holds.
	
	Furthermore, for any $(z^1,z^2)\in\mathbb{H}_{\boxtimes^l}^{\infty,2}(\mathbb{R}^{n_1})\times\mathbb{H}_{\mathcal{F}^l}^2(\mathbb{R}^{n_1})$ and $(\alpha^f,\alpha^l)\in\mathcal{U}_s^f\times\mathcal{U}^l$, equation (\ref{eq01}) admits a unique solution $(X^f,X^l)\in\mathbb{S}_{\boxtimes^l}^{\infty,2}(\mathbb{R}^{n_1})\times\mathbb{S}_{\mathcal{F}^l}^2(\mathbb{R}^{n_2})$, and the estimate
	\begin{multline}\label{eq19}
		\sup_{u\in I}\mathbb{E}\biggl[ \sup_{0\le t\le T}|X_t^{f,u}|^2 \biggr]+\mathbb{E}\biggl[ \sup_{0\le t\le T}|X_t^l|^2 \biggr]\le K\sup_{u\in I}I_0^{f,u}+KI_0^l\\
		+K\biggl\{ \sup_{u\in I}\mathbb{E}\biggl[ \int_0^T\big\{ |z_t^{1,u}|^2+|\alpha_t^{f,u}|^2 \big\}\mathrm{d}t \biggr]+\mathbb{E}\biggl[ \int_0^T\big\{ |z_t^2|^2+|\alpha_t^l|^2 \big\}\mathrm{d}t \biggr] \biggr\}
	\end{multline}
	holds.
\end{prp}

\begin{proof}
	By the classical SDE theory, for any $u\in I$, the equation admits a unique solution $(X^{f,u},X^l)\in\mathbb{S}_{\mathcal{F}^{f,u}\lor\mathcal{F}^l}^2(\mathbb{R}^{n_1})\times\mathbb{S}_{\mathcal{F}^l}^2(\mathbb{R}^{n_2})$, $X^f$ is $\mathcal{I}$-measurable with respect to $u$, and the estimate
	\begin{equation*}
		\mathbb{E}\biggl[ \sup_{0\le t\le T}\big\{|X_t^{f,u}|^2+|X_t^l|^2\big\} \biggr]\le K\biggl\{I_0^{f,u}+I_0^l+\mathbb{E}\biggl[ \int_0^T\big\{ |z_t^{1,u}|^2+|\alpha_t^{f,u}|^2+|z_t^2|^2+|\alpha_t^l|^2 \big\}\mathrm{d}t \biggr]\biggr\}
	\end{equation*}
	holds.
	
	If $(z^1,z^2)\in\mathbb{H}^2_{\boxtimes^l}(\mathbb{R}^{n_1})\times\mathbb{H}_{\mathcal{F}^l}^2(\mathbb{R}^{n_1})$ and $(\alpha^f,\alpha^l)\in\mathcal{U}^f\times\mathcal{U}^l$, integrating the above inequality with respect to $u$ gives $(X^f,X^l)\in\mathbb{S}^2_{\boxtimes^l}(\mathbb{R}^{n_1})\times\mathbb{S}_{\mathcal{F}^l}^2(\mathbb{R}^{n_2})$ and estimate (\ref{eq18}). If $(z^1,z^2)\in\mathbb{H}_{\boxtimes^l}^{\infty,2}(\mathbb{R}^{n_1})\times\mathbb{H}_{\mathcal{F}^l}^2(\mathbb{R}^{n_1})$ and $(\alpha^f,\alpha^l)\in\mathcal{U}_s^f\times\mathcal{U}^l$, taking the supremum of the above inequality with respect to $u$ gives $(X^f,X^l)\in\mathbb{H}_{\boxtimes^l}^{\infty,2}(\mathbb{R}^{n_1})\times\mathbb{H}_{\mathcal{F}^l}^2(\mathbb{R}^{n_2})$ and estimate (\ref{eq19}).
\end{proof}

For fixed $(\alpha^f,\alpha^l)\in\mathcal{U}^f\times\mathcal{U}^l$, define the mapping
\begin{equation*}
	\varphi(z^1,z^2)=(GX^f,M^f),
\end{equation*}
where $(X^f,X^l)$ is the solution of equation (\ref{eq01}) under $(z^1,z^2)$, and
\begin{gather*}
	GX^{f,u}_t=\int_I G(u,v)X_t^{f,v}\lambda(\mathrm{d}v),\quad M_t^f=\int_I w^u X_t^{f,u}\lambda(\mathrm{d}u).
\end{gather*}

\begin{prp}\label{Prp:3.3}
	For any $(\alpha^f,\alpha^l)\in\mathcal{U}^f\times\mathcal{U}^l$, $\varphi$ is a mapping from $\mathbb{H}_{\boxtimes^l}^2(\mathbb{R}^{n_1})\times\mathbb{H}_{\mathcal{F}^l}^2(\mathbb{R}^{n_1})$ to itself and admits a unique fixed point. If $(\alpha^f,\alpha^l)\in\mathcal{U}_s^f\times\mathcal{U}^l$, then $\varphi$ is a mapping from $\mathbb{H}_{\boxtimes^l}^{\infty,2}(\mathbb{R}^{n_1})\times\mathbb{H}_{\mathcal{F}^l}^2(\mathbb{R}^{n_2})$ to itself and admits a unique fixed point.
\end{prp}

\begin{proof}
	For any fixed $(\alpha^f,\alpha^l)\in\mathcal{U}^f\times\mathcal{U}^l$ and any $(z^1,z^2),({z^1}',{z^2}')\in\mathbb{H}^2_{\boxtimes^l}(\mathbb{R}^{n_1})\times\mathbb{H}_{\mathcal{F}^l}^2(\mathbb{R}^{n_1})$, let $(X^f,X^l),({X^{f}}',{X^{l}}')$ be the solutions of the corresponding equation (\ref{eq01}), and denote $(\varDelta X^f,\varDelta X^l)=(X^f-{X^f}',X^l-{X^l}')$. Then, by Proposition \ref{prp07}, for any $t\in[0,T]$,
	\begin{equation*}
		\mathbb{E}^\boxtimes\biggl[ \sup_{0\le s\le t}\big\{|\varDelta X_s^f|^2+|\varDelta X_s^l|^2\big\} \biggr]\le K\mathbb{E}^\boxtimes\biggl[\int_0^t \big\{|z_s^1-{z_s^1}'|^2+|z_s^2-{z_s^2}'|^2\big\}\mathrm{d}s \biggr],
	\end{equation*}
	and, by the properties of bounded linear operators, we obtain
	\begin{equation*}
    \begin{aligned}
		&\sup_{0\le s\le t}\mathbb{E}^\boxtimes[|\varphi(z^1,z^2)_s-\varphi({z^1}',{z^2}')_s|^2]
         \le \mathbb{E}^\boxtimes\biggl[ \sup_{0\le s\le t} |\varphi(z^1,z^2)_s-\varphi({z^1}',{z^2}')_s|^2 \biggr]\\
		&=\mathbb{E}^\boxtimes\biggl[ \sup_{0\le s\le t}\big\{ |G\varDelta X_s^f|^2+|\varDelta M^f_s|^2 \big\} \biggr]
         \le K\mathbb{E}^\boxtimes\biggl[\sup_{0\le s\le t}|\varDelta X_s^f|^2\biggr]\\
		&\le K\int_0^t \mathbb{E}^\boxtimes[|z_{s}^1-{z_{s}^1}'|^2+|z_s^2-{z_s^2}'|^2]\mathrm{d}s.
	\end{aligned}
    \end{equation*}
	Let $\varphi^{(k)}$ denote the $k$-th iteration of the mapping $\varphi$. By mathematical induction, for any $t\in [0,T]$ and any positive integer $k$, we have
	\begin{equation*}
		\sup_{0\le s\le t}\mathbb{E}^\boxtimes[|\varphi^{(k)}(z^1,z^2)_s-\varphi^{(k)}({z^1}',{z^2}')_s|^2]\le K^{k+1}\frac{t^k}{k!}\int_0^t \mathbb{E}^\boxtimes[|z_s^1-{z_s^1}'|^2+|z_s^2-{z_s^2}'|^2]\mathrm{d}s.
	\end{equation*}
	There exists a sufficiently large $k$ such that $\varphi^{(k)}$ is a contraction mapping. Hence the mapping $\varphi$ admits a unique fixed point.
	
	Similarly, when $(\alpha^f,\alpha^l)\in\mathcal{U}_s^f\times\mathcal{U}^l$, Proposition \ref{prp07} and induction imply that $\varphi$ admits a unique fixed point.
\end{proof}

We next give several lemmas on essential pairwise conditional independence.

\begin{lem}\label{Lem:3.1}
	Let $(\varOmega,\mathcal{F},P)$ be a probability space, and let $\mathcal{F}^1,\mathcal{F}^2,\mathcal{F}^0$ be sub-$\sigma$-algebras of $\mathcal{F}$. If $\mathcal{F}^1$ and $\mathcal{F}^2$ are independent, and $\mathcal{F}^0$ is independent of $\mathcal{F}^1\lor\mathcal{F}^2$, then $\mathcal{F}^1$ and $\mathcal{F}^2$ are conditionally independent given $\mathcal{F}^0$.
\end{lem}

\begin{proof}
	Since $\mathcal{F}^1$ and $\mathcal{F}^2$ are independent, and $\mathcal{F}^0$ is independent of $\mathcal{F}^1\lor\mathcal{F}^2$, for any $A\in\mathcal{F}^1$ and $B\in\mathcal{F}^2$, we have
	\begin{equation*}
    \begin{aligned}
		&P(A\cap B|\mathcal{F}^0)=\mathbb{E}[\mathbf{1}_{A\cap B}|\mathcal{F}^0]=P(A\cap B)=P(A)P(B)\\
		&=\mathbb{E}[\mathbf{1}_{A}|\mathcal{F}^0]\mathbb{E}[\mathbf{1}_B|\mathcal{F}^0]=P(A|\mathcal{F}^0)P(B|\mathcal{F}^0),\quad a.s.,
	\end{aligned}
    \end{equation*}
	and therefore $\mathcal{F}^1,\mathcal{F}^2$ are conditionally independent given $\mathcal{F}^0$.
\end{proof}

\begin{lem}\label{Lem:3.2}
	Let $(\varOmega,\mathcal{F},P)$ be a probability space, let $W^1,W^2$ be mutually independent Brownian motions, let $\mathcal{F}^1,\mathcal{F}^2$ be the natural filtrations generated by the two Brownian motions, and let $X$ be a progressively measurable process on $\mathcal{F}^1\lor\mathcal{F}^2$ satisfying $\mathbb{E}\big[ \int_0^t |X_s|^2\mathrm{d}s \big]<\infty$. Then, for any $t\in[0,T]$,
	\begin{equation*}
		\mathbb{E}\biggl[ \int_0^t X_s\mathrm{d}W_s^1\bigg|\mathcal{F}_t^2 \biggr]=0.
	\end{equation*}
\end{lem}

\begin{proof}
	For fixed $t\in [0,T]$ and any $s\in[0,t]$, set $\mathcal{G}_s=\mathcal{F}_s^1\lor\mathcal{F}_t^2$. Then, on $[0,t]$, $W^1$ is a $\mathcal{G}$-Brownian motion, and $X$ is a $\mathcal{G}$-progressively measurable process. Let $N_s=\int_0^s X_r\mathrm{d}W_r^1$. Then $N$ is a square-integrable $\mathcal{G}$-martingale, and hence
	\begin{equation*}
		\mathbb{E}[N_t|\mathcal{F}_t^2]=\mathbb{E}\big[\mathbb{E}[N_t|\mathcal{G}_0]\big|\mathcal{F}_t^2\big]=\mathbb{E}[N_0|\mathcal{F}_t^2]=0.\qedhere
	\end{equation*}
\end{proof}

\begin{prp}\label{Prp:3.4}
	Let $X\in L_{\boxtimes^l}^2(0,T;\mathbb{R}^n)$. Then $X$ is essentially pairwise conditionally independent given $\mathcal{F}_T^l$, and for any $t\in[0,T]$,
	\begin{equation*}
		\int_I\int_0^t X_s^u\mathrm{d}W_s^{f,u}\lambda(\mathrm{d}u)=0,\quad a.s..
	\end{equation*}
\end{prp}

\begin{proof}
	Since $W^f$ is a family of essentially pairwise independent Brownian motions, there exists $I_0\in\mathcal{I}$ with $P(I_0)=1$ such that, for any $u,v\in I_0$, $W^{f,u}$ and $W^{f,v}$ are independent, and $W^l$ is independent of $W^{f,u}\lor W^{f,v}$. By Lemma \ref{Lem:3.1}, $W^{f,u}$ and $W^{f,v}$ are conditionally independent given $\mathcal{F}_T^l$, and hence $X$ is essentially pairwise conditionally independent given $\mathcal{F}_T^l$. By Theorem \ref{Thm:2.1} and Lemma \ref{Lem:3.2}, we obtain
	\begin{equation*}
		\int_I\int_0^t X_s^u\mathrm{d}W_s^{f,u}\lambda(\mathrm{d}u)=\int_I\mathbb{E}\biggl[ \int_0^t X_s^u\mathrm{d}W_s^{f,u}\lambda(\mathrm{d}u) \bigg|\mathcal{F}_T^l \biggr]=0,\quad a.s..\qedhere
	\end{equation*}
\end{proof}

\begin{thm}\label{Thm:3.1}
	For any $(\alpha^f,\alpha^l)\in\mathcal{U}^f\times\mathcal{U}^l$, system equation (\ref{state}) admits a unique solution $(X^f,X^l)\in\mathbb{S}_{\boxtimes^l}^2(\mathbb{R}^{n_1})\times\mathbb{S}^2_{\mathcal{F}^l}(\mathbb{R}^{n_2})$, with $GX^f\in \bar{\mathbb{S}}_{\boxtimes^l}^{\infty,2}(\mathbb{R}^{n_1})$ and $M^f\in \mathbb{S}_{\mathcal{F}^l}^2(\mathbb{R}^{n_1})$. In particular, for any $(\alpha^f,\alpha^l)\in\mathcal{U}_s^f\times\mathcal{U}^l$, the system admits a unique solution $(X^f,X^l)\in\mathbb{S}_{\boxtimes^l}^{\infty,2}(\mathbb{R}^{n_1})\times\mathbb{S}^2_{\mathcal{F}^l}(\mathbb{R}^{n_2})$.
\end{thm}

\begin{proof}
	The existence and uniqueness of a solution to the system equation follow from Propositions \ref{prp07} and \ref{Prp:3.3}. For $X^f\in\mathbb{S}_{\boxtimes^l}^2(\mathbb{R}^{n_1})$, Proposition \ref{Prp:3.4} implies that $X^f$ is essentially pairwise conditionally independent given $\mathcal{F}_T^l$. By Propositions \ref{Prp:3.1} and \ref{Prp:2.1}, we obtain $GX^f\in \bar{\mathbb{S}}_{\boxtimes^l}^{\infty,2}(\mathbb{R}^{n_1})$ and $M^f\in \mathbb{S}_{\mathcal{F}^l}^2(\mathbb{R}^{n_1})$.
\end{proof}

\section{Solution of the limiting problem}\label{Sec:04}

\subsection{The followers' problem}

We use a fixed-point method to solve the Nash equilibrium of the followers' problem. The steps are as follows:

(1) for fixed $\alpha^l\in\mathcal{U}^l$ and fixed $z\in\bar{\mathbb{S}}_{\boxtimes^l}^{\infty,2}(\mathbb{R}^{n_1})$, solve the optimal control problem of follower $u\in I$;

(2) find $z\in\bar{\mathbb{S}}_{\boxtimes^l}^{\infty,2}(\mathbb{R}^{n_1})$ such that $z$ and the corresponding optimal pair $(\hat{X}^f,\hat{\alpha}^f)$ satisfy $z=G\hat{X}^f$.

\begin{Ass}\label{Ass:02}
	$b^f,\sigma^f,h^f$ are differentiable with respect to $(x^f,\alpha^f)$; $\partial_{x^f} b^f,\partial_{x^f}\sigma^f,\partial_{x^f}h^f$ are continuous with respect to $(x^f,\alpha^f)$; $g^f$ is continuously differentiable with respect to $x^f$; $b_t^{f,u}(0,0,0,0,0),\sigma_t^{f,u}(0,0,0,0,0),h_t^{f,u}(0,0,0,0,0),g_t^{f,u}(0,0,0)$ are uniformly bounded with respect to $(t,u)$; $\partial_{x^f}b^f,\partial_{\alpha^f}b^f,\partial_{x^f}\sigma^f,\partial_{\alpha^f}\sigma^f$ are uniformly bounded; $\partial_{x^f}h^f,\partial_{\alpha^f}h^f$ have at most linear growth with respect to $(x^f,\alpha^f)$; $\partial_{x^f}g^f$ has at most linear growth with respect to $x^f$; $b^f,\sigma^f,\partial_{x^f}h^f,\partial_{\alpha^f}h^f,\partial_{x^f}g^f$ have at most linear growth with respect to $(Gx^f,x^l,\alpha^l)$; and $h^f,g^f$ have at most quadratic growth with respect to $(Gx^f,x^l,\alpha^l)$.
\end{Ass}

\begin{prp}[Maximum principle]
	Assume that Assumptions \ref{Ass:05}, \ref{Ass:01}, and \ref{Ass:02} hold. For fixed $\alpha^l\in\mathcal{U}^l$ and $z\in\bar{\mathbb{S}}_{\boxtimes^l}^{\infty,2}(\mathbb{R}^{n_1})$, suppose that the problem of follower $u$ admits an optimal pair $(\hat{\alpha}^{f,u},\hat{X}^{f,u})$. Then the adjoint equation
	\begin{equation*}
		\begin{cases}
			\mathrm{d}p_t^{f,u}=-\big\{\partial_{x^f} b_t^{f,u}(X_t^{f,u},\alpha_t^{f,u},z_t^u,X_t^l,\alpha_t^l)^\top p_t^{f,u}\\
			\qquad\qquad +\partial_{x^f}\sigma_t^{f,u}(X_t^{f,u},\alpha_t^{f,u},z_t^u,X_t^l,\alpha_t^l)^\top q_t^{f,u}\\
			\qquad\qquad +\partial_{x^f} h_t^{f,u}(X_t^{f,u},\alpha_t^{f,u},z_t^u,X_t^l,\alpha_t^l)\big\}\mathrm{d}t+q_t^{f,u}\mathrm{d}W_t^{f,u}+q_t^{l,u}\mathrm{d}W_t^l,\\
			p_T^{f,u}=\partial_{x^f} g^{f,u}(X_T^{f,u},z_T^u,X_T^l),
		\end{cases}
	\end{equation*}
	admits a unique solution $(p^{f,u},q^{f,u},q^{l,u})\in\mathbb{S}_{\mathcal{F}^{f,u}\lor\mathcal{F}^l}^2(\mathbb{R}^{n_1})\times\mathbb{H}_{\mathcal{F}^{f,u}\lor\mathcal{F}^l}^2(\mathbb{R}^{n_1})\times\mathbb{H}_{\mathcal{F}^{f,u}\lor\mathcal{F}^l}^2(\mathbb{R}^{n_1})$, and the optimal control of follower $u$ satisfies
	\begin{equation*}
		\hat{\alpha}_t^{f,u}=\arg\min_\alpha H_t^{f,u}(\hat{X}_t^{f,u},\alpha,z_t^u,\hat{p}_t^{f,u},\hat{q}_t^{f,u}),
	\end{equation*}
	where the Hamiltonian is
	\begin{equation*}
    \begin{aligned}
		H^{f,u}_t(x,\alpha,z,p,q)&:=p^\top b_t^{f,u}(x,\alpha,z,X_t^l,\alpha_t^l)+q^\top \sigma_t^{f,u}(x,\alpha,z,X_t^l,\alpha_t^l)\\
		&\quad +h_t^{f,u}(x,\alpha,z,X_t^l,\alpha_t^l).
	\end{aligned}
    \end{equation*}
\end{prp}

\begin{proof}
	For fixed $\alpha^l\in\mathcal{U}^l$ and $z\in\bar{\mathbb{S}}_{\boxtimes^l}^{\infty,2}(\mathbb{R}^{n_1})$, for each $u\in I$, changing $\alpha^{f,u}$ alone does not affect the value of $X^l$. Therefore, the follower problem reduces to a family of independently optimized optimal control problems. Since the control domain is convex, the conclusion follows from the stochastic maximum principle; see Theorem 6.14 in Volume I of \cite{Carmona-Delarue-2018} or Theorem 3.2 in Chapter 3 of \cite{Yong-Zhou-1999}.
\end{proof}

\begin{prp}[Consistency condition]
	Assume that Assumptions \ref{Ass:05}, \ref{Ass:01}, and \ref{Ass:02} hold. For fixed $\alpha^l\in\mathcal{U}^l$ and $z\in\bar{\mathbb{S}}_{\boxtimes^l}^{\infty,2}(\mathbb{R}^{n_1})$, suppose that the corresponding follower problem admits an optimal pair $(\hat{\alpha}^f,\hat{X}^f)$. Then $z=G\hat{X}^f$ if and only if $z$ satisfies the equation
	\begin{equation*}
		\begin{cases}
			\mathrm{d}z_t^u=\int_I G(u,v)b_t^{f,v}(\hat{X}_t^{f,v},\hat{\alpha}_t^{f,v},z_t^v,X_t^l,\alpha_t^l)\lambda(\mathrm{d}v)\mathrm{d}t,\\
			z_0^u=\int_I G(u,v)x_0^{f,v}\lambda(\mathrm{d}v).
		\end{cases}
	\end{equation*}
\end{prp}

\begin{proof}
	Substituting the optimal control into the state equation, integrating the state equation with respect to the index $u$, and using the exact law of large numbers (Theorem \ref{Thm:2.1} and Proposition \ref{Prp:3.4}) yield the result.
\end{proof}

\begin{thm}[Collection of necessary conditions in the general case]
	Assume that Assumptions \ref{Ass:05}, \ref{Ass:01}, and \ref{Ass:02} hold. For fixed $\alpha^l\in\mathcal{U}^l$, suppose that the follower problem admits a Nash equilibrium pair $(\hat{X}^f,\hat{\alpha}^f)$. Then the forward-backward system
	\begin{equation*}
		\begin{cases}
			\mathrm{d}z_t^u=\int_I G(u,v)b_t^{f,v}(\hat{X}_t^{f,v},\hat{\alpha}_t^{f,v},z_t^v,X_t^l,\alpha_t^l)\lambda(\mathrm{d}v)\mathrm{d}t,\\
			\mathrm{d}p_t^{f,u}=-\big\{\partial_{x^f} b_t^{f,u}(X_t^{f,u},\alpha_t^{f,u},z_t^u,X_t^l,\alpha_t^l)^\top p_t^{f,u}\\
			\qquad\qquad +\partial_{x^f}\sigma_t^{f,u}(X_t^{f,u},\alpha_t^{f,u},z_t^u,X_t^l,\alpha_t^l)^\top q_t^{f,u}\\
			\qquad\qquad +\partial_{x^f} h_t^{f,u}(X_t^{f,u},\alpha_t^{f,u},z_t^u,X_t^l,\alpha_t^l)\big\}\mathrm{d}t +q_t^{f,u}\mathrm{d}W_t^{f,u}+q_t^{l,u}\mathrm{d}W_t^l,\\
			z_0^u=\int_I G(u,v)x_0^{f,v}\lambda(\mathrm{d}v),\quad p_T^{f,u}=\partial_{x^f} g^{f,u}(X_T^{f,u},z_T^u,X_T^l),
		\end{cases}
	\end{equation*}
	admits a unique solution $(z,p^f,q^f,q^l)\in\bar{\mathbb{S}}_{\boxtimes^l}^{\infty,2}(\mathbb{R}^{n_1})\times\mathbb{S}_{\boxtimes^l}^2(\mathbb{R}^{n_1})\times\mathbb{H}_{\boxtimes^l}^2(\mathbb{R}^{n_1})\times\mathbb{H}_{\boxtimes^l}^2(\mathbb{R}^{n_1})$, and the Nash equilibrium $\hat{\alpha}$ satisfies, for any $(t,u)\in[0,T]\times I$,
	\begin{equation*}
		\hat{\alpha}_t^{f,u}=\arg\min_\alpha H_t^{f,u}(\hat{X}_t^{f,u},\alpha,z_t^u,\hat{p}_t^{f,u},\hat{q}_t^{f,u}).
	\end{equation*}
\end{thm}

Substituting the LQ case and applying the sufficient maximum principle \cite[Volume I, Theorem 6.16]{Carmona-Delarue-2018}, we obtain the following result.

\begin{prp}[Necessary and sufficient maximum principle]\label{Prp:4.3}
	Assume that Assumption \ref{Ass:LQ} holds. For fixed $\alpha^l\in\mathcal{U}^l$ and $z\in\bar{\mathbb{S}}_{\boxtimes^l}^{\infty,2}(\mathbb{R}^{n_1})$, the follower problem admits a Nash equilibrium if and only if the Hamiltonian system
	\begin{equation}\label{Eq:follower_hamilton}
		\begin{cases}
			\mathrm{d}\hat{X}_t^{f,u}=\Big\{ b_t^{f,u,0}+b_t^{f,u,1}\hat{X}_t^{f,u}+b_t^{f,u,2}\hat{\alpha}_t^{f,u}
             +b_t^{f,u,3}z_t^u+b_t^{f,u,4}X_t^l+b_t^{f,u,5}\alpha_t^l \Big\}\mathrm{d}t\\
			\qquad\qquad +\Big\{ \sigma_t^{f,u,0}+\sigma_t^{f,u,1}\hat{X}_t^{f,u}+\sigma_t^{f,u,2}\hat{\alpha}_t^{f,u}
             +\sigma_t^{f,u,3}z_t^u+\sigma_t^{f,u,4}X_t^l+\sigma_t^{f,u,5}\alpha_t^l \Big\}\mathrm{d}W_t^{f,u},\\
			\mathrm{d}p_t^{f,u}=-\Big\{ b_t^{f,u,1\top}p_t^{f,u}+\sigma_t^{f,u,1\top}q_t^{f,u}+Q_t^{f,u}(\hat{X}_t^{f,u}-h_t^{f,u,1}z_t^u-h_t^{f,u,2}X_t^l) \Big\}\mathrm{d}t\\
			\qquad\qquad +q_t^{f,u}\mathrm{d}W_t^{f,u}+q_t^{l,u}\mathrm{d}W_t^l,\\
			X_0^{f,u}=x_0^{f,u},\quad p_T^{f,u}=G^{f,u}(\hat{X}_T^{f,u}-g^{f,u,1}z_T^u-g^{f,u,2}X_T^l),
		\end{cases}
	\end{equation}
	admits a unique solution $(\hat{X}^f,p^f,q^f,q^l)\in\mathbb{S}_{\boxtimes^l}^2(\mathbb{R}^{n_1})\times\mathbb{S}_{\boxtimes^l}^2(\mathbb{R}^{n_1})\times\mathbb{H}_{\boxtimes^l}^2(\mathbb{R}^{n_1})\times\mathbb{H}_{\boxtimes^l}^2(\mathbb{R}^{n_1})$, where the equilibrium strategy $\hat{\alpha}^f$ satisfies
	\begin{equation}\label{eq05}
		R_t^{f,u}(\hat{\alpha}_t^{f,u}-h_t^{f,u,3}\alpha_t^l)+b_t^{f,u,2\top}p_t^{f,u}+\sigma_t^{f,u,2\top}q_t^{f,u}=0.
	\end{equation}
\end{prp}

As in the treatment of general linear FBSDEs, assume that there exist a deterministic process $P^f$ and a random process $\varphi^f$ such that
\begin{equation}\label{eq04}
	p_t^{f,u}=P_t^{f,u}\hat{X}_t^{f,u}+\varphi_t^{f,u},
\end{equation}
and write $\mathrm{d}\varphi_t^{f,u}=\varPhi_t^{f,u}\mathrm{d}t+\varPsi_t^{f,u}\mathrm{d}W_t^l$. Substituting (\ref{eq04}) into (\ref{eq05}) gives
\begin{equation}\label{eq07}
	R_t^{f,u}(\hat{\alpha}_t^{f,u}-h_t^{f,u,3}\alpha_t^l)+b_t^{f,u,2\top}P_t^{f,u}\hat{X}_t^{f,u}+b_t^{f,u,2\top}\varphi_t^{f,u}+\sigma_t^{f,u,2\top}q_t^{f,u}=0.
\end{equation}

Substituting (\ref{eq04}) into the Hamiltonian system (\ref{Eq:follower_hamilton}) and comparing the coefficients of $\mathrm{d}W_t^l$ and $\mathrm{d}W_t^{f,u}$, we obtain
\begin{equation}\label{eq12}
	q_t^{l,u}=\varPsi_t^{f,u},
\end{equation}
\begin{equation}\label{eq08} q_t^{f,u}=P_t^{f,u}(\sigma_t^{f,u,0}+\sigma_t^{f,u,1}\hat{X}_t^{f,u}+\sigma_t^{f,u,2}\hat{\alpha}_t^{f,u}+\sigma_t^{f,u,3}z_t^u+\sigma_t^{f,u,4}X_t^l+\sigma_t^{f,u,5}\alpha_t^l).
\end{equation}

Set
\begin{equation*}
	\begin{cases}
		\hat{R}_t^{f,u}=R_t^{f,u}+\sigma_t^{f,u,2\top}P_t^{f,u}\sigma_t^{f,u,2},\\
		\hat{S}_t^{f,u}=b_t^{f,u,2\top}P_t^{f,u}+\sigma_t^{f,u,2\top}P_t^{f,u}\sigma_t^{f,u,1}.
	\end{cases}
\end{equation*}
Assume that $\hat{R}_t^{f,u}$ is invertible for any $(t,u)$. By (\ref{eq07}) and (\ref{eq08}), we have
\begin{equation}\label{eq10}
	\begin{aligned}
		\alpha_t^{f,u}=&-(\hat{R}_t^{f,u})^{-1}\big[ \hat{S}_t^{f,u}X_t^{f,u}+\sigma_t^{f,u,2\top}P_t^{f,u}\sigma_t^{f,u,3}z_t^u+b_t^{f,u,2\top}\varphi_t^{f,u}\\
		&+\sigma_t^{f,u,2\top}P_t^{f,u}\sigma_t^{f,u,4}X_t^l+(\sigma_t^{f,u,2\top}P_t^{f,u}\sigma_t^{f,u,5}-R_t^{f,u}h_t^{f,u,3})\alpha_t^l
        +\sigma_t^{f,u,2\top}P_t^{f,u}\sigma_t^{f,u,0} \big],
	\end{aligned}
\end{equation}
\begin{equation}\label{eq09}
	\begin{aligned}
		q_t^{f,u}=&\ P_t^{f,u}\big\{[\sigma_t^{f,u,1}-\sigma_t^{f,u,2}(\hat{R}_t^{f,u})^{-1}\hat{S}_t^{f,u}]X_t^{f,u}\\
		&+[\sigma_t^{f,u,3}-\sigma_t^{f,u,2}(\hat{R}_t^{f,u})^{-1}\sigma_t^{f,u,2\top}P_t^{f,u}\sigma_t^{f,u,3} ]z_t^u
        -\sigma_t^{f,u,2}(\hat{R}_t^{f,u})^{-1}b_t^{f,u,2\top}\varphi_t^{f,u}\\
		&+[\sigma_t^{f,u,4}-\sigma_t^{f,u,2}(\hat{R}_t^{f,u})^{-1}\sigma_t^{f,u,2\top}P_t^{f,u}\sigma_t^{f,u,4} ]X_t^l\\
		&+[\sigma_t^{f,u,5}-\sigma_t^{f,u,2}(\hat{R}_t^{f,u})^{-1}(\sigma_t^{f,u,2\top}P_t^{f,u}\sigma_t^{f,u,5}-R_t^{f,u}h_t^{f,u,3})]\alpha_t^l\\
		&+[\sigma_t^{f,u,0}-\sigma_t^{f,u,2}(\hat{R}_t^{f,u})^{-1}\sigma_t^{f,u,2\top}P_t^{f,u}\sigma_t^{f,u,0}]\big\}.
	\end{aligned}
\end{equation}

Substituting (\ref{eq04}), (\ref{eq12}), (\ref{eq10}), and (\ref{eq09}) back into the Hamiltonian system and comparing the coefficients of the $\mathrm{d}t$ terms, we obtain the Riccati equation and the forward-backward system
\begin{equation}\label{Eq:follower_riccati}
	\begin{cases}
		\mathrm{d}P_t^{f,u}+\big\{ P_t^{f,u}b_t^{f,u,1}+b_t^{f,u,1\top}P_t^{f,u}+\sigma_t^{f,u,1\top}P_t^{f,u}\sigma_t^{f,u,1}\\
		\qquad\qquad -\hat{S}_t^{f,u\top}(\hat{R}_t^{f,u})^{-1}\hat{S}_t^{f,u}+Q_t^{f,u} \big\}\mathrm{d}t=0,\\
		P_T^{f,u}=G^{f,u},
	\end{cases}
\end{equation}
\begin{equation}\label{eq13}
	\begin{cases}
		\mathrm{d}X_t^{f,u}=\hat{b}_t^{f,u}(X_t^{f,u},\varphi_t^{f,u},z_t^u,X_t^l,\alpha_t^l)\mathrm{d}t\\
		\qquad\qquad+\hat{\sigma}_t^{f,u}(X_t^{f,u},\varphi_t^{f,u},z_t^u,X_t^l,\alpha_t^l)\mathrm{d}W_t^{f,u},\\
		\mathrm{d}\varphi_t^{f,u}=\hat{g}_t^{f,u}(\varphi_t^{f,u},z_t^u,X_t^l,\alpha_t^l)\mathrm{d}t+q_t^{l,u}\mathrm{d}W_t^l,\\
		X_0^{f,u}=x_0^{f,u},\quad \varphi_T^{f,u}=-G^{f,u}(g^{f,u,1}z_T^u+g^{f,u,2}X_T^l),
	\end{cases}
\end{equation}
where
\begin{equation*}
	\begin{cases}		
        \hat{b}_t^{f,u}(x^f,\varphi^f,Gx^f,x^l,\alpha^l)=\hat{b}_t^{f,u,0}+\hat{b}_t^{f,u,1}x^f+\hat{b}_t^{f,u,2}\varphi^f+\hat{b}_t^{f,u,3}Gx^f
        +\hat{b}_t^{f,u,4}x^l+\hat{b}_t^{f,u,5}\alpha^l,\\
		\hat{b}_t^{f,u,0}=b_t^{f,u,0}-b_t^{f,u,2}(\hat{R}_t^{f,u})^{-1}\sigma_t^{f,u,2\top}P_t^{f,u}\sigma_t^{f,u,0},\\
		\hat{b}_t^{f,u,1}=b_t^{f,u,1}-b_t^{f,u,2}(\hat{R}_t^{f,u})^{-1}\hat{S}_t^{f,u},\\
		\hat{b}_t^{f,u,2}=-b_t^{f,u,2}(\hat{R}_t^{f,u})^{-1}b_t^{f,u,2\top},\\
		\hat{b}_t^{f,u,3}=b_t^{f,u,3}-b_t^{f,u,2}(\hat{R}_t^{f,u})^{-1}\sigma_t^{f,u,2\top}P_t^{f,u}\sigma_t^{f,u,3},\\
		\hat{b}_t^{f,u,4}=b_t^{f,u,4}-b_t^{f,u,2}(\hat{R}_t^{f,u})^{-1}\sigma_t^{f,u,2\top}P_t^{f,u}\sigma_t^{f,u,4},\\
		\hat{b}_t^{f,u,5}=b_t^{f,u,5}-b_t^{f,u,2}(\hat{R}_t^{f,u})^{-1}(\sigma_t^{f,u,2\top}P_t^{f,u}\sigma_t^{f,u,5}-R_t^{f,u}h_t^{f,u,3}),\\	
        \hat{\sigma}_t^{f,u}(x^f,\varphi^f,Gx^f,x^l,\alpha^l)=\hat{\sigma}_t^{f,u,0}+\hat{\sigma}_t^{f,u,1}x^f+\hat{\sigma}_t^{f,u,2}\varphi^f
        +\hat{\sigma}_t^{f,u,3}Gx^f+\hat{\sigma}_t^{f,u,4}x^l+\hat{\sigma}_t^{f,u,5}\alpha^l,\\
		\hat{\sigma}_t^{f,u,0}=\sigma_t^{f,u,0}-\sigma_t^{f,u,2}(\hat{R}_t^{f,u})^{-1}\sigma_t^{f,u,2\top}P_t^{f,u}\sigma_t^{f,u,0},\\
		\hat{\sigma}_t^{f,u,1}=\sigma_t^{f,u,1}-\sigma_t^{f,u,2}(\hat{R}_t^{f,u})^{-1}\hat{S}_t^{f,u},\\
		\hat{\sigma}_t^{f,u,2}=-\sigma_t^{f,u,2}(\hat{R}_t^{f,u})^{-1}b_t^{f,u,2\top},\\
		\hat{\sigma}_t^{f,u,3}=\sigma_t^{f,u,3}-\sigma_t^{f,u,2}(\hat{R}_t^{f,u})^{-1}\sigma_t^{f,u,2\top}P_t^{f,u}\sigma_t^{f,u,3},\\
		\hat{\sigma}_t^{f,u,4}=\sigma_t^{f,u,4}-\sigma_t^{f,u,2}(\hat{R}_t^{f,u})^{-1}\sigma_t^{f,u,2\top}P_t^{f,u}\sigma_t^{f,u,4},\\
		\hat{\sigma}_t^{f,u,5}=\sigma_t^{f,u,5}-\sigma_t^{f,u,2}(\hat{R}_t^{f,u})^{-1}(\sigma_t^{f,u,2\top}P_t^{f,u}\sigma_t^{f,u,5}-R_t^{f,u}h_t^{f,u,3}),\\
		\hat{g}_t^{f,u}(\varphi^f,Gx^f,x^l,\alpha^l)=\hat{g}_t^{f,u,0}+\hat{g}_t^{f,u,1}\varphi^f+\hat{g}_t^{f,u,2}Gx^f+\hat{g}_t^{f,u,3}x^l+\hat{g}_t^{f,u,4}\alpha^l,\\
		\hat{g}_t^{f,u,0}=-P_t^{f,u}b_t^{f,u,0}-\sigma_t^{f,u,1\top}P_t^{f,u}\sigma_t^{f,u,0}
        +\hat{S}_t^{f,u\top}(\hat{R}_t^{f,u})^{-1}\sigma_t^{f,u,2\top}P_t^{f,u}\sigma_t^{f,u,0},\\
		\hat{g}_t^{f,u,1}=-b_t^{f,u,1\top}+\hat{S}_t^{f,u\top}(\hat{R}_t^{f,u})^{-1}b_t^{f,u,2\top},\\
		\hat{g}_t^{f,u,2}=-P_t^{f,u}b_t^{f,u,3}-\sigma_t^{f,u,1\top}P_t^{f,u}\sigma_t^{f,u,3}+Q_t^{f,u}h_t^{f,u,1}
        +\hat{S}_t^{f,u\top}(\hat{R}_t^{f,u})^{-1}\sigma_t^{f,u,2\top}P_t^{f,u}\sigma_t^{f,u,3},\\
		\hat{g}_t^{f,u,3}=-P_t^{f,u}b_t^{f,u,4}-\sigma_t^{f,u,1\top}P_t^{f,u}\sigma_t^{f,u,4}+Q_t^{f,u}h_t^{f,u,2}
        +\hat{S}_t^{f,u\top}(\hat{R}_t^{f,u})^{-1}\sigma_t^{f,u,2\top}P_t^{f,u}\sigma_t^{f,u,4},\\
		\hat{g}_t^{f,u,4}=-P_t^{f,u}b_t^{f,u,5}-\sigma_t^{f,u,1\top}P_t^{f,u}\sigma_t^{f,u,5}
        +\hat{S}_t^{f,u\top}(\hat{R}_t^{f,u})^{-1}(\sigma_t^{f,u,2\top}P_t^{f,u}\sigma_t^{f,u,5}-R_t^{f,u}h_t^{f,u,3}).
	\end{cases}
\end{equation*}

\begin{Ass}\label{Ass:03}
	The equation coefficients are independent of $u$.
\end{Ass}

When the coefficients are independent of $u$, the follower index $u$ is omitted from the notation.

\begin{prp}[Consistency condition]\label{Prp:4.4}
	Assume that Assumptions \ref{Ass:LQ} and \ref{Ass:03} hold. Then $z=G\hat{X}^f$ if and only if
	\begin{equation*}
		\begin{cases}
			\mathrm{d}z_t^u=(\hat{b}_t^{f,1}z_t^u+\hat{b}_t^{f,2}G\varphi_t^{f,u}+\hat{b}_t^{f,3}Gz_t^u)\mathrm{d}t\\
			\qquad\qquad +\|G(u,\cdot)\|_1(\hat{b}_t^{f,0}+\hat{b}_t^{f,4}X_t^l+\hat{b}_t^{f,5}\alpha_t^l)\mathrm{d}t,\\
			z_0^u=Gx_0^{f,u}.
		\end{cases}
	\end{equation*}
\end{prp}

\begin{proof}
	Integrating both sides of the forward equation and using the exact law of large numbers gives
	\begin{equation*}
		X_t^{f,u}=\int_0^t \hat{b}_s^{f}(X_s^{f,u},\varphi_s^{f,u},z_s^u,X_s^l,\alpha_s^l)\mathrm{d}s+\int_0^t \hat{\sigma}_s^{f}(X_s^{f,u},\varphi_s^{f,u},z_s^u,X_s^l,\alpha_s^l)\mathrm{d}W_s^{f,u},
	\end{equation*}
	\begin{equation*}
		\begin{aligned}
			GX_t^{f,u}&=\int_I G(u,v) \int_0^t \hat{b}_s^{f}(X_s^{f,v},\varphi_s^{f,v},z_s^v,X_s^l,\alpha_s^l)\mathrm{d}s\lambda(\mathrm{d}v)\\
			&=\int_0^t \int_I G(u,v)\hat{b}_s^{f}(X_s^{f,v},\varphi_s^{f,v},z_s^v,X_s^l,\alpha_s^l)\lambda(\mathrm{d}v)\mathrm{d}s.
		\end{aligned}
	\end{equation*}
	Substituting the coefficients shows that $GX^f$ satisfies the equation
	\begin{equation*}
		\begin{cases}
			\mathrm{d}GX_t^{f,u}=(\hat{b}_t^{f,1}GX_t^{f,u}+\hat{b}_t^{f,2}G\varphi_t^{f,u}+\hat{b}_t^{f,3}Gz_t^u)\mathrm{d}t\\
			\qquad\qquad\quad +\|G(u,\cdot)\|_1(\hat{b}_t^{f,0}+\hat{b}_t^{f,4}X_t^l+\hat{b}_t^{f,5}\alpha_t^l)\mathrm{d}t,\\
			GX_0^{f,u}=Gx_0^{f,u},
		\end{cases}
	\end{equation*}
	and the proposition follows from the existence and uniqueness of solutions to linear SDEs.
\end{proof}

Writing the consistency condition and the inhomogeneous tail terms as a forward-backward system gives
\begin{equation}\label{Eq:follower_consistent}
	\begin{cases}
		\mathrm{d}z_t^u=(\hat{b}_t^{f,1}z_t^u+\hat{b}_t^{f,2}G\varphi_t^{f,u}+\hat{b}_t^{f,3}Gz_t^u)\mathrm{d}t\\
		\qquad\qquad +\|G(u,\cdot)\|_1(\hat{b}_t^{f,0}+\hat{b}_t^{f,4}X_t^l+\hat{b}_t^{f,5}\alpha_t^l)\mathrm{d}t,\\
		\mathrm{d}\varphi_t^{f,u}=(\hat{g}_t^{f,0}+\hat{g}_t^{f,1}\varphi_t^{f,u}+\hat{g}_t^{f,2}z_t^u+\hat{g}_t^{f,3}X_t^l+\hat{g}_t^{f,4}\alpha_t^l)\mathrm{d}t+q_t^{l,u}\mathrm{d}W_t^l,\\
		z_0^u=Gx_0^{f,u},\quad \varphi_T^{f,u}=-G^{f}(g^{f,1}z_T^u+g^{f,2}X_T^l).
	\end{cases}
\end{equation}

\begin{thm}[Combined conclusion]\label{Thm:4.2}
	Assume that Assumptions \ref{Ass:LQ} and \ref{Ass:03} hold, that the Riccati equation (\ref{Eq:follower_riccati}) admits a unique solution $P^f\in L^\infty(0,T;\mathcal{S}^{n_1})$, and that the forward-backward system (\ref{Eq:follower_consistent}) admits a unique solution $(z,\varphi^f,q^l)\in\bar{\mathbb{S}}_{\boxtimes^l}^{\infty,2}(\mathbb{R}^{n_1})\times\mathbb{S}_{\boxtimes^l}^2(\mathbb{R}^{n_1})\times\mathbb{H}_{\boxtimes^l}^2(\mathbb{R}^{n_1})$. Then, for fixed $\alpha^l\in\mathcal{U}^l$, the follower problem admits a unique Nash equilibrium $\hat{\alpha}^f$. Denote the Nash equilibrium pair by $(\hat{X}^f,\hat{\alpha}^f)$. The equilibrium strategy $\hat{\alpha}^f$ satisfies
	\begin{equation*}
		\hat{\alpha}_t^{f,u}=\hat{\alpha}^f_t(\hat{X}_t^{f,u},z_t^u,\varphi_t^{f,u},X_t^l,\alpha_t^l),
	\end{equation*}
	where
	\begin{equation*}
		\begin{cases}
			\hat{\alpha}^f(x^f,z,\varphi^f,x^l,\alpha^l)=\hat{\alpha}_t^{f,0}+\hat{\alpha}_t^{f,1}x^f+\hat{\alpha}_t^{f,2}z+\hat{\alpha}_t^{f,3}\varphi^f+\hat{\alpha}_t^{f,4}x^l+\hat{\alpha}_t^{f,5}\alpha^l,\\
			\hat{\alpha}_t^{f,0}=-(\hat{R}_t^f)^{-1}\sigma_t^{f,2\top}P_t^{f}\sigma_t^{f,0},\quad \hat{\alpha}_t^{f,1}=-(\hat{R}_t^f)^{-1}\hat{S}_t^f,\\
			\hat{\alpha}_t^{f,2}=-(\hat{R}_t^{f})^{-1}\sigma_t^{f,2\top}P_t^f\sigma_t^{f,3},\quad \hat{\alpha}_t^{f,3}=-(\hat{R}_t^f)^{-1}b_t^{f,2\top},\\
			\hat{\alpha}_t^{f,4}=-(\hat{R}_t^f)^{-1}\sigma_t^{f,2\top}P_t^f\sigma_t^{f,4},\quad \hat{\alpha}_t^{f,5}=-(\hat{R}_t^f)^{-1}(\sigma_t^{f,2\top}P_t^f\sigma_t^{f,5}-R_t^fh_t^{f,3}).
		\end{cases}
	\end{equation*}
\end{thm}

\begin{proof}
	The conclusion follows from Propositions \ref{Prp:4.3} and \ref{Prp:4.4} and from the derivation of the Riccati equation.
\end{proof}

Sufficient conditions for the existence and uniqueness of a solution to the forward-backward system (\ref{Eq:follower_consistent}) can be obtained by using eigenvalue decomposition and the method of continuation.

Since the graphon operator is a compact self-adjoint operator, there exist a sequence of unit eigenvectors $\{e_n\}_{n=1}^\infty\subset L_\mathcal{I}^2(\mathbb{R}^{n_1})$ and eigenvalues $\{\lambda_n\}_{n=1}^\infty$ satisfying $|\lambda_i|\ge |\lambda_j|$ for any $i>j$, and, for any $u\in I$,
\begin{equation*}
	\int_I G(u,v)e_n^v\lambda(\mathrm{d}v)=\lambda_n e_n^u.
\end{equation*}

By the properties of compact self-adjoint operators, write
\begin{equation*}
	z_t^n=\langle z_t,e_n \rangle,\quad \varphi_t^n=\langle \varphi_t,e_n \rangle,\quad \varphi_t^0=\varphi_t-\sum_{n=1}^\infty \langle \varphi_t,e_n \rangle e_n,\quad 
\end{equation*}
and, for any $t$, we have the eigenvector expansions
\begin{equation*}
	z_t=\sum_{n=1}^\infty z_t^n e_n,\quad \varphi=\sum_{n=1}^\infty \varphi_t^n e_n+\varphi_t^0.
\end{equation*}

\begin{prp}
	If for any positive integer $n$, the linear FBSDE
	\begin{equation}\label{eq15}
		\begin{cases}
			\mathrm{d}z_t^n=\big\{ (\hat{b}_t^{f,1}+\lambda_n \hat{b}_t^{f,3})z_t^n+\lambda_n\hat{b}_t^{f,2}\varphi_t^n+\|G(u,\cdot)\|_1^n(\hat{b}_t^{f,0}+\hat{b}_t^{f,4}X_t^l+\hat{b}_t^{f,5}\alpha_t^l)\big\}\mathrm{d}t,\\
			\mathrm{d}\varphi_t^n=\big\{ \hat{g}_t^{f,2}z_t^n+\hat{g}_t^{f,1}\varphi_t^n+(\hat{g}_t^{f,0}+\hat{g}_t^{f,3}X_t^l+\hat{g}_t^{f,4}\alpha_t^l)\mathbf{1}^n \big\}\mathrm{d}t+q_t^{l,n}\mathrm{d}W_t^l,\\
			z_0^n=\lambda_n x_0^n,\quad \varphi_T^n=-G^f(\lambda_n g^{f,1} z_T^n+g^{f,2}X_T^l\mathbf{1}^n),
		\end{cases}
	\end{equation}
	admits a unique solution $(z^n,\varphi^n,q^{l,n})\in\mathbb{S}_{\boxtimes^l}^2(\mathbb{R}^{n_1})\times\mathbb{S}_{\boxtimes^l}^2(\mathbb{R}^{n_1})\times\mathbb{H}_{\boxtimes^l}^2(\mathbb{R}^{n_1})$, and the linear BSDE
	\begin{equation}\label{eq16}
		\begin{cases}
			\mathrm{d}\varphi_t^0=\big\{ \hat{g}_t^{f,1}\varphi_t^0+(\hat{g}_t^{f,0}+\hat{g}_t^{f,3}X_t^l+\hat{g}_t^{f,4}\alpha_t^l)\mathbf{1}^0 \big\}\mathrm{d}t+q_t^{l,0}\mathrm{d}W_t^l,\\
			\varphi_T^0=-G^fg^{f,2}X_T^l\mathbf{1}^0,
		\end{cases}
	\end{equation}
	admits a unique solution $(\varphi^0,q^{l,0})\in\mathbb{S}_{\boxtimes^l}^2(\mathbb{R}^{n_1})\times\mathbb{H}_{\boxtimes^l}^2(\mathbb{R}^{n_1})$. Define
	\begin{equation}\label{eq33}
		z_t=\sum_{n=1}^\infty z_t^ne_n,\quad \varphi^f_t=\sum_{n=1}^\infty \varphi_t^n e_n+\varphi_t^0,\quad q^{l}_t=\sum_{n=1}^\infty q_t^{l,n}e_n+q_t^{l,0}.
	\end{equation}
	If $(z,\varphi^f,q^l)\in\mathbb{S}_{\boxtimes^l}^2(\mathbb{R}^{n_1})\times\mathbb{S}_{\boxtimes^l}^2(\mathbb{R}^{n_1})\times\mathbb{H}_{\boxtimes^l}^2(\mathbb{R}^{n_1})$, then $(z,\varphi^f,q^l)$ is the unique solution of the forward-backward system (\ref{Eq:follower_consistent}).
\end{prp}

\begin{proof}
	Existence follows by writing out the equation of $(z,\varphi,q^l)$ from (\ref{eq15}), (\ref{eq16}), and (\ref{eq33}). Uniqueness follows from the uniqueness of the eigenvalue decomposition and the uniqueness of solutions to (\ref{eq15}) and (\ref{eq16}).
\end{proof}

\begin{Ass}[Uniform monotonicity condition]\label{Ass:07}
	For any eigenvalue $\lambda_n$, one of the following two conditions holds:
	
	(1) there exists a constant $K>0$ such that, for any $\theta=(\theta_1,\theta_2)\in\mathbb{R}^{2n_1}$ and any $t\in[0,T]$,
	\begin{equation*}
		\begin{pmatrix}
			\theta_1 \\ \theta_2
		\end{pmatrix}^\top\begin{pmatrix}
			\hat{g}_t^{f,2} & \hat{g}_t^{f,1} \\
			\hat{b}_t^{f,1}+\lambda_n\hat{b}_t^{f,3} & \lambda_n\hat{b}_t^{f,2}
		\end{pmatrix}\begin{pmatrix}
			\theta_1 \\ \theta_2
		\end{pmatrix}\le -K|\theta_1|^2,\quad G^fg^{f,1}<0,
	\end{equation*}
	
	(2) there exists a constant $K>0$ such that, for any $\theta=(\theta_1,\theta_2)\in\mathbb{R}^{2n_1}$ and any $t\in[0,T]$,
	\begin{equation*}
		\begin{pmatrix}
			\theta_1 \\ \theta_2
		\end{pmatrix}^\top\begin{pmatrix}
			\hat{g}_t^{f,2} & \hat{g}_t^{f,1} \\
			\hat{b}_t^{f,1}+\lambda_n\hat{b}_t^{f,3} & \lambda_n\hat{b}_t^{f,2}
		\end{pmatrix}\begin{pmatrix}
			\theta_1 \\ \theta_2
		\end{pmatrix}\ge K|\theta_1|^2,\quad G^fg^{f,1}>0.
	\end{equation*}
\end{Ass}

\begin{Ass}[Stronger uniform monotonicity condition]\label{Ass:08}
	One of the following two conditions holds:
	
	(1) there exists a constant $K>0$ such that, for any $\theta=(\theta_1,\theta_2)\in\mathbb{R}^{2n_1}$ and any $t\in[0,T]$,
	\begin{gather*}
		\begin{pmatrix}
			\theta_1 \\ \theta_2
		\end{pmatrix}^\top\begin{pmatrix}
			\hat{g}_t^{f,2} & \hat{g}_t^{f,1} \\
			\hat{b}_t^{f,1}-\hat{b}_t^{f,3} & -\hat{b}_t^{f,2}
		\end{pmatrix}\begin{pmatrix}
			\theta_1 \\ \theta_2
		\end{pmatrix}\le -K|\theta_1|^2,\\
		\begin{pmatrix}
			\theta_1 \\ \theta_2
		\end{pmatrix}^\top\begin{pmatrix}
			\hat{g}_t^{f,2} & \hat{g}_t^{f,1} \\
			\hat{b}_t^{f,1}+\hat{b}_t^{f,3} & \hat{b}_t^{f,2}
		\end{pmatrix}\begin{pmatrix}
			\theta_1 \\ \theta_2
		\end{pmatrix}\le -K|\theta_1|^2,\quad G^fg^{f,1}<0,
	\end{gather*}
	
	(2) there exists a constant $K>0$ such that, for any $\theta=(\theta_1,\theta_2)\in\mathbb{R}^{2n_1}$ and any $t\in[0,T]$,
	\begin{gather*}
		\begin{pmatrix}
			\theta_1 \\ \theta_2
		\end{pmatrix}^\top\begin{pmatrix}
			\hat{g}_t^{f,2} & \hat{g}_t^{f,1} \\
			\hat{b}_t^{f,1}-\hat{b}_t^{f,3} & -\hat{b}_t^{f,2}
		\end{pmatrix}\begin{pmatrix}
			\theta_1 \\ \theta_2
		\end{pmatrix}\ge K|\theta_1|^2,\\
		\begin{pmatrix}
			\theta_1 \\ \theta_2
		\end{pmatrix}^\top\begin{pmatrix}
			\hat{g}_t^{f,2} & \hat{g}_t^{f,1} \\
			\hat{b}_t^{f,1}+\hat{b}_t^{f,3} & \hat{b}_t^{f,2}
		\end{pmatrix}\begin{pmatrix}
			\theta_1 \\ \theta_2
		\end{pmatrix}\ge K|\theta_1|^2,\quad G^fg^{f,1}>0.
	\end{gather*}
\end{Ass}

If Assumption \ref{Ass:08} holds, then Assumption \ref{Ass:07} holds.

\begin{prp}
	Assume that Assumption \ref{Ass:07} holds. Then, for any positive integer $n$, the linear FBSDE (\ref{eq15}) admits a unique solution $(z^n,\varphi^n,q^{l,n})\in\mathbb{S}_{\boxtimes^l}^2(\mathbb{R}^{n_1})\times\mathbb{S}_{\boxtimes^l}^2(\mathbb{R}^{n_1})\times\mathbb{H}_{\boxtimes^l}^2(\mathbb{R}^{n_1})$, the linear BSDE (\ref{eq16}) admits a unique solution, and $(z,\varphi^f,q^l)$ defined by (\ref{eq33}) belongs to $\mathbb{S}_{\boxtimes^l}^2(\mathbb{R}^{n_1})\times\mathbb{S}_{\boxtimes^l}^2(\mathbb{R}^{n_1})\times\mathbb{H}_{\boxtimes^l}^2(\mathbb{R}^{n_1})$ and is the unique solution of the forward-backward system (\ref{Eq:follower_consistent}).
\end{prp}

\begin{proof}
	The existence and uniqueness of solutions to the component equation (\ref{eq15}) follow directly from the method of continuation for classical linear FBSDEs; see \cite{Hu-Peng-1995,Yong-1997,Peng-Wu-1999}. Moreover, the following $L^2$ estimate holds:
	\begin{equation*}
    \begin{aligned}
		&\mathbb{E}^\boxtimes\biggl[ \sup_{0\le t\le T}\big\{ |z_t^n|^2+|\varphi_t^n|^2 \big\}+\int_0^T |q_t^{l,n}|^2\mathrm{d}t \biggr]\\
		&\qquad \le K\mathbb{E}^\boxtimes\biggl[ \int_0^T \big\{ |\hat{b}_t^{f,0}|^2+|X_t^l|^2+|\alpha_t^l|^2 \big\}(\|G(u,\cdot)\|_1^n)^2\mathrm{d}t\\
		&\qquad\qquad +\int_0^T \big\{ |\hat{g}_t^{f,0}|^2+|X_t^l|^2+|\alpha_t^l|^2 \big\}(\mathbf{1}^n)^2\mathrm{d}t+|X_T^l|^2(\mathbf{1}^n)^2 \biggr],
	\end{aligned}
    \end{equation*}
    where the constant $K$ depends only on the Lipschitz constant and the monotonicity constant, and is independent of the positive integer $n$.
    
	By the eigenvalue decomposition,
	\begin{equation*}
		\sum_{n=1}^\infty (\|G(u,\cdot)\|_1^n)^2=(\|G(u,\cdot)\|_1)^2,\quad \sum_{n=1}^\infty (\mathbf{1}^n)^2+(\mathbf{1}^0)^2=1,
	\end{equation*}
	we obtain
	\begin{equation*}
    \begin{aligned}
		&\sum_{n=1}^{\infty}\mathbb{E}^\boxtimes\biggl[ \sup_{0\le t\le T}\big\{ |z_t^n|^2+|\varphi_t^n|^2 \big\}+\int_0^T |q_t^{l,n}|^2\mathrm{d}t \biggr]+\mathbb{E}^\boxtimes\biggl[ \sup_{0\le t\le T}|\varphi_t^0|^2+\int_0^T|q_t^{l,0}|^2\mathrm{d}t \biggr]\\
		&\qquad \le K\mathbb{E}^\boxtimes\biggl[ \int_0^T\big\{ |\hat{b}_t^{f,0}|^2+|X_t^l|^2+|\alpha_t^l|^2 \big\}(\|G(u,\cdot)\|_1)^2\mathrm{d}t\\
		&\qquad\qquad +\int_0^T\big\{ |\hat{g}_t^{f,0}|^2+|X_t^l|^2+|\alpha_t^l|^2 \big\}\mathrm{d}t+|X_T^l|^2 \biggr]<\infty,
	\end{aligned}
    \end{equation*}
	and therefore $(z,\varphi^f,q^l)\in\mathbb{S}_{\boxtimes^l}^2(\mathbb{R}^{n_1})\times\mathbb{S}_{\boxtimes^l}^2(\mathbb{R}^{n_1})\times\mathbb{H}_{\boxtimes^l}^2(\mathbb{R}^{n_1})$.
\end{proof}

\begin{remark}
	Since, for classical linear FBSDEs, the constants in the $L^2$ estimates depend only on the Lipschitz constants of the coefficients and the constants in the monotonicity condition, the existence and uniqueness of solutions to the graphon aggregated FBSDE obtained here, together with the corresponding $L^2$ estimates, do not rely on the finite-rank condition of the graphon operator. This relaxes the requirement imposed in \cite{Xu-Gou-Huang-Gao-2025}.
\end{remark}

\begin{prp}
	If equation (\ref{Eq:follower_consistent}) admits a solution $(z,\varphi^f,q^l)\in\mathbb{S}_{\boxtimes^l}^2(\mathbb{R}^{n_1})\times\mathbb{S}_{\boxtimes^l}^2(\mathbb{R}^{n_1})\times\mathbb{H}_{\boxtimes^l}^2(\mathbb{R}^{n_1})$, then $(z,\varphi^f,q^l)\in\bar{\mathbb{S}}_{\boxtimes^l}^{\infty,2}(\mathbb{R}^{n_1})\times\mathbb{S}_{\boxtimes^l}^{\infty,2}(\mathbb{R}^{n_1})\times\mathbb{H}_{\boxtimes^l}^{\infty,2}(\mathbb{R}^{n_1})$, and the estimate
	\begin{equation*}
    \begin{aligned}
		&\sup_{u\in I}\mathbb{E}\biggl[ \sup_{0\le t\le T}\big\{|z_t^u|^2+|\varphi_t^{f,u}|^2\big\}+\int_0^T |q_t^{l,u}|^2\mathrm{d}t \biggr]\\
		&\qquad \le K\mathbb{E}\biggl[ \int_0^T\big\{ |X_t^l|^2+|\alpha_t^l|^2+|\hat{b}_t^{f,0}|^2+|\hat{g}_t^{f,0}|^2 \big\}\mathrm{d}t \biggr],
	\end{aligned}
    \end{equation*}
	holds.
\end{prp}

\begin{proof}
	By Propositions \ref{Prp:3.1} and \ref{Prp:3.4}, $Gz,G\varphi^f\in\bar{\mathbb{S}}_{\boxtimes^l}^{\infty,2}(\mathbb{R}^{n_1})$. For fixed $u$, using the classical SDE estimate gives
	\begin{equation*}
		\mathbb{E}\biggl[ \sup_{0\le t\le T}|z_t^u|^2 \biggr]\le K\mathbb{E}\biggl[ \int_0^T \big\{ |G\varphi_t^{f,u}|^2+|Gz_t^{f,u}|^2+|X_t^l|^2+|\alpha_t^l|^2+|\hat{b}_t^{f,0}|^2 \big\}\mathrm{d}t \biggr].
	\end{equation*}
	Taking the supremum with respect to $u$ yields
	\begin{equation*}
    \begin{aligned}
		&\sup_{u\in I}\mathbb{E}\biggl[ \sup_{0\le t\le T}|z_t^u|^2 \biggr]
        \le K\sup_{u\in I}\mathbb{E}\biggl[ \int_0^T\big\{|G\varphi_t^{f,u}|^2+|Gz_t^{f,u}|^2\big\}\mathrm{d}t \biggr] \\
		&\quad +K\mathbb{E}\biggl[ \int_0^T \big\{ |X_t^l|^2+|\alpha_t^l|^2+|\hat{b}_t^{f,0}|^2 \big\}\mathrm{d}t \biggr]<\infty,
	\end{aligned}
    \end{equation*}
	Combining this with Proposition \ref{Prp:3.4}, we obtain $z\in\bar{\mathbb{S}}_{\boxtimes^l}^{\infty,2}(\mathbb{R}^{n_1})$. Similarly, using BSDE estimates and taking the supremum give $(\varphi^f,q^l)\in\mathbb{S}_{\boxtimes^l}^{\infty,2}(\mathbb{R}^{n_1})\times\mathbb{H}_{\boxtimes^l}^{\infty,2}(\mathbb{R}^{n_1})$. The estimate follows from Proposition \ref{Prp:3.1} and Gronwall inequality.
\end{proof}

\subsection{The leader's problem}

\begin{Ass}\label{Ass:06}
	$w$ is an eigenvector of the graphon operator $G$, that is, there exists $\lambda_0$ such that, for any $u\in I$,
	\begin{equation*}
		\int_I G(u,v)w^v\lambda(\mathrm{d}v)=\lambda_0 w^u.
	\end{equation*}
\end{Ass}

\begin{prp}
	Assume that Assumption \ref{Ass:06} holds. For any $X\in L_\mathcal{I}^2(\mathbb{R}^{n_1})$, we have
	\begin{equation*}
		\int_I w^u GX^u\lambda(\mathrm{d}u)=\lambda_0\int_I w^u X^u\lambda(\mathrm{d}u).
	\end{equation*}
	In particular,
	\begin{equation*}
		\int_I w^u GX_t^{f,u}\lambda(\mathrm{d}u)=\lambda_0 M_t^{f}.
	\end{equation*}
\end{prp}

\begin{proof}
	By exchanging the order of integration, for any $X\in L_\mathcal{I}^2(\mathbb{R}^{n_1})$, we have
	\begin{equation*}
    \begin{aligned}
		&\int_I w^u GX^u\lambda(\mathrm{d}u)=\int_I w^u\int_I G(u,v)X^v\lambda(\mathrm{d}v)\lambda(\mathrm{d}u)\\
		&=\int_I\int_I w^uG(u,v)\lambda(\mathrm{d}u)X^v\lambda(\mathrm{d}v)=\lambda_0\int_I w^uX^u\lambda(\mathrm{d}u).
	\end{aligned}
    \end{equation*}
    The proof is complete.
\end{proof}

Set
\begin{equation*}
	N_t^f:=\int_I w^u\varphi_t^{f,u}\lambda(\mathrm{d}u),\quad L_t^{f}:=\int_I w^uq_t^{l,u}\lambda(\mathrm{d}u).
\end{equation*}
Integrating the forward-backward system (\ref{eq13}) gives
\begin{equation*}
	\begin{cases}
		\mathrm{d}M_t^f=\big\{ (\hat{b}_t^{f,1}+\hat{b}_t^{f,3}\lambda_0)M_t^f+\hat{b}_t^{f,2}N_t^f+\hat{b}_t^{f,4}X_t^l+\hat{b}_t^{f,5}\alpha_t^l+\hat{b}_t^{f,0} \big\}\mathrm{d}t,\\
		\mathrm{d}N_t^f=\big\{\hat{g}_t^{f,2}\lambda_0 M_t^f+\hat{g}_t^{f,1}N_t^{f}+\hat{g}_t^{f,3}X_t^l+\hat{g}_t^{f,4}\alpha_t^l+\hat{g}_t^{f,0}\big\}\mathrm{d}t+L_t^f\mathrm{d}W_t^l,\\
		M_0^f=m_0^f,\quad N_T^f=-G^f(g^{f,1}\lambda_0M_T^f+g^{f,2}X_T^l),
	\end{cases}
\end{equation*}
where
\begin{equation*}
	m_0^f=\int_I w^u x_0^{f,u}\lambda(\mathrm{d}u).
\end{equation*}

\begin{prp}
	If the nonsymmetric Riccati equation
	\begin{equation}\label{Eq:leader_riccati1}
		\begin{cases}
			\mathrm{d}\hat{P}_t^f+\big\{ \hat{P}_t^f(\hat{b}_t^{f,1}+\lambda_0\hat{b}_t^{f,3})-\hat{g}_t^{f,1}\hat{P}_t^f+\hat{P}_t^f\hat{b}_t^{f,2}\hat{P}_t^f-\lambda_0\hat{g}_t^{f,2} \big\}\mathrm{d}t=0,\\
			\hat{P}_T^{f}=-G^fg^{f,1}\lambda_0,
		\end{cases}
	\end{equation}
	admits a solution $\hat{P}^f\in L^\infty(0,T;\mathbb{R}^{n_1\times n_1})$, then the forward-backward system
	\begin{equation}\label{eq14}
		\begin{cases}
			\mathrm{d}M_t^f=\big\{ (\hat{b}_t^{f,1}+\lambda_0\hat{b}_t^{f,3}+\hat{b}_t^{f,2}\hat{P}_t^f)M_t^f+\hat{b}_t^{f,2}N_t^l+\hat{b}_t^{f,4}X_t^l+\hat{b}_t^{f,5}\alpha_t^l+\hat{b}_t^{f,0} \big\}\mathrm{d}t,\\
			\mathrm{d}N_t^l=\big\{ (\hat{g}_t^{f,1}-\hat{P}_t^f\hat{b}_t^{f,2})N_t^l+(\hat{g}_t^{f,3}-\hat{P}_t^f\hat{b}_t^{f,4})X_t^l+(\hat{g}_t^{f,4}-\hat{P}_t^f\hat{b}_t^{f,5})\alpha_t^l\\
			\qquad\qquad +(\hat{g}_t^{f,0}-\hat{P}_t^fb_t^{f,0}) \big\}\mathrm{d}t+L_t^f\mathrm{d}W_t^l,\\
			M_0^f=m_0^f,\quad N_T^l=-G^fg^{f,2}X_T^l,
		\end{cases}
	\end{equation}
	admits a unique solution $(M^f,N^l,L^f)\in\mathbb{S}_{\mathcal{F}^l}^2(\mathbb{R}^{n_1})\times\mathbb{S}_{\mathcal{F}^l}^2(\mathbb{R}^{n_1})\times\mathbb{H}_{\mathcal{F}^l}^2(\mathbb{R}^{n_1})$, and satisfies
	\begin{equation*}
		N_t^f=\hat{P}_t^fM_t^f+N_t^l.
	\end{equation*}
\end{prp}

Combining (\ref{eq14}) with the system equation (\ref{state}) of the leader, we write
\begin{equation*}
	\tilde{X}_t^l=\begin{pmatrix}
		X_t^l \\ M_t^f
	\end{pmatrix}.
\end{equation*}
Then $(\tilde{X}^l,N^l,L^f)$ satisfies the forward-backward system
\begin{equation*}
	\begin{cases}
		\mathrm{d}\tilde{X}_t^l=\tilde{b}_t^l(\tilde{X}_t^l,N_t^l,\alpha_t^l)\mathrm{d}t+\tilde{\sigma}_t^l(\tilde{X}_t^l,\alpha_t^l)\mathrm{d}W_t^l,\\
		\mathrm{d}N_t^l=\tilde{g}_t^l(\tilde{X}_t^l,N_t^l,\alpha_t^l)\mathrm{d}t+L_t^f\mathrm{d}W_t^l,\\
		\tilde{X}_0^l=\tilde{x}_0^l,\quad N_T^l=\tilde{G}\tilde{X}_T^l,
	\end{cases}
\end{equation*}
where
\begin{equation*}
	\begin{cases}
		\tilde{b}_t^l(\tilde{x}^l,n^l,\alpha^l)=\tilde{b}_t^{l,0}+\tilde{b}_t^{l,1}\tilde{x}^l+\tilde{b}_t^{l,2}n^l+\tilde{b}_t^{l,3}\alpha^l,\\
		\tilde{b}_t^{l,0}=\begin{pmatrix}
			b_t^{l,0} \\ \hat{b}_t^{f,0}
		\end{pmatrix},\quad \tilde{b}_t^{l,1}=\begin{pmatrix}
			b_t^{l,1} & b_t^{l,3} \\ \hat{b}_t^{f,4} & \hat{b}_t^{f,1}+\lambda_0\hat{b}_t^{f,3}+\hat{b}_t^{f,2}\hat{P}_t^f
		\end{pmatrix},\\
		\tilde{b}_t^{l,2}=\begin{pmatrix}
			0 \\ \hat{b}_t^{f,2}
		\end{pmatrix},\quad \tilde{b}_t^{l,3}=\begin{pmatrix}
			b_t^{l,2} \\ \hat{b}_t^{f,5}
		\end{pmatrix},\\
		\tilde{\sigma}_t^l(\tilde{x}^l,n^l,\alpha^l)=\tilde{\sigma}_t^{l,0}+\tilde{\sigma}_t^{l,1}\tilde{x}^l+\tilde{\sigma}_t^{l,2}\alpha^l,\\
		\tilde{\sigma}_t^{l,0}=\begin{pmatrix}
			\sigma_t^{l,0} \\ 0
		\end{pmatrix},\quad \tilde{\sigma}_t^{l,1}=\begin{pmatrix}
			\sigma_t^{l,1} & \sigma_t^{l,3} \\ 0 & 0
		\end{pmatrix},\quad \tilde{\sigma}_t^{l,2}=\begin{pmatrix}
			\sigma_t^{l,2} \\ 0
		\end{pmatrix},\\
		\tilde{g}_t^l(\tilde{x}^l,n^l,\alpha^l)=\tilde{g}_t^{l,0}+\tilde{g}_t^{l,1}\tilde{x}^l+\tilde{g}_t^{l,2}n^l+\tilde{g}_t^{l,3}\alpha^l,\\
		\tilde{g}_t^{l,0}=\hat{g}_t^{f,0}-\hat{P}_t^f\hat{b}_t^{f,0},\quad \tilde{g}_t^{l,1}=\begin{pmatrix}
			\hat{g}_t^{f,3}-\hat{P}_t^{f}\hat{b}_t^{f,4} \\ 0
		\end{pmatrix},\\
		\tilde{g}_t^{l,2}=\hat{g}_t^{f,1}-\hat{P}_t^f\hat{b}_t^{f,2},\quad \tilde{g}_t^{l,3}=\hat{g}_t^{f,4}-\hat{P}_t^f\hat{b}_t^{f,5},\\
		\tilde{x}_0^l=\begin{pmatrix}
			x_0^l \\ m_0^f
		\end{pmatrix},\quad \tilde{G}=\begin{pmatrix}
			-G^f g^{f,2} \\ 0
		\end{pmatrix}.
	\end{cases}
\end{equation*}

The cost functional is
\begin{equation*}
	\tilde{J}^l(\alpha^l)=J^l\big(\alpha^l;\hat{\alpha}^f(\alpha^l)\big)=\frac{1}{2}\mathbb{E}\biggl[ \int_0^T \big\{ \tilde{X}_t^{l\top}\tilde{Q}_t^l\tilde{X}_t^l+\alpha_t^{l\top}R_t^l\alpha_t^l \big\}\mathrm{d}t+\tilde{X}_T^{l\top}\tilde{G}^l\tilde{X}_T^l \biggr],
\end{equation*}
where
\begin{equation*}
	\tilde{Q}_t^l=\begin{pmatrix}
		Q^l & -Q^lh^{l,1} \\
		-h^{l,1\top}Q^l & h^{l,1\top}Q^lh^{l,1}
	\end{pmatrix},\quad \tilde{G}^l=\begin{pmatrix}
		G^l & -G^lg^{l,1} \\
		-g^{l,1\top}G^l & g^{l,1\top}G^l g^{l,1}
	\end{pmatrix}.
\end{equation*}

\begin{Ass}[Uniform convexity condition]\label{Ass:4.6}
	For any admissible control $\beta\in\mathcal{U}^l$, the unique solution $(\tilde{X}^0,N^0,L^0)\in\mathbb{S}_{\mathcal{F}^l}^2(\mathbb{R}^{n_1+n_2})\times\mathbb{S}_{\mathcal{F}^l}^2(\mathbb{R}^{n_1})\times\mathbb{H}_{\mathcal{F}^l}^2(\mathbb{R}^{n_1})$ of the forward-backward system
	\begin{equation*}
		\begin{cases}
			\mathrm{d}\tilde{X}_t^0=(\tilde{b}_t^{l,1}\tilde{X}_t^0+\tilde{b}_t^{l,2}N_t^0+\tilde{b}_t^{l,3}\beta_t)\mathrm{d}t+(\tilde{\sigma}_t^{l,1}\tilde{X}_t^0+\tilde{\sigma}_t^{l,2}\beta_t)\mathrm{d}W_t^l,\\
			\mathrm{d}N_t^0=(\tilde{g}_t^{l,1}\tilde{X}_t^0+\tilde{g}_t^{l,2}N_t^0+\tilde{g}_t^{l,3}\beta_t)\mathrm{d}t+L_t^0\mathrm{d}W_t^l,\\
			\tilde{X}_0^0=0,\quad N_T^l=\tilde{G}\tilde{X}_T^0
		\end{cases}
	\end{equation*}
	satisfies
	\begin{equation*}
		\mathbb{E}\biggl[ \int_0^T\big\{ \tilde{X}_t^{0\top}\tilde{Q}_t^{l}\tilde{X}_t^0+\beta_t^\top R_t^l\beta_t \big\}\mathrm{d}t+\tilde{X}_T^{0\top}G^l\tilde{X}_T^{0} \biggr]\ge 0.
	\end{equation*}
\end{Ass}

When $Q^l\ge 0$, $R^l\gg 0$, and $G^l\ge 0$, Assumption \ref{Ass:4.6} holds automatically.

\begin{prp}
	Assume that Assumptions \ref{Ass:LQ}, \ref{Ass:03}, \ref{Ass:06}, and \ref{Ass:4.6} hold. If, for any $\alpha^l\in\mathcal{U}^l$, the forward-backward system
	\begin{equation}\label{eq34}
		\begin{cases}
			\mathrm{d}\tilde{X}_t^l=(\tilde{b}_t^{l,0}+\tilde{b}_t^{l,1}\tilde{X}_t^l+\tilde{b}_t^{l,2}N_t^l+\tilde{b}_t^{l,3}\alpha_t^l)\mathrm{d}t\\
			\qquad\qquad +(\tilde{\sigma}_t^{l,0}+\tilde{\sigma}_t^{l,1}\tilde{X}_t^l+\tilde{\sigma}_t^{l,2}\alpha_t^l)\mathrm{d}W_t^l,\\
			\mathrm{d}N_t^l=(\tilde{g}_t^{l,0}+\tilde{g}_t^{l,1}\tilde{X}_t^l+\tilde{g}_t^{l,2}N_t^l+\tilde{g}_t^{l,3}\alpha_t^l)\mathrm{d}t+L_t^f\mathrm{d}W_t^l,\\
			\mathrm{d}\tilde{y}_t=-(\tilde{b}_t^{l,1\top}\tilde{y}_t+\tilde{\sigma}_t^{l,1\top}\tilde{z}_t+\tilde{g}_t^{l,1\top}\psi_t+\tilde{Q}_t^l\tilde{X}_t^l)\mathrm{d}t+\tilde{z}_t\mathrm{d}W_t^l,\\
			\mathrm{d}\psi_t^l=-(\tilde{b}_t^{l,2\top}\tilde{y}_t+\tilde{g}_t^{l,2\top}\psi_t^l)\mathrm{d}t,\\
			\tilde{X}_0^l=\tilde{x}_0^l,\quad N_T^l=\tilde{G}\tilde{X}_T^l,\quad \tilde{y}_T=\tilde{G}^l\tilde{X}_T^l-\tilde{G}^\top\psi_T,\quad \psi_0=0,
		\end{cases}
	\end{equation}
	admits a unique solution, then the leader's problem admits an optimal control $\hat{\alpha}^l$. Denote the solution of the corresponding forward-backward system (\ref{eq34}) by $(\hat{X}^l,\allowbreak\hat{N}^l,\allowbreak\hat{L}^f,\hat{y},\hat{z},\hat{\psi})$. It satisfies
	\begin{equation}\label{eq35}
		R_t^l\hat{\alpha}_t^l+\tilde{b}_t^{l,3\top}\hat{y}_t+\tilde{\sigma}_t^{l,2\top}\hat{z}_t+\hat{g}_t^{l,3\top}\hat{\psi}_t=0.
	\end{equation}
\end{prp}

\begin{proof}
	Let $\hat{\alpha}^l$ be an optimal control of the leader's problem, and let $\beta\in\mathcal{U}^l$ be an admissible control such that $\hat{\alpha}+\beta\in\mathcal{U}^l$. Under the controls $\hat{\alpha}$ and $\hat{\alpha}+\varepsilon\beta$, denote the solutions of the forward-backward system (\ref{eq34}) by $(\hat{X}^l,\hat{N}^l,\hat{L}^f,\hat{y},\hat{z},\hat{\psi})$ and $(\hat{X}^\varepsilon,\hat{N}^\varepsilon,\hat{L}^\varepsilon,\hat{y}^\varepsilon,\hat{z}^\varepsilon,\hat{\psi}^\varepsilon)$, respectively. Then $(\hat{X}^\varepsilon-\hat{X}^l,\hat{N}^\varepsilon-\hat{N}^l,\hat{L}^\varepsilon-\hat{L}^f,\hat{y}^\varepsilon-\hat{y},\hat{z}^\varepsilon-\hat{z},\hat{\psi}^\varepsilon-\hat{\psi})$ satisfies the corresponding homogeneous forward-backward system
	\begin{equation*}
		\begin{cases}
			\mathrm{d}\tilde{X}_t^0=(\tilde{b}_t^{l,1}\tilde{X}_t^0+\tilde{b}_t^{l,2}N_t^0+\tilde{b}_t^{l,3}\varepsilon\beta_t)\mathrm{d}t+(\tilde{\sigma}_t^{l,1}\tilde{X}_t^0+\tilde{\sigma}_t^{l,2}\varepsilon\beta_t)\mathrm{d}W_t^l,\\
			\mathrm{d}N_t^0=(\tilde{g}_t^{l,1}\tilde{X}_t^0+\tilde{g}_t^{l,2}N_t^0+\tilde{g}_t^{l,3}\varepsilon\beta_t)\mathrm{d}t+L_t^0\mathrm{d}W_t^l,\\
			\mathrm{d}\tilde{y}^0_t=-(\tilde{b}_t^{l,1\top}\tilde{y}^0_t+\tilde{\sigma}_t^{l,1\top}\tilde{z}^0_t+\tilde{g}_t^{l,1\top}\psi^0_t+\tilde{Q}_t^l\tilde{X}_t^0)\mathrm{d}t+\tilde{z}^0_t\mathrm{d}W_t^l,\\
			\mathrm{d}\psi_t^0=-(\tilde{b}_t^{l,2\top}\tilde{y}^0_t+\tilde{g}_t^{l,2\top}\psi_t^0)\mathrm{d}t,\\
			\tilde{X}_0^0=0,\quad N_T^l=\tilde{G}\tilde{X}_T^0,\quad \tilde{y}_T=\tilde{G}^l\tilde{X}_T^0-\tilde{G}^\top\psi_T^0,\quad \psi_0^0=0.
		\end{cases}
	\end{equation*}
	Substituting into the cost functional of the leader and writing $\mathcal{J}(\varepsilon)=\tilde{J}^l(\hat{\alpha}^l+\varepsilon\beta)-\tilde{J}^l(\hat{\alpha}^l)$, we obtain
	\begin{equation*}
    \begin{aligned}
		\mathcal{J}(\varepsilon)&=\frac{1}{2}\varepsilon^2\mathbb{E}\biggl[ \int_0^T\big\{ \tilde{X}_t^{0\top}\tilde{Q}_t^l\tilde{X}_t^0+\beta_t^\top R_t^l\beta_t \big\}\mathrm{d}t+\tilde{X}_T^{0\top}\tilde{G}^l\tilde{X}_T^0 \biggr]\\
		&\qquad +\varepsilon\mathbb{E}\biggl[ \int_0^T \big\{ \tilde{X}_t^{0\top}\tilde{Q}_t^l\hat{X}_t^l+\beta_t^{\top}R_t^l\hat{\alpha}_t^l \big\}\mathrm{d}t+\tilde{X}_T^{0\top}\tilde{G}^l\hat{X}_T^l \biggr]:=\frac{1}{2}\varepsilon^2\mathcal{J}^1+\varepsilon\mathcal{J}^2.
	\end{aligned}
    \end{equation*}
	By Assumption \ref{Ass:4.6}, $\mathcal{J}^1\ge 0$. Since $\hat{\alpha}^l$ is an optimal control, $\mathcal{J}^2=0$ for any $\beta$.
	
	By It\^o's formula,
	\begin{multline*}
		\mathbb{E}[\tilde{X}_T^{0\top}\tilde{y}_T^0]-\mathbb{E}[\tilde{X}_0^{0\top}\tilde{y}_0^0]=\mathbb{E}\biggl[ \int_0^T \big\{ -\tilde{X}_t^{0\top}\tilde{Q}^l\tilde{X}_t^0+\beta_t^\top(\tilde{b}_t^{l,3\top}\tilde{y}_t^0+\tilde{\sigma}_t^{l,2\top}\tilde{z}_t^0)\\
		+N_t^{0\top}\tilde{b}_t^{l,2\top}\tilde{y}_t^0-\tilde{X}_t^{0\top}\tilde{g}_t^{l,1\top}\psi_t^0 \big\}\mathrm{d}t \biggr],
	\end{multline*}
	\begin{equation*}
		\mathbb{E}[N_T^{0\top}\psi_T^0]-\mathbb{E}[N_0^{0\top}\psi_0^0]=\mathbb{E}\biggl[ \int_0^T\big\{ \beta_t^\top\tilde{g}_t^{l,3\top}\psi_t+\tilde{X}_t^{0\top}\tilde{g}_t^{l,1\top}\psi_t^0-N_t^{0\top}\tilde{b}_t^{l,2\top}\tilde{y}_t^0 \big\}\mathrm{d}t \biggr].
	\end{equation*}
	Adding the above two identities yields
	\begin{equation*}
		\mathbb{E}[\tilde{X}_T^{0\top}\tilde{G}^l\tilde{X}_T^0]=\mathbb{E}\biggl[ \int_0^T\big\{ -\tilde{X}_t^{0\top}\tilde{Q}^l\tilde{X}_t^0+\beta_t^\top(\tilde{b}_t^{l,3\top}\tilde{y}_t^0+\tilde{\sigma}_t^{l,2\top}\tilde{z}_t^0+\tilde{g}_t^{l,3\top}\psi_t) \big\}\mathrm{d}t \biggr],
	\end{equation*}
	\begin{equation*}
		\mathcal{J}^1=\mathbb{E}\biggl[ \int_0^T\big\{ \beta_t^\top(R_t^l\beta_t+\tilde{b}_t^{l,3\top}\tilde{y}_t^0+\tilde{\sigma}_t^{l,2\top}\tilde{z}_t^0+\tilde{g}_t^{l,3\top}\psi_t) \big\}\mathrm{d}t \biggr].
	\end{equation*}
	
	Similarly, by It\^o's formula, we obtain
	\begin{equation*}
		\mathbb{E}[\tilde{X}_T^{0\top}\tilde{G}^l\hat{X}_T^l]=\mathbb{E}\biggl[ \int_0^T\big\{ -\tilde{X}_T^{0\top}\tilde{Q}^l\hat{X}_t^l+\beta_t^{l\top}(\tilde{b}_t^{l,3\top}\hat{y}_t+\tilde{\sigma}_t^{l,2\top}\hat{z}_t+\tilde{g}_t^{l,3\top}\hat{\psi}_t) \big\}\mathrm{d}t \biggr],
	\end{equation*}
	\begin{equation*}
		\mathcal{J}^2=\mathbb{E}\biggl[\int_0^T\beta_t^{l\top}\big\{ R_t^l\hat{\alpha}_t^l+\tilde{b}_t^{l,3\top}\hat{y}_t+\tilde{\sigma}_t^{l,2\top}\hat{z}_t+\tilde{g}_t^{l,3\top}\hat{\psi}_t \big\}\mathrm{d}t\biggr].
	\end{equation*}
	Since $\mathcal{J}^2=0$ for any $\beta$, (\ref{eq35}) holds.
\end{proof}

Next, we use a dimension-lifting method to solve (\ref{eq34}) and (\ref{eq35}). Write
\begin{equation*}
	\bar{\bar{X}}=\begin{pmatrix}
		\tilde{X} \\ -\psi
	\end{pmatrix},\quad \bar{\bar{Y}}=\begin{pmatrix}
		\tilde{y} \\ N^l
	\end{pmatrix},\quad \bar{\bar{Z}}=\begin{pmatrix}
		\tilde{z} \\ L^f
	\end{pmatrix}.
\end{equation*}
Then (\ref{eq34}) and (\ref{eq35}) are equivalent to the constrained forward-backward system
\begin{equation*}
	\begin{cases}
		\mathrm{d}\bar{\bar{X}}_t=\bar{\bar{b}}_t(\bar{\bar{X}}_t,\bar{\bar{Y}}_t,\hat{\alpha}_t^l)\mathrm{d}t+\bar{\bar{\sigma}}_t(\bar{\bar{X}}_t,\hat{\alpha}_t^l)\mathrm{d}W_t^l,\\
		\mathrm{d}\bar{\bar{Y}}_t=\bar{\bar{g}}_t(\bar{\bar{X}}_t,\bar{\bar{Y}}_t,\bar{\bar{Z}}_t,\hat{\alpha}_t^l)\mathrm{d}t+\bar{\bar{Z}}_t\mathrm{d}W_t^l,\\
		\bar{\bar{X}}_0=\bar{\bar{x}}_0,\quad \bar{\bar{Y}}_T=\bar{\bar{G}}\bar{\bar{X}}_T,\\
		R_t^l\hat{\alpha}_t^l+\bar{\bar{h}}_t^1\bar{\bar{X}}_t+\bar{\bar{h}}_t^2\bar{\bar{Y}}_t+\bar{\bar{h}}_t^{3}\bar{\bar{Z}}_t=0,
	\end{cases}
\end{equation*}
where
\begin{equation*}
	\begin{cases}
		\bar{\bar{b}}_t(x,y,\alpha)=\bar{\bar{b}}_t^{0}+\bar{\bar{b}}_t^1 x+\bar{\bar{b}}_t^2 y+\bar{\bar{b}}_t^3\alpha,\\
		\bar{\bar{b}}_t^0=\begin{pmatrix}
			\tilde{b}_t^{l,0} \\ 0
		\end{pmatrix},\quad \bar{\bar{b}}_t^1=\begin{pmatrix}
			\tilde{b}_t^{l,1} & 0 \\
			0 & -\tilde{g}_t^{l,2\top}
		\end{pmatrix},\quad \bar{\bar{b}}_t^2=\begin{pmatrix}
			0 & \tilde{b}_t^{l,2} \\
			\tilde{b}_t^{l,2\top} & 0
		\end{pmatrix},\quad \bar{\bar{b}}_t^3=\begin{pmatrix}
			\tilde{b}_t^{l,3} \\ 0
		\end{pmatrix},\\
		\bar{\bar{\sigma}}_t(x,\alpha)=\bar{\bar{\sigma}}_t^{0}+\bar{\bar{\sigma}}_t^1 x+\bar{\bar{\sigma}}_t^2\alpha,\\
		\bar{\bar{\sigma}}_t^0=\begin{pmatrix}
			\tilde{\sigma}_t^{l,0} \\ 0
		\end{pmatrix},\quad \bar{\bar{\sigma}}_t^1=\begin{pmatrix}
			\tilde{\sigma}_t^{l,1} & 0 \\ 0 & 0
		\end{pmatrix},\quad \bar{\bar{\sigma}}_t^2=\begin{pmatrix}
			\tilde{\sigma}_t^{l,2} \\ 0
		\end{pmatrix},\\
		\bar{\bar{g}}_t(x,y,z,\alpha)=\bar{\bar{g}}_t^0+\bar{\bar{g}}_t^1x+\bar{\bar{g}}_t^2y+\bar{\bar{g}}_t^3z+\bar{\bar{g}}_t^4\alpha,\\
		\bar{\bar{g}}_t^0=\begin{pmatrix}
			0 \\ \tilde{g}_t^{l,0}
		\end{pmatrix},\quad \bar{\bar{g}}_t^1=\begin{pmatrix}
			-\tilde{Q}_t^l & \tilde{g}_t^{l,1\top} \\
			\tilde{g}_t^{l,1} & 0
		\end{pmatrix},\quad \bar{\bar{g}}_t^2=-\bar{\bar{b}}_t^{1\top},\quad \bar{\bar{g}}_t^3=-\bar{\bar{\sigma}}_t^{1\top},\quad \bar{\bar{g}}_t^4=\begin{pmatrix}
			0 \\ \tilde{g}_t^{l,3}
		\end{pmatrix},\\
		\bar{\bar{x}}_0=\begin{pmatrix}
			\tilde{x}_0^l \\ 0
		\end{pmatrix},\quad \bar{\bar{G}}=\begin{pmatrix}
			\tilde{G}^l & \tilde{G}^\top \\
			\tilde{G} & 0
		\end{pmatrix},\quad \bar{\bar{h}}_t^1=-\bar{\bar{g}}_t^{4\top},\quad \bar{\bar{h}}_t^2=\bar{\bar{b}}_t^{3\top},\quad \bar{\bar{h}}_t^3=\bar{\bar{\sigma}}_t^{2\top}.
	\end{cases}
\end{equation*}

Assume that there exist deterministic processes $\bar{\bar{P}},\bar{\bar{\varphi}}$ such that
\begin{equation*}
	\bar{\bar{Y}}_t=\bar{\bar{P}}_t\bar{\bar{X}}_t+\bar{\bar{\varphi}}_t.
\end{equation*}
Substituting this into the forward-backward system and comparing the coefficients of $\mathrm{d}W_t^l$ gives
\begin{equation*}
	\bar{\bar{Z}}_t=\bar{\bar{P}}_t(\bar{\bar{\sigma}}_t^0+\bar{\bar{\sigma}}_t^1\bar{\bar{X}}_t+\bar{\bar{\sigma}}_t^2\hat{\alpha}_t^l).
\end{equation*}
Substituting the expressions of $Y,Z$ into the maximum condition and setting
\begin{equation*}
	\begin{cases}
		\bar{\bar{R}}_t=R_t^l+\bar{\bar{\sigma}}_t^{2\top}\bar{\bar{P}}_t\bar{\bar{\sigma}}_t^2,\\
		\bar{\bar{S}}_t=-\bar{\bar{g}}_t^{4\top}+\bar{\bar{b}}_t^{3\top}\bar{\bar{P}}_t+\bar{\bar{\sigma}}_t^{2\top}\bar{\bar{P}}_t\bar{\bar{\sigma}}_t^1,
	\end{cases}
\end{equation*}
we obtain
\begin{equation*}
	\hat{\alpha}_t^l=-\bar{\bar{R}}_t^{-1}( \bar{\bar{S}}_t\bar{\bar{X}}_t+\bar{\bar{b}}_t^{3\top}\bar{\bar{\varphi}}_t+\bar{\bar{\sigma}}_t^{2\top}\bar{\bar{P}}_t\bar{\bar{\sigma}}_t^0 ),
\end{equation*}
\begin{equation*}
	\bar{\bar{Z}}_t=\bar{\bar{P}}_t\big\{(\bar{\bar{\sigma}}_t^1-\bar{\bar{\sigma}}_t^2\bar{\bar{R}}_t^{-1}\bar{\bar{S}}_t)\bar{\bar{X}}_t-\bar{\bar{\sigma}}_t^2\bar{\bar{R}}_t^{-1}\bar{\bar{b}}_t^{3\top}\bar{\bar{\varphi}}_t+(I-\bar{\bar{\sigma}}_t^2\bar{\bar{R}}_t^{-1}\bar{\bar{\sigma}}_t^{2\top}\bar{\bar{P}}_t)\bar{\bar{\sigma}}_t^0\big\}.
\end{equation*}
Substituting the expressions of $Y,Z,\alpha$ into the forward-backward system and comparing the coefficients of the $\mathrm{d}t$ terms, we obtain the Riccati equation and the backward {\it ordinary differential equation} (ODE)
\begin{equation}\label{Eq:leader_riccati2}
	\begin{cases}
		\mathrm{d}\bar{\bar{P}}_t+\big\{ \bar{\bar{P}}_t\bar{\bar{b}}_t^1+\bar{\bar{b}}_t^{1\top}\bar{\bar{P}}_t+\bar{\bar{\sigma}}_t^{1\top}\bar{\bar{P}}_t\bar{\bar{\sigma}}_t^1+\bar{\bar{P}}_t\bar{\bar{b}}_t^2\bar{\bar{P}}_t-\bar{\bar{S}}_t^{l\top}\bar{\bar{R}}_t^{-1}\bar{\bar{S}}_t-\bar{\bar{g}}_t^1 \big\}\mathrm{d}t=0,\\
		\bar{\bar{P}}_T=\bar{\bar{G}},
	\end{cases}
\end{equation}
\begin{equation}\label{Eq:leader_BODE}
	\begin{cases}
		\mathrm{d}\bar{\bar{\varphi}}_t=\big\{(-\bar{\bar{b}}_t^{1\top}-\bar{\bar{P}}_t\bar{\bar{b}}_t^2+\bar{\bar{S}}_t^{\top}\bar{\bar{R}}_t^{-1}\bar{\bar{b}}_t^{3\top})\bar{\bar{\varphi}}_t\\
		\qquad\qquad +(\bar{\bar{g}}_t^0-\bar{\bar{P}}_t\bar{\bar{b}}_t^0-\bar{\bar{\sigma}}_t^{1\top}\bar{\bar{P}}_t\bar{\bar{\sigma}}_t^0+\bar{\bar{S}}_t^{\top}\bar{\bar{\sigma}}_t^2\bar{\bar{R}}_t^{-1}\bar{\bar{\sigma}}_t^{2\top}\bar{\bar{P}}_t\bar{\bar{\sigma}}_t^0)\big\}\mathrm{d}t,\\
		\bar{\bar{\varphi}}_T=0.
	\end{cases}
\end{equation}

Combining the above derivations, we obtain the following result.

\begin{thm}[Conclusion for the leader problem]\label{Thm:4.3}
	Assume that Assumptions \ref{Ass:LQ}, \ref{Ass:03}, \ref{Ass:06}, and \ref{Ass:4.6} hold, and that $\bar{\bar{R}}\gg 0$. If the nonsymmetric Riccati equation (\ref{Eq:leader_riccati1}) admits a solution $\hat{P}^f\in L^\infty(0,T;\mathbb{R}^{n_1\times n_1})$, and the Riccati equation (\ref{Eq:leader_riccati2}) admits a unique solution $\bar{\bar{P}}\in L^{\infty}(0,T;\mathcal{S}^{2n_1+n_2})$, then the leader problem admits an optimal control $\hat{\alpha}^l$ satisfying
	\begin{equation}\label{eq36}
		\hat{\alpha}_t^l=-\bar{\bar{R}}_t^{-1}( \bar{\bar{S}}_t\bar{\bar{X}}_t+\bar{\bar{b}}_t^{3\top}\bar{\bar{\varphi}}_t+\bar{\bar{\sigma}}_t^{2\top}\bar{\bar{P}}_t\bar{\bar{\sigma}}_t^0 ),
	\end{equation}
	where $\bar{\bar{\varphi}}$ is the solution of the ODE (\ref{Eq:leader_BODE}), and $\bar{\bar{X}}\in\mathbb{S}_{\mathcal{F}^l}^2(\mathbb{R}^{2n_1+n_2})$ is the solution of the SDE
	\begin{equation}\label{Eq:leader_SDE}
		\begin{cases}
			\mathrm{d}\bar{\bar{X}}_t=\big\{(\bar{\bar{b}}_t^1+\bar{\bar{b}}_t^2\bar{\bar{P}}_t-\bar{\bar{b}}_t^3\bar{\bar{R}}_t^{-1}\bar{\bar{S}}_t)\bar{\bar{X}}_t+(\bar{\bar{b}}_t^2-\bar{\bar{b}}_t^3\bar{\bar{R}}_t^{-1}\bar{\bar{b}}_t^{3\top})\bar{\bar{\varphi}}_t\\
			\qquad\qquad +(\bar{\bar{b}}_t^0-\bar{\bar{b}}_t^3\bar{\bar{R}}_t^{-1}\bar{\bar{\sigma}}_t^{2\top}\bar{\bar{P}}_t\bar{\bar{\sigma}}_t^0)\big\}\mathrm{d}t+\big\{(\bar{\bar{\sigma}}_t^1-\bar{\bar{\sigma}}_t^2\bar{\bar{R}}_t^{-1}\bar{\bar{S}}_t)\bar{\bar{X}}_t\\
			\qquad\qquad -\bar{\bar{\sigma}}_t^2\bar{\bar{R}}_t^{-1}\bar{\bar{b}}_t^{3\top}\bar{\bar{\varphi}}_t+(\bar{\bar{\sigma}}_t^0-\bar{\bar{\sigma}}_t^2\bar{\bar{R}}_t^{-1}\bar{\bar{\sigma}}_t^{2\top}\bar{\bar{P}}_t\bar{\bar{\sigma}}_t^0)\big\}\mathrm{d}W_t^l,\\
			\bar{\bar{X}}_0=\bar{\bar{x}}_0.
		\end{cases}
	\end{equation}
\end{thm}

By the form of $\bar{\bar{X}}$, write
\begin{equation*}
	\hat{\alpha}_t^l=\hat{\alpha}_t^l(\hat{X}_t^{l},\hat{M}_t^f,\hat{\psi}_t),
\end{equation*}
where $\hat{\alpha}_t^l(x,m,\psi)$ is a linear function of $x,m,\psi$ determined by the maximum condition (\ref{eq36}).

\subsection{Stackelberg-Nash equilibrium}

\begin{thm}[Stackelberg-Nash equilibrium]
	Assume that Assumptions \ref{Ass:LQ}, \ref{Ass:03}, \ref{Ass:06}, and \ref{Ass:4.6} hold, $\hat{R}^f\gg 0$ and $\bar{\bar{R}}\gg 0$, the Riccati equations (\ref{Eq:follower_riccati}) and (\ref{Eq:leader_riccati2}) admit unique solutions $P^f\in L^\infty(0,T;\mathcal{S}^{n_1})$ and $\bar{\bar{P}}\in L^\infty(0,T;\mathcal{S}^{2n_1+n_2})$, respectively, the nonsymmetric Riccati equation (\ref{Eq:leader_riccati1}) admits a unique solution $\hat{P}^f\in L^\infty(0,T;\mathbb{R}^{n_1\times n_1})$, $\bar{\bar{\varphi}}\in L^2(0,T;\mathbb{R}^{2n_1+n_2})$ is the unique solution of the ODE (\ref{Eq:leader_BODE}), $\bar{\bar{X}}\in\mathbb{S}_{\mathcal{F}^l}^2(\mathbb{R}^{2n_1+n_2})$ is the unique solution of the SDE (\ref{Eq:leader_SDE}), and the forward-backward system (\ref{Eq:follower_consistent}) admits a unique solution $(\hat{z},\hat{\varphi},\hat{q})\in\mathbb{S}_{\boxtimes^l}^{\infty,2}(\mathbb{R}^{n_1})\times\mathbb{S}_{\boxtimes^l}^{\infty,2}(\mathbb{R}^{n_1})\times\mathbb{H}_{\boxtimes^l}^{\infty,2}(\mathbb{R}^{n_1})$ under the optimal pair of the leader
	\begin{equation*}
		\hat{\alpha}_t^l=-\bar{\bar{R}}_t^{-1}( \bar{\bar{S}}_t\bar{\bar{X}}_t+\bar{\bar{b}}_t^{3\top}\bar{\bar{\varphi}}_t+\bar{\bar{\sigma}}_t^{2\top}\bar{\bar{P}}_t\bar{\bar{\sigma}}_t^0 ),\quad \hat{X}_t^l=\begin{pmatrix}
			I^{n_2} \\ \mathbf{0}^{2n_1+n_2}
		\end{pmatrix}\bar{\bar{X}}_t.
	\end{equation*}
	Then the problem admits a unique Stackelberg-Nash equilibrium
	\begin{equation}
		\begin{cases}
			\hat{\alpha}_t^{f,u}=\hat{\alpha}^f_t(\hat{X}_t^{f,u},z_t^u,\varphi_t^{f,u},X_t^l,\alpha_t^l),\\
			\hat{\alpha}_t^l=\hat{\alpha}_t^l(\hat{X}_t^{l},\hat{M}_t^f,\hat{y}_t).
		\end{cases}
	\end{equation}
\end{thm}

\begin{proof}
	It follows from Theorems \ref{Thm:4.2} and \ref{Thm:4.3}.
\end{proof} 
	
\section{Propagation of chaos}\label{Sec:05}

\subsection{Finite-player game}

Consider a leader-follower game with one leader and $N$ followers:

\begin{equation}\label{Eq:N_follower_sys}
	\begin{cases}
		\mathrm{d}X_t^{N,i}=b_t^f(X_t^{N,i},\alpha_t^{N,i},\upsilon_t^{N,i},X_t^{N,0},\alpha_t^{N,0})\mathrm{d}t\\
		\qquad\qquad +\sigma_t^f(X_t^{N,i},\alpha_t^{N,i},\upsilon_t^{N,i},X_t^{N,0},\alpha_t^{N,0})\mathrm{d}W_t^{f,\frac{i}{N}},\\
		\mathrm{d}X_t^{N,0}=b_t^l(X_t^{N,0},\alpha_t^{N,0},\upsilon_t^{N,0})\mathrm{d}t+\sigma_t^l(X_t^{N,0},\alpha_t^{N,0},\upsilon_t^{N,0})\mathrm{d}W_t^l,\\
		X_0^{N,i}=x_0^{N,i},\quad X_0^{N,0}=x_0^{l},
	\end{cases}
\end{equation}
where
\begin{equation*}
	\upsilon_t^{N,i}=\frac{1}{N}\sum_{j=1}^N g^{N,i,j}X_t^{N,j},\quad i=1,\cdots,N,\quad \upsilon_t^{N,0}=\frac{1}{N}\sum_{i=1}^N w^{N,i}X_t^{N,i}.
\end{equation*}

The cost functional of follower $i$ is
\begin{multline}\label{eq40}
	J^{N,i}(\alpha^{N,i};\alpha^{N,-i},\alpha^{N,0})\\
	=\mathbb{E}\biggl[ \int_0^T h_t^{f}(X_t^{N,i},\alpha_t^{N,i},\upsilon_t^{N,i},X_t^{N,0},\alpha_t^{N,0})\mathrm{d}t+g^{f}(X_T^{N,i},\upsilon_T^{N,i},X_T^{N,0}) \biggr],
\end{multline}
and the cost functional of the leader is
\begin{equation}
	J^{N,0}(\alpha^{N,0};\alpha^N)=\mathbb{E}\biggl[ \int_0^T h_t^l(X_t^{N,0},\alpha_t^{N,0},\upsilon_t^{N,0})\mathrm{d}t+g^l(X_T^{N,0},\upsilon_T^{N,0}) \biggr].
\end{equation}

Define two classes of admissible control sets as follows.

\begin{dfn}
	Let $\mathcal{F}^{N,i}$ be the completed natural filtration generated by the Brownian motion $W^{f,\frac{i}{N}}$, let $\mathcal{F}^0$ be the completed natural filtration generated by $W^l$, and let $\mathcal{F}^N=\mathcal{F}^0\lor\mathcal{F}^{N,1}\lor\cdots\lor\mathcal{F}^{N,N}$ be the completed natural filtration generated by all Brownian motions in the $N$-follower problem. Define
	
	(1) centralized strategy profiles
	\begin{equation*}
		\mathcal{U}_c^{N,0}=L_{\mathcal{F}^N}^2(0,T;\mathbb{R}^{m_2}),\quad \mathcal{U}_c^{N,i}=L_{\mathcal{F}^N}^2(0,T;\mathbb{R}^{m_1}),\quad i=1,\cdots,N,
	\end{equation*}
	
	(2) decentralized strategy profiles
	\begin{equation*}
		\mathcal{U}_d^{N,0}=L_{\mathcal{F}^0}^2(0,T;\mathbb{R}^{m_2}),\quad \mathcal{U}_d^{N,i}=L_{\mathcal{F}^{N,i}\lor\mathcal{F}^0}^2(0,T;\mathbb{R}^{m_1}),\quad i=1,\cdots,N.
	\end{equation*}
\end{dfn}

One can directly use classical results to solve for the exact centralized strategy, but the computational complexity is very high and may suffer from the curse of dimensionality. Therefore, we use the graphon mean field approximation method to provide decentralized approximate equilibrium strategies for the players.

\begin{dfn}
	Let $\hat{\alpha}^N$ be a measurable function from $\mathcal{U}_d^{N,0}$ to $\mathcal{U}_d^N$. We say that $(\hat{\alpha}^{N,0},\hat{\alpha}^N(\hat{\alpha}^{N,0}))$ is a decentralized $(\varepsilon_1,\varepsilon_2)$-Stackelberg-Nash equilibrium of the finite-player game if
	
	(1) for any $\alpha^{N,0}\in\mathcal{U}_d^{N,0}$, $\hat{\alpha}^N(\alpha^{N,0})$ is an $\varepsilon_2$-Nash equilibrium of the follower problem, that is, for any $i=1,\cdots,N$ and $\beta\in\mathcal{U}_d^{N,i}$,
	\begin{equation*}
		J^{N,i}(\beta;\hat{\alpha}^N(\alpha^{N,0})^{-i},\alpha^{N,0})\ge J^{N,i}(\hat{\alpha}^N(\alpha^{N,0})^i;\hat{\alpha}^N(\alpha^{N,0})^{-i},\alpha^{N,0})-\varepsilon_2,
	\end{equation*}
	
	(2) $\hat{\alpha}^{N,0}\in\mathcal{U}_d^{N,0}$ is an $\varepsilon_1$-optimal control of the leader problem when the approximate follower response $\hat{\alpha}^N$ is fixed, that is, for any $\beta\in\mathcal{U}_d^{N,0}$,
	\begin{equation*}
		J^{N,0}(\beta;\hat{\alpha}^N(\beta))\ge J^{N,0}(\hat{\alpha}^{N,0};\hat{\alpha}^N(\hat{\alpha}^{N,0}))-\varepsilon_1.
	\end{equation*}
\end{dfn}

\subsection{Propagation of chaos}

For brevity, the problem with $N$ followers is called the $N$-problem, and the auxiliary problem constructed below is called the $N$-auxiliary problem. The continuum follower problem is called the limiting problem.

Write
\begin{gather*}
	G_N=\sum_{i,j=1}^N g^{N,i,j}\mathbf{1}_{\left(\frac{i-1}{N},\frac{i}{N}\right]\times\left(\frac{j-1}{N},\frac{j}{N}\right]},\\
	x_0^{[N]}=\sum_{i=1}^N x_0^{N,i}\mathbf{1}_{\left(\frac{i-1}{N},\frac{i}{N}\right]},\quad w^{[N]}=\sum_{i=1}^N w^{N,i}\mathbf{1}_{\left( \frac{i-1}{N},\frac{i}{N} \right]}.
\end{gather*}

\begin{Ass}[The corresponding limiting problem is solvable]\label{Ass:POC1}
	In the limiting problem, for any $\alpha^l\in\mathcal{U}^l$, the follower problem admits a Nash equilibrium $\hat{\alpha}_\cdot^f(\cdot,\cdot,\cdot,\cdot,\cdot)$, and the leader problem admits an optimal control $\hat{\alpha}_\cdot^l(\cdot,\cdot,\cdot)$.
\end{Ass}

\begin{Ass}[Convergence of graphons and initial states]\label{Ass:POC2}
	The graphons $G,G_N$ satisfy
	\begin{equation*}
		\lim_{N\to\infty}N\max_{i=1,\cdots,N}\int_{\left(\frac{i-1}{N},\frac{i}{N}\right]}\|G(u,\cdot)-G_N(u,\cdot)\|_1\lambda(\mathrm{d}u)=0,
	\end{equation*}
	and the initial states $x_0,x_0^{[N]}$ and the leader weights $w,w^{[N]}$ satisfy
	\begin{gather*}
		\lim_{N\to\infty} N\max_{i=1,\cdots,N}\int_{\left(\frac{i-1}{N},\frac{i}{N}\right]}|x_0^{f,u}-x_0^{[N],u}|\lambda(\mathrm{d}u)=0,\\
		\lim_{N\to\infty}N\max_{i=1,\cdots,N}\int_{\left(\frac{i-1}{N},\frac{i}{N}\right]}|w^u-w^{[N],u}|\lambda(\mathrm{d}u)=0.
	\end{gather*}
\end{Ass}

\begin{remark}
	Regarding notations, when discussing the approximate Nash equilibrium of the followers below, in the $N$-problem and the limiting problem, for a given $\alpha^{N,0}=\alpha^l$, we denote the follower states, aggregates, and leader state in the $N$-problem under the constructed strategy profile by $\hat{X}^{N},\hat{\upsilon}^N,\hat{X}^{N,0}$ respectively, and denote the follower states, aggregates, and leader state in the limiting problem under the follower Nash equilibrium by $\hat{X}^f,\hat{z},\hat{X}^l$ respectively,. When discussing the approximate Stackelberg equilibrium of the leader, under the optimal leader control $\hat{\alpha}^l$ in the limiting problem, the corresponding optimal follower state and mean state are denoted by $(\check{X}^f,\check{M}^f)$, and the optimal leader state is denoted by $\check{X}^l$.
\end{remark}

\begin{Ass}\label{Ass:POC3}
	There exists a constant $K>0$ such that, for any positive integer $N$, $i=1,\cdots,N$, $\alpha^N\in\mathcal{U}_d^{N}$, and $\alpha^0\in\mathcal{U}_d^{N,0}$, we have
	\begin{equation*}
		J^{N,i}(\alpha^{N,i};\alpha^{N,-i},\alpha^0)\ge K\mathbb{E}\biggl[ \int_0^T |\alpha_t^{N,i}|^2\mathrm{d}t \biggr]-K.
	\end{equation*}
\end{Ass}

When $Q\ge 0$, $R\gg 0$, and $G\ge 0$, Assumption \ref{Ass:POC3} holds automatically.

\begin{thm}[Propagation of chaos]\label{Thm:POC}
	Assume that Assumptions \ref{Ass:POC1}, \ref{Ass:POC2} and \ref{Ass:POC3} hold. For any $\alpha^{l}\in\mathcal{U}^l$, let the graphon aggregate, adjoint process, and leader state of the limiting problem under the follower Nash equilibrium be $\big(\hat{z}(\alpha^l),\hat{\varphi}(\alpha^l),\hat{X}^{l}(\alpha^l)\big)$. For any $t\in[0,T]$, define
	\begin{equation*}
		\bar{z}_t^{N,i}(\alpha^l)=N\int_{\left(\frac{i-1}{N},\frac{i}{N}\right]}\hat{z}_t^u(\alpha^l)\lambda(\mathrm{d}u),\quad \bar{\varphi}_t^{N,i}(\alpha^l)=N\int_{\left(\frac{i-1}{N},\frac{i}{N}\right]}\hat{\varphi}_t^u(\alpha^l)\lambda(\mathrm{d}u),
	\end{equation*}
	\begin{equation*}
		\hat{\alpha}_t^{N,i}(\alpha^l)=\hat{\alpha}^f_t\big(\hat{X}_t^{N,i},\bar{\varphi}_t^{N,i}(\alpha^l),\bar{z}_t^{N,i}(\alpha^l),\hat{X}_t^{l}(\alpha^l),\alpha_t^{l}\big),
	\end{equation*}
	\begin{equation*}
		\hat{\alpha}^N(\alpha^l)=\big(\hat{\alpha}^{N,1}(\alpha^l),\cdots,\hat{\alpha}^{N,N}(\alpha^l)\big),\quad \hat{\alpha}^{N,0}=\hat{\alpha}_t^l(\check{X}_t^l,\check{M}_t^f,\check{\psi}).
	\end{equation*}
	Then, for any $\varepsilon_1,\varepsilon_2>0$, there exists a positive integer $N_0$ such that, for any $N\ge N_0$, $\big( \hat{\alpha}^{N,0},\hat{\alpha}^N(\hat{\alpha}^{N,0}) \big)$ is a decentralized $(\varepsilon_1,\varepsilon_2)$-Stackelberg-Nash equilibrium of the $N$-problem.
\end{thm}

This theorem will be proved in the next subsection.

\subsection{Proof of the propagation of chaos}

Fix $\alpha^{N,0}=\alpha^l$. For simplicity, write $\bar{z}_t^{N,i}(\alpha^l),\bar{\varphi}_t^{N,i}(\alpha^l)$ as $\bar{z}_t^{N,i},\bar{\varphi}_t^{N,i}$, and define $\bar{X}_t^{N,i},\bar{q}_t^{N,i}$ similarly.

Set
\begin{multline*}
	K_N=N\max_{i=1,\cdots,N}\int_{\left(\frac{i-1}{N},\frac{i}{N}\right]}\|G(u,\cdot)-G_N(u,\cdot)\|_1\lambda(\mathrm{d}u)\\
	+N\max_{i=1,\cdots,N}\int_{\left(\frac{i-1}{N},\frac{i}{N}\right]}\big\{|x_0^{f,u}-x_0^{[N],u}|+|w^u-w^{[N],u}|\big\}\lambda(\mathrm{d}u)+\frac{1}{\sqrt{N}}.
\end{multline*}
Assumption \ref{Ass:POC2} is equivalent to $\lim_{N\to\infty}K_N=0$.

\begin{prp}[Uniform boundedness estimate for interaction parameters in the finite problem]\label{Prp:5.1}
	Assume that Assumption \ref{Ass:POC2} holds. Then there exists a constant $K$ independent of $N$ such that
	
	(1) for any positive integer $N$,
	\begin{gather}
		\max_{i=1,\cdots,N}\biggl\{ \frac{1}{N}\sum_{j=1}^N g^{N,i,j}-N\int_{\left(\frac{i-1}{N},\frac{i}{N}\right]}\|G(u,\cdot)\|_1\lambda(\mathrm{d}u) \biggr\}\le K_N,\\
		\max_{i=1,\cdots,N}\frac{1}{N}\sum_{j=1}^N g^{N,i,j}=\sup_{u\in I}\|G_N(u,\cdot)\|_1\le K_N+1\le K,\\
		\biggl|\frac{1}{N}\sum_{i=1}^N w^{N,i}-1\biggr|\le \int_I |w^{[N],u}-w^u|\lambda(\mathrm{d}u)\le K_N,\\
		\max_{i=1,\cdots,N}w^{N,i}\le\sup_{u\in I}w^u+K_N\le K,\label{eq30}
	\end{gather}
	
	(2) for any $X\in L_{\mathcal{F}\boxtimes\mathcal{I}}^{\infty,2}(\mathbb{R}^n)$,
	\begin{gather*}
		\max_{i=1,\cdots,N}\mathbb{E}\biggl[ \biggl( N\int_{\left(\frac{i-1}{N},\frac{i}{N}\right]}(G_N-G)X^u\lambda(\mathrm{d}u) \biggr)^2 \biggr]\le KK_N^2\sup_{u\in I}\mathbb{E}[|X^u|^2],\\
		\mathbb{E}\biggl[\biggl( \int_I (w^{[N],u}-w^u)X^u\lambda(\mathrm{d}u) \biggr)^2\biggr]\le KK_N^2\sup_{u\in I}\mathbb{E}[|X^u|^2].
	\end{gather*}
\end{prp}

\begin{proof}
	The result follows from the definition of $K_N$ and Proposition \ref{Prp:2.1}, (2).
\end{proof}

\begin{prp}[Uniform boundedness estimates for the averaged states, adjoints, and aggregates of the limiting problem]\label{Prp:5.2}
	Assume that Assumptions \ref{Ass:POC1} and \ref{Ass:POC2} hold. There exists a constant $K>0$ independent of $N$ such that
	\begin{gather}
		\sup_{N,i}\mathbb{E}\biggl[ \sup_{0\le t\le T}\big\{|\bar{X}_t^{N,i}|^2+|\bar{z}_t^{N,i}|^2+|\bar{\varphi}_t^{N,i}|^2\big\}+\int_0^T|\bar{q}_t^{N,i}|^2\mathrm{d}t \biggr]\le K,\\
		\mathbb{E}\biggl[ \sup_{0\le t\le T}\big\{ |\hat{X}_t^l|^2+|\hat{M}_t^f|^2 \big\} \biggr]\le K.
	\end{gather}
\end{prp}

\begin{proof}
	By definition, for any positive integer $N$, $i=1,\cdots,N$, and $t\in[0,T]$,
	\begin{equation*}
		|\bar{X}_t^{N,i}|\le\sup_{u\in I}|\hat{X}_t^{f,u}|^2,\quad |\bar{z}_t^{N,i}|^2\le\sup_{u\in I}|\hat{z}_t^u|^2,\quad |\bar{\varphi}_t^{N,i}|\le \sup_{u\in I}|\hat{\varphi}_t^{u}|^2,\quad |\bar{q}_t^{N,i}|\le \sup_{u\in I}|\hat{q}_t^{u}|^2.
	\end{equation*}
	Substituting the optimal control of the limiting problem and using Theorem \ref{Thm:3.1}, we have $\hat{X}^f\in\mathbb{S}_{\boxtimes^l}^{\infty,2}(\mathbb{R}^{n_1})$ and $(\hat{z},\hat{\varphi},\hat{q})\in\mathbb{S}_{\boxtimes^l}^{\infty,2}(\mathbb{R}^{n_1})\times\mathbb{S}_{\boxtimes^l}^{\infty,2}(\mathbb{R}^{n_1})\times\mathbb{H}_{\boxtimes^l}^{\infty,2}(\mathbb{R}^{n_1})$, which further yields the conclusion.
\end{proof}

Substituting the strategy profile constructed in Theorem \ref{Thm:POC} into the state equation (\ref{Eq:N_follower_sys}), and noting that
\begin{gather*}
	b^f\big(x^f,\hat{\alpha}^f(x^f,z,\varphi^f,x^l,\alpha^l),z,x^l,\alpha^l\big)=\hat{b}^f(x^f,\varphi^f,z,x^l,\alpha^l),\\
	\sigma^f\big(x^f,\hat{\alpha}^f(x^f,z,\varphi^f,x^l,\alpha^l),z,x^l,\alpha^l\big)=\hat{\sigma}^f(x^f,\varphi^f,z,x^l,\alpha^l),
\end{gather*}
we obtain that the corresponding state $\hat{X}^N$ satisfies
\begin{equation*}
	\begin{cases}
		\mathrm{d}\hat{X}_t^{N,i}=\big\{\hat{b}_t^f(\hat{X}_t^{N,i},\bar{\varphi}_t^{N,i},\bar{z}_t^{N,i},\hat{X}_t^{N,0},\alpha_t^l)+b_t^{f,3}(\hat{\upsilon}_t^{N,i}-\bar{z}_t^{N,i})\big\}\mathrm{d}t\\
		\qquad\qquad +\big\{ \hat{\sigma}_t^f(\hat{X}_t^{N,i},\bar{\varphi}_t^{N,i},\bar{z}_t^{N,i},\hat{X}_t^{N,0},\alpha_t^l)+\sigma_t^{f,3}(\hat{\upsilon}_t^{N,i}-\bar{z}_t^{N,i}) \big\}\mathrm{d}W_t^{f,\frac{i}{N}},\\
		X_0^{N,i}=x_0^{N,i}.
	\end{cases}
\end{equation*}
From the definition of $\hat{\upsilon}^{N,0}$,
\begin{equation*}
	\begin{cases}
		\mathrm{d}\hat{\upsilon}_t^{N,0}=\big\{ \frac{1}{N}\sum_{j=1}^N w^{N,j}(\hat{b}_t^{f,0}+\hat{b}_t^{f,4}\hat{X}_t^{N,0}+\hat{b}_t^{f,5}\alpha_t^l)+\hat{b}_t^{f,1}\hat{\upsilon}_t^{N,0}\\
		\qquad\qquad +\frac{1}{N}\sum_{j=1}^N w^{N,j}\big( \hat{b}_t^{f,2}\bar{\varphi}_t^{N,j}+\hat{b}_t^{f,3}\bar{z}_t^{N,j}+b_t^{f,3}(\hat{\upsilon}_t^{N,j}-\bar{z}_t^{N,j}) \big) \big\}\mathrm{d}t\\
		\qquad\qquad +\frac{1}{N}\sum_{j=1}^N w^{N,j}\big( \hat{\sigma}_t^{f,0}+\hat{\sigma}_t^{f,1}\hat{X}_t^{N,j}+\hat{\sigma}_t^{f,2}\bar{\varphi}_t^{N,j}+\hat{\sigma}_t^{f,3}\bar{z}_t^{N,j}\\
		\qquad\qquad +\hat{\sigma}_t^{f,4}\hat{X}_t^{N,0}+\hat{\sigma}_t^{f,5}\alpha_t^l+\sigma_t^{f,3}(\hat{\upsilon}_t^{N,j}-\bar{z}_t^{N,j}) \big)\mathrm{d}W_t^{f,\frac{j}{N}},\\
		\hat{\upsilon}_0^{N,0}=\frac{1}{N}\sum_{j=1}^N w^{N,j}x_0^{N,j}.
	\end{cases}
\end{equation*}
By the CELLN, the expression of $\hat{M}_t^f$ is
\begin{equation*}
	\begin{cases}
		\mathrm{d}\hat{M}_t^f=\hat{b}_t^f\big(\hat{M}_t^f,\int_I w^u\hat{\varphi}_t^u\lambda(\mathrm{d}u),\int_I w^u\hat{z}_t^u\lambda(\mathrm{d}u),\hat{X}_t^l,\alpha_t^l\big)\mathrm{d}t,\\
		\hat{M}_0^f=\int_I w^ux_0^{f,u}\lambda(\mathrm{d}u).
	\end{cases}
\end{equation*}
Set $\varDelta_t^{N,0}=\hat{\upsilon}_t^{N,0}-\hat{M}_t^f$. Then $\varDelta^{N,0}$ satisfies
\begin{equation*}
	\begin{cases}
		\mathrm{d}\varDelta_t^{N,0}=\big\{ \big( \frac{1}{N}\sum_{j=1}^N w^{N,j}-1 \big)(\hat{b}_t^{f,0}+\hat{b}_t^{f,5}\alpha_t^l)+\hat{b}_t^{f,4}\big( \frac{1}{N}\sum_{j=1}^N w^{N,j}\hat{X}_t^{N,0}-\hat{X}_t^l\big)\\
		\qquad\qquad+\hat{b}_t^{f,1}\varDelta_t^{N,0}+\hat{b}_t^{f,2}\big( \frac{1}{N}\sum_{j=1}^N w^{N,j}\bar{\varphi}_t^{N,j}-\int_I w^u\hat{\varphi}_t^u\lambda(\mathrm{d}u) \big)\\
		\qquad\qquad+\hat{b}_t^{f,3}\big(\frac{1}{N}\sum_{j=1}^N w^{N,j}\bar{z}_t^{N,j}-\int_I w^u\hat{z}_t^u\lambda(\mathrm{d}u)\big)\\
		\qquad\qquad+b_t^{f,3}\frac{1}{N}\sum_{j=1}^N w^{N,j}(\hat{\upsilon}_t^{N,j}-\bar{z}_t^{N,j}) \big\}\mathrm{d}t\\
		\qquad\qquad +\frac{1}{N}\sum_{j=1}^N w^{N,j}\big\{ \hat{\sigma}_t^{f,0}+\hat{\sigma}_t^{f,1}\hat{X}_t^{N,j}+\hat{\sigma}_t^{f,2}\bar{\varphi}_t^{N,j}+\hat{\sigma}_t^{f,3}\bar{z}_t^{N,j}\\
		\qquad\qquad +\hat{\sigma}_t^{f,4}\hat{X}_t^{N,0}+\hat{\sigma}_t^{f,5}\alpha_t^l+\sigma_t^{f,3}(\hat{\upsilon}_t^{N,j}-\bar{z}_t^{N,j}) \big\}\mathrm{d}W_t^{f,\frac{j}{N}},\\
		\varDelta_0^{N,0}=\frac{1}{N}\sum_{j=1}^N w^{N,j}x_0^{N,j}-\int_I w^ux_0^{f,u}\lambda(\mathrm{d}u).
	\end{cases}
\end{equation*}

\begin{prp}\label{Prp:5.3}
	Assume that Assumptions \ref{Ass:POC1} and \ref{Ass:POC2} hold. For any $t\in[0,T]$, the estimate
	\begin{equation*}
		\mathbb{E}\biggl[ \sup_{0\le s\le t}\big\{ |\hat{\upsilon}_s^{N,0}-\hat{M}_s^f|^2+|\hat{X}_s^{N,0}-\hat{X}_s^l|^2 \big\} \biggr]\le K\max_{i=1,\cdots,N}\mathbb{E}\biggl[ \sup_{0\le s\le t}|\hat{\upsilon}_s^{N,i}-\bar{z}_s^{N,i}|^2 \biggr]+KK_N^2
	\end{equation*}
	holds.
\end{prp}

\begin{proof}
	Note that
	\begin{equation*}
		\frac{1}{N}\sum_{j=1}^N w^{N,j}\bar{\varphi}_t^{N,j}=\int_I w^{[N],u}\hat{\varphi}_t^u\lambda(\mathrm{d}u),\qquad \frac{1}{N}\sum_{j=1}^N w^{N,j}\bar{z}_t^{N,j}=\int_I w^{[N],u}\hat{z}_t^u\lambda(\mathrm{d}u).
	\end{equation*}
	By Proposition \ref{Prp:5.1}, we obtain
	\begin{gather*}
		\mathbb{E}\biggl[\int_0^T \biggl( \frac{1}{N}\sum_{j=1}^N w^{N,j}-1 \biggr)^2(\hat{b}_t^{f,0}+\hat{b}_t^{f,5}\alpha_t^l)^2\mathrm{d}t\biggr]\le KK_N^2,\\
		\mathbb{E}\biggl[ \int_0^t \biggl| b_s^{f,3}\frac{1}{N}\sum_{j=1}^N w^{N,j}(\hat{\upsilon}_s^{N,j}-\bar{z}_s^{N,j}) \biggr|^2\mathrm{d}s \biggr]\le K\max_{i=1,\cdots,N}\mathbb{E}\biggl[\sup_{0\le s\le t}|\hat{\upsilon}_s^{N,i}-\bar{z}_s^{N,i}|^2\biggr].
	\end{gather*}
	Using SDE estimates gives
	\begin{equation*}
    \begin{aligned}
		\mathbb{E}\biggl[ \sup_{0\le s\le t}|\hat{\upsilon}_s^{N,0}-\hat{M}_s^f|^2 \biggr]&\le K\mathbb{E}\biggl[ \sup_{0\le s\le t}|\hat{X}_s^{N,0}-\hat{X}_s^l|^2 \biggr]\\
		&\qquad +K\max_{i=1,\cdots,N}\mathbb{E}\biggl[ \sup_{0\le s\le t}|\hat{\upsilon}_s^{N,j}-\bar{z}_s^{N,j}|^2 \biggr]+KK_N^2.
	\end{aligned}
    \end{equation*}
	Applying the classical SDE estimate and Gronwall inequality to the leader state equation yields the desired estimate.
\end{proof}

\begin{prp}[Uniform boundedness rough estimate]\label{Prp:5.4}
	Assume that Assumptions \ref{Ass:POC1} and \ref{Ass:POC2} hold. For any $t\in[0,T]$, the estimate
	\begin{equation*}
    \begin{aligned}
		&\max_{i=1,\cdots,N}\mathbb{E}\biggl[ \sup_{0\le s\le t}\big\{|\hat{X}_s^{N,i}|^2+|\hat{\upsilon}_s^{N,i}|^2+|\hat{\alpha}_s^{N,i}|^2+|\hat{X}_s^{N,0}|^2+|\hat{\upsilon}_s^{N,0}|^2\big\} \biggr]\\
		&\le K\max_{i=1,\cdots,N}\mathbb{E}\biggl[ \sup_{0\le s\le t}|\hat{\upsilon}_s^{N,i}-\bar{z}_s^{N,i}|^2 \biggr]+K
	\end{aligned}
    \end{equation*}
	holds.
\end{prp}

\begin{proof}
	By Proposition \ref{Prp:5.2},
	\begin{equation*}
		\max_{i=1,\cdots,N}\mathbb{E}\biggl[ \sup_{0\le s\le t}|\hat{\upsilon}_s^{N,i}|^2 \biggr]\le \max_{i=1,\cdots,N}\mathbb{E}\biggl[ \sup_{0\le s\le t}|\hat{\upsilon}_s^{N,i}-\bar{z}_s^{N,i}|^2 \biggr]+K.
	\end{equation*}
	By Propositions \ref{Prp:5.2} and \ref{Prp:5.3},
	\begin{equation*}
		\max_{i=1,\cdots,N}\mathbb{E}\biggl[ \sup_{0\le s\le t}\big\{ |\hat{X}_s^{N,0}|^2+|\hat{\upsilon}_s^{N,0}|^2 \big\} \biggr]\le K\max_{i=1,\cdots,N}\mathbb{E}\biggl[ \sup_{0\le s\le t}|\hat{\upsilon}_s^{N,i}-\bar{z}_s^{N,i}|^2 \biggr]+K.
	\end{equation*}
	From the definition of $\hat{\alpha}^{N,i}$,
	\begin{equation*}
    \begin{aligned}
		&\max_{i=1,\cdots,N}\mathbb{E}\biggl[ \sup_{0\le s\le t}|\hat{\alpha}_s^{N,i}|^2 \biggr]
\le K\max_{i=1,\cdots,N}\mathbb{E}\biggl[ \sup_{0\le s\le t}|\hat{X}_s^{N,i}|^2 \biggr]\\
        &\qquad + K\max_{i=1,\cdots,N}\mathbb{E}\biggl[ \sup_{0\le s\le t}|\hat{\upsilon}_s^{N,i}-\bar{z}_s^{N,i}|^2 \biggr]+K.
	\end{aligned}
    \end{equation*}
	Substituting this into the equation of $\hat{X}^{N,i}$ and using the classical SDE estimate and Gronwall inequality,
	\begin{equation*}
		\max_{i=1,\cdots,N}\mathbb{E}\biggl[ \sup_{0\le s\le t}|\hat{X}_s^{N,i}|^2 \biggr]\le K\max_{i=1,\cdots,N}\mathbb{E}\biggl[ \sup_{0\le s\le t}|\hat{\upsilon}_s^{N,i}-\bar{z}_s^{N,i}|^2 \biggr]+K,
	\end{equation*}
	and the result follows.
\end{proof}

From the definition of $\hat{\upsilon}^N$, we have
\begin{equation*}
	\begin{cases}
		\mathrm{d}\hat{\upsilon}_t^{N,i}=\big\{\frac{1}{N}\sum_{j=1}^N g^{N,i,j}(\hat{b}_t^{f,0}+\hat{b}_t^{f,4}\hat{X}_t^{N,0}+\hat{b}_t^{f,5}\alpha_t^{l})+\hat{b}_t^{f,1}\hat{\upsilon}_t^{N,i}\\
		\qquad\qquad +\frac{1}{N}\sum_{j=1}^Ng^{N,i,j}\big(\hat{b}_t^{f,2}\bar{\varphi}_t^{N,j}+\hat{b}_t^{f,3}\bar{z}_t^{N,j}+b_t^{f,3}(\hat{\upsilon}_t^{N,j}-\bar{z}_t^{N,j})\big)\big\}\mathrm{d}t\\
		\qquad\qquad +\frac{1}{N}\sum_{j=1}^N g^{N,i,j}\big(\hat{\sigma}_t^{f,0}+\hat{\sigma}_t^{f,1}\hat{X}_t^{N,j}+\hat{\sigma}_t^{f,2}\bar{\varphi}_t^{N,j}+\hat{\sigma}_t^{f,3}\bar{z}_t^{N,j}\\
		\qquad\qquad +\hat{\sigma}_t^{f,4}\hat{X}_t^{N,0}+\hat{\sigma}_t^{f,5}\alpha_t^l+\sigma_t^{f,3}(\hat{\upsilon}_t^{N,j}-\bar{z}_t^{N,j})\big)\mathrm{d}W_t^{f,\frac{j}{N}},\\
		\hat{\upsilon}_0^{N,i}=\frac{1}{N}\sum_{j=1}^N g^{N,i,j}x_0^{N,j}.
	\end{cases}
\end{equation*}
By (\ref{Eq:follower_consistent}), $(\bar{z}^{N},\bar{\varphi}^N,\bar{q}^{N})$ satisfies
\begin{equation*}
	\begin{cases}
		\mathrm{d}\bar{z}_t^{N,i}=\big\{\hat{b}_t^{f,1}\bar{z}_t^{N,i}+\hat{b}_t^{f,2}N\int_{\left(\frac{i-1}{N},\frac{i}{N}\right]}G\hat{\varphi}_t^{u}\lambda(\mathrm{d}u)+\hat{b}_t^{f,3}N\int_{\left(\frac{i-1}{N},\frac{i}{N}\right]}G\hat{z}_t^{u}\lambda(\mathrm{d}u)\\
		\qquad\qquad\quad  +N\int_{\left(\frac{i-1}{N},\frac{i}{N}\right]}\|G(u,\cdot)\|_1\lambda(\mathrm{d}u)(\hat{b}_t^{f,0}+\hat{b}_t^{f,4}\hat{X}_t^l+\hat{b}_t^{f,5}\hat{\alpha}_t^l)\big\}\mathrm{d}t,\\
		\mathrm{d}\bar{\varphi}_t^{N,i}=\big\{\hat{g}_t^{f,1}\bar{\varphi}_t^{N,i}+\hat{g}_t^{f,2}\bar{z}_t^{N,i}+(\hat{g}_t^{f,0}+\hat{g}_t^{f,3}\hat{X}_t^l+\hat{g}_t^{f,4}\hat{\alpha}_t^l)\big\}\mathrm{d}t+\bar{q}_t^{N,i}\mathrm{d}W_t^l,\\
		\bar{z}_0^{N,i}=N\int_{\left(\frac{i-1}{N},\frac{i}{N}\right]} Gx_0^{f,u}\lambda(\mathrm{d}u),\quad \bar{\varphi}_T^{N,i}=-G^f\big(g^{f,1}\bar{z}_T^{N,i}+g^{f,2}\hat{X}_T^l\big).
	\end{cases}
\end{equation*}
Set $\varDelta^N=\hat{\upsilon}^N-\bar{z}^N$. Then $\varDelta^N$ satisfies
\begin{equation}\label{eq31}
	\begin{cases}
		\mathrm{d}\varDelta^N_t=(A_t\varDelta_t^N+B_t)\mathrm{d}t+\varSigma_t\mathrm{d}W_t^N,\\
		\varDelta_0^N=\delta_0^N,
	\end{cases}
\end{equation}
where the $N\times N$ and $N\times 1$ block coefficient matrices are
\begin{equation*}
	\begin{cases}
		A_t=\hat{b}_t^{f,1}I_N+\frac{1}{N}A_t^{[N]},\quad A_t^{[N]}=\big(A_t^{[N],i,j}\big)_{N\times N},\quad A_t^{[N],i,j}=g^{N,i,j}b_t^{f,3},\\
		B_t=B_t^0+B_t^1+B_t^2+B_t^3+B_t^4,\\
		B_t^{0,i}=\hat{b}_t^{f,0}\big(\frac{1}{N}\sum_{j=1}^N g^{N,i,j}-N\int_{\left(\frac{i-1}{N},\frac{i}{N}\right]}\|G(u,\cdot)\|_1\lambda(\mathrm{d}u)\big),\\
		B_t^{1,i}=\hat{b}_t^{f,2}\big( \frac{1}{N}\sum_{j=1}^N g^{N,i,j}\bar{\varphi}_t^{N,j}-N\int_{\left(\frac{i-1}{N},\frac{i}{N}\right]}G\hat{\varphi}_t^u\lambda(\mathrm{d}u) \big),\\
		B_t^{2,i}=\hat{b}_t^{f,3}\big( \frac{1}{N}\sum_{j=1}^N g^{N,i,j}\bar{z}_t^{N,j}-N\int_{\left(\frac{i-1}{N},\frac{i}{N}\right]}G\hat{z}_t^u\lambda(\mathrm{d}u) \big),\\
		B_t^{3,i}=\hat{b}_t^{f,4}\big( \frac{1}{N}\sum_{j=1}^Ng^{N,i,j}\hat{X}_t^{N,0}-N\int_{\left(\frac{i-1}{N},\frac{i}{N}\right]}\|G(u,\cdot)\|_1\lambda(\mathrm{d}u)\hat{X}_t^l \big),\\
		B_t^{4,i}=\hat{b}_t^{f,5}\big( \frac{1}{N}\sum_{j=1}^N g^{N,i,j}\hat{\alpha}_t^{N,0}-N\int_{\left(\frac{i-1}{N},\frac{i}{N}\right]}\|G(u,\cdot)\|_1\lambda(\mathrm{d}u)\hat{\alpha}_t^l \big),\\
		\varSigma_t=\big(\varSigma_t^{i,j}\big)_{N\times N},\\
		\varSigma_t^{i,j}=\frac{1}{N}g^{N,i,j}\big(\hat{\sigma}_t^{f,0}+\hat{\sigma}_t^{f,1}\hat{X}_t^{N,j}+\hat{\sigma}_t^{f,2}\bar{\varphi}_t^{N,j}+\hat{\sigma}_t^{f,3}\bar{z}_t^{N,j}+\hat{\sigma}_t^{f,4}\hat{X}_t^{N,0}+\hat{\sigma}_t^{f,5}\alpha_t^l\\
		\qquad\qquad+\sigma_t^{f,3}(\hat{\upsilon}_t^{N,j}-\bar{z}_t^{N,j})\big),\\
		\delta_0^{N,i}=\frac{1}{N}\sum_{j=1}^N g^{N,i,j}x_0^{N,j}-N\int_{\left(\frac{i-1}{N},\frac{i}{N}\right]}Gx_0^{f,u}\lambda(\mathrm{d}u).
	\end{cases}
\end{equation*}

Take the auxiliary process
\begin{equation*}
	\begin{cases}
		\mathrm{d}\varPhi_t=-\varPhi_t\big(\hat{b}_t^{f,1}I_N+\frac{1}{N}A_t^{[N]}\big)\mathrm{d}t,\\
		\varPhi_0=I.
	\end{cases}
\end{equation*}
Then the linear SDE (\ref{eq31}) has the representation
\begin{equation}\label{eq32}
	\varDelta_t^N=\varPhi_t^{-1}\varDelta_0^N+\varPhi_t^{-1}\int_0^t \varPhi_sB_s\mathrm{d}s+\varPhi_t^{-1}\int_0^t \varPhi_s\varSigma_s\mathrm{d}W_s^N,
\end{equation}
and the solution of the auxiliary process is
\begin{equation*}
	\varPhi_t=\exp\biggl\{-\int_0^t\hat{b}_s^{f,1}\mathrm{d}s\biggr\}\exp\biggl\{ -\frac{1}{N}\int_0^t A_s^{[N]}\mathrm{d}s \biggr\}.
\end{equation*}

\begin{prp}
	There exists a constant $K>0$ independent of $N$ such that, for any positive integer $N$ and $t\in[0,T]$,
	\begin{equation*}
		\biggl\| \exp\biggl\{ -\frac{1}{N}\int_0^t A_s^{[N]}\mathrm{d}s \biggr\} \biggr\|_{\infty}\le K,\quad \biggl\| \exp\biggl\{ \frac{1}{N}\int_0^t A_s^{[N]}\mathrm{d}s \biggr\} \biggr\|_{\infty}\le K.
	\end{equation*}
	Furthermore, there exists a constant $K>0$ independent of $N$ such that, for any positive integer $N$ and $t\in [0,T]$,
	\begin{equation*}
		\|\varPhi_t\|_\infty\le K,\quad \|\varPhi_t^{-1}\|_\infty\le K.
	\end{equation*}
\end{prp}

\begin{proof}
	The result follows directly from Lemma 4.11 in \cite{Xu-Gou-Huang-Gao-2025}.
\end{proof}

\begin{prp}[Coefficient estimate]\label{Prp:5.6}
	Assume that Assumptions \ref{Ass:POC1} and \ref{Ass:POC2} hold. For any $t\in[0,T]$, the estimate
	\begin{equation*}
		\mathbb{E}\biggl[ \int_0^t \|B_s\|_\infty^2\mathrm{d}s \biggr]\le KK_N^2+\int_0^t \max_{i=1,\cdots,N}\mathbb{E}\biggl[ \sup_{0\le r\le s}|\hat{\upsilon}_r^{N,i}-\bar{z}_r^{N,i}|^2 \biggr]\mathrm{d}s
	\end{equation*}
	holds.
\end{prp}

\begin{proof}
	By the definitions of the step graphon and the averaged states, for any $i=1,\dots,N$,
	\begin{equation*}
		B_t^{1,i}=\hat{b}_t^{f,2}\biggl( N\int_{\left(\frac{i-1}{N},\frac{i}{N}\right]}(G_N\varphi_t^{[N],u}-G\hat{\varphi}_t^u)\lambda(\mathrm{d}u) \biggr)=\hat{b}_t^{f,2}\biggl(N\int_{\left(\frac{i-1}{N},\frac{i}{N}\right]}(G_N-G)\hat{\varphi}_t^u\lambda(\mathrm{d}u)\biggr).
	\end{equation*}
	By Proposition \ref{Prp:5.1}, (2),
	\begin{equation*}
    \begin{aligned} 
		&\mathbb{E}\big[\|B_t^1\|_\infty^2\big]=\max_{i=1,\cdots,N}\mathbb{E}\big[|B_t^{1,i}|^2\big]
\le K\max_{i=1,\cdots,N}\mathbb{E}\biggl[ \biggl( N\int_{\left(\frac{i-1}{N},\frac{i}{N}\right]}(G_N-G)\hat{\varphi}_t^u\lambda(\mathrm{d}u) \biggr)^2 \biggr]\\
		&\le KK_N^2\sup_{u\in I}\mathbb{E}[|\hat{\varphi}_t^u|^2]\le KK_N^2.
	\end{aligned}
    \end{equation*}
	Similarly,
	\begin{equation*}
		\|B_t^2\|_\infty^2\le KK_N^2\sup_{u\in I}\mathbb{E}[|\hat{z}_t^u|^2]\le KK_N^2.
	\end{equation*}
	By Proposition \ref{Prp:5.1}, (1),
	\begin{equation*}
		\|B_t^0\|_\infty^2\le KK_N^2,\quad \mathbb{E}[\|B_t^4\|_\infty^2]\le KK_N^2,
	\end{equation*}
	\begin{equation*}
    \begin{aligned} 
		&\mathbb{E}\big[\|B_t^3\|_\infty^2\big]=\max_{i=1,\cdots,N}\mathbb{E}\big[|B_t^{3,i}|^2\big]\le K\biggl( \biggl( \frac{1}{N}\sum_{j=1}^N g^{N,i,j}(\hat{X}_t^{N,0}-\hat{X}_t^l) \biggr)^2+K_N^2\mathbb{E}[|\hat{X}_t^l|^2] \biggr)\\
		&\le K\mathbb{E}\big[|\hat{X}_t^{N,0}-\hat{X}_t^l|^2\big]+KK_N^2.
	\end{aligned}
    \end{equation*}
	Therefore,
	\begin{equation*}
    \begin{aligned} 
		&\mathbb{E}\biggl[ \int_0^t \|B_s\|_\infty^2\mathrm{d}s \biggr]\le KK_N^2+\mathbb{E}\biggl[ \int_0^t|\hat{X}_s^{N,0}-\hat{X}_s^l|^2\mathrm{d}s \biggr]\\
		&\le KK_N^2+\int_0^t \max_{i=1,\cdots,N}\mathbb{E}\biggl[ \sup_{0\le r\le s}|\hat{\upsilon}_r^{N,i}-\bar{z}_r^{N,i}|^2 \biggr]\mathrm{d}s.\qedhere
	\end{aligned}
    \end{equation*}
\end{proof}

\begin{prp}\label{Prp:5.7}
	Assume that Assumptions \ref{Ass:POC1} and \ref{Ass:POC2} hold. The estimate
	\begin{equation}\label{eq37}
		\max_{i=1,\cdots,N}\mathbb{E}\biggl[ \sup_{0\le t\le T}|\hat{\upsilon}_t^{N,i}-\bar{z}_t^{N,i}|^2 \biggr]\le KK_N^2
	\end{equation}
	holds.
\end{prp}

\begin{proof}
	By (\ref{eq32}),
	\begin{equation*}
		(\varDelta_t^N\varDelta_t^{N\top})_{ii}\le 3\biggl( (\varPhi_t^{-1}\varDelta_0^N)_i^2+\biggl( \varPhi_t^{-1}\int_0^t \varPhi_sB_s\mathrm{d}s \biggr)_i^2+\biggl( \varPhi_t^{-1}\int_0^t \varPhi_s\varSigma_s\mathrm{d}W_s^N \biggr)_i^2 \biggr).
	\end{equation*}
	We estimate the three terms on the right-hand side separately.
	
	Estimate of the first term: by the definition of $K_N$,
	\begin{equation*}
    \begin{aligned} 
		&\mathbb{E}[\|\varDelta_0^N\|_\infty^2]=\mathbb{E}\biggl[ \max_{i=1,\cdots,N}|\varDelta_0^{N,i}|^2 \biggr]\\
		&=\mathbb{E}\biggl[ \max_{i=1,\cdots,N}\biggl| N\int_{\left(\frac{i-1}{N},\frac{i}{N}\right]}\big\{ G_N(x_0^{[N],u}-x_0^{f,u})+(G_N-G)x_0^{f,u} \big\}\lambda(\mathrm{d}u) \biggr|^2 \biggr]\\
		&\le KK_N^2+KK_N^2\sup_{u\in I}\mathbb{E}[|x_0^{f,u}|^2]\le KK_N^2,
	\end{aligned}
    \end{equation*}
	\begin{equation*}
		\max_{i=1,\cdots,N}\mathbb{E}[(\varPhi_t^{-1}\varDelta_0^N)_i^2]\le\mathbb{E}[\|\varPhi_t^{-1}\|_\infty^2\|\varDelta_0^N\|_\infty^2]\le KK_N^2.
	\end{equation*}
	
	Estimate of the second term: by Proposition \ref{Prp:5.6},
	\begin{multline*}
		\max_{i=1,\cdots,N}\mathbb{E}\biggl[\biggl( \varPhi_t^{-1}\int_0^t \varPhi_s B_s\mathrm{d}s \biggr)_i^2\biggr]\le \mathbb{E}\biggl[\biggl( \int_0^t \|\varPhi_t^{-1}\varPhi_s\|_\infty\cdot\|B_s\|_\infty\mathrm{d}s \biggr)^2\biggr]\\
		\le K\mathbb{E}\biggl[\biggl( \int_0^t \|B_s\|_\infty\mathrm{d}s \biggr)^2\biggr]\le KK_N^2+K\int_0^t \max_{i=1,\cdots,N}\mathbb{E}\biggl[ \sup_{0\le r\le s}|\hat{\upsilon}_r^{N,i}-\bar{z}_r^{N,i}|^2 \biggr]\mathrm{d}s.
	\end{multline*}
	
	Estimate of the third term: by Propositions \ref{Prp:5.2} and \ref{Prp:5.4},
	\begin{equation*}
    \begin{aligned} 
		&\max_{i,j=1,\cdots,N}\mathbb{E}\biggl[ \sup_{0\le s\le t} |\varSigma_s^{i,j}|^2 \biggr]\\
		&\le \frac{K}{N^2}\biggl\{1+\max_{i=1,\cdots,N}\mathbb{E}\biggl[\sup_{0\le s\le t}\big\{ |\hat{X}_s^{N,i}|^2+|\bar{\varphi}_s^{N,i}|^2+|\bar{z}_t^{N,i}|^2+|\hat{\upsilon}_s^{N,i}|^2+|\hat{X}_s^{N,0}|^2+|\alpha_t^l|^2 \big\}\biggr]\biggr\}\\
		&\le \frac{K}{N^2}+\frac{K}{N^2}\max_{i=1,\cdots,N}\mathbb{E}\biggl[ \sup_{0\le s\le t}|\hat{\upsilon}_s^{N,i}-\bar{z}_s^{N,i}|^2 \biggr].
	\end{aligned}
    \end{equation*}
	
	By It\^o's isometry,
	\begin{multline*}
		\mathbb{E}\biggl[ \biggl| \varPhi_t^{-1}\int_0^t\varPhi_s\varSigma_s\mathrm{d}W_s^N \biggr|_i^2 \biggr]=\mathbb{E}\biggl[ \biggl| \int_0^t(\varPhi_t^{-1}\varPhi_s\varSigma_s)_{i:}\mathrm{d}W_s^N \biggr|^2 \biggr]=\mathbb{E}\biggl[ \int_0^t\sum_{j=1}^N (\varPhi_t^{-1}\varPhi_s\varSigma_s)_{ij}^2\mathrm{d}s \biggr]\\
		\le \mathbb{E}\biggl[ \int_0^t \sum_{j=1}^N\|\varPhi_t^{-1}\varPhi_s\|_\infty^2\max_{i,j=1,\cdots,N}|\varSigma_s^{i,j}|^2\mathrm{d}s \biggr]\le KN\int_0^t\mathbb{E}\biggl[ \max_{i,j=1,\cdots,N}|\varSigma_s^{i,j}|^2\mathrm{d}s \biggr]\\
		\le KN\int_0^t\max_{i,j=1,\cdots,N}\mathbb{E}\biggl[ \sup_{0\le r\le s}|\varSigma_r^{i,j}|^2 \biggr]\mathrm{d}s\le \frac{K}{N}+\frac{K}{N}\int_0^t \max_{i=1,\cdots,N}\mathbb{E}\biggl[ \sup_{0\le r\le s}|\hat{\upsilon}_r^{N,i}-\bar{z}_r^{N,i}|^2 \biggr]\mathrm{d}s.
	\end{multline*}
	Combining the estimates,
	\begin{multline*}
		\max_{i=1,\cdots,N}\mathbb{E}\biggl[\sup_{0\le s\le t}|\hat{\upsilon}_s^{N,i}-\bar{z}_s^{N,i}|^2\biggr]=\max_{i=1,\cdots,N}\mathbb{E}\biggl[\sup_{0\le s\le t}(\varDelta_s^N\varDelta_s^{N\top})_{ii}\biggr]\\
		\le KK_N^2+K\int_0^t \max_{i=1,\cdots,N}\mathbb{E}\biggl[ \sup_{0\le r\le s}|\hat{\upsilon}_r^{N,i}-\bar{z}_r^{N,i}|^2 \biggr]\mathrm{d}s.
	\end{multline*}
	Gronwall inequality yields estimate (\ref{eq37}).
\end{proof}

\begin{prp}[Summary of the estimates obtained so far]\label{Prp:5.8}
	Assume that Assumptions \ref{Ass:POC1} and \ref{Ass:POC2} hold. The estimates
	\begin{gather*}
		\max_{i=1,\cdots,N}\mathbb{E}\biggl[ \sup_{0\le t\le T}|\hat{\upsilon}_t^{N,i}-\bar{z}_t^{N,i}|^2 \biggr]\le KK_N^2,\\
		\mathbb{E}\biggl[ \sup_{0\le t\le T}\big\{ |\hat{\upsilon}_t^{N,0}-\hat{M}_t^f|^2+|\hat{X}_t^{N,0}-\hat{X}_t^l|^2 \big\} \biggr]\le KK_N^2,\\
		\max_{i=1,\cdots,N}\mathbb{E}\biggl[ \sup_{0\le t\le T}\big\{|\hat{X}_t^{N,i}|^2+|\hat{\upsilon}_t^{N,i}|^2+|\hat{\alpha}_t^{N,i}|^2+|\hat{X}_t^{N,0}|^2+|\hat{\upsilon}_t^{N,0}|^2\big\} \biggr]\le K
	\end{gather*}
	hold.
\end{prp}

\begin{proof}
	The result follows from Propositions \ref{Prp:5.2}, \ref{Prp:5.3}, and \ref{Prp:5.7}.
\end{proof}

Consider the $N$-auxiliary problem with system equation
\begin{equation}\label{eq38}
	\begin{cases}
		\mathrm{d}X_t^{N,i}=b_t^f(X_t^{N,i},\alpha_t^{N,i},\bar{z}_t^{N,i},X_t^{N,0},\alpha_t^{N,0})\mathrm{d}t\\
		\qquad\qquad +\sigma_t^f(X_t^{N,i},\alpha_t^{N,i},\bar{z}_t^{N,i},X_t^{N,0},\alpha_t^{N,0})\mathrm{d}W_t^{f,\frac{i}{N}},\\
		X_0^{N,i}=x_0^{N,i},
	\end{cases}
\end{equation}
and cost functional of player $i$
\begin{multline}\label{eq39}
	J^{*N,i}(\alpha^{N,i};\alpha^{N,-i},\alpha^{N,0})\\
	=\mathbb{E}\biggl[ \int_0^T h_t^{f}(X_t^{N,i},\alpha_t^{N,i},\bar{z}_t^{N,i},X_t^{N,0},\alpha_t^{N,0})\mathrm{d}t+g^{f}(X_T^{N,i},\bar{z}_T^{N,i},X_T^{N,0}) \biggr].
\end{multline}

Since the $N$-auxiliary problem is a decoupled optimal control problem and the players do not influence one another, applying the maximum principle directly to the $N$ players gives the following result.

\begin{prp}[Optimal control of the $N$-auxiliary problem]
	Assume that Assumption \ref{Ass:POC1} holds. Then the optimal control of player $i$ in the $N$-auxiliary problem is
	\begin{equation*}
		\hat{\alpha}_t^{*N,i}=\hat{\alpha}_t^f(\hat{X}_t^{*N,i},\bar{z}_t^{N,i},\bar{\varphi}_t^{N,i},X_t^l,\alpha_t^l).
	\end{equation*}
\end{prp}

For fixed $\alpha^l$, under $\hat{\alpha}^N$ in the original $N$-problem, denote the corresponding states and aggregate terms by $\hat{X}^{N,i},\hat{\upsilon}^{N,i},\hat{X}^{N,0}$. Under the control $\hat{\alpha}^{*N}$ in the $N$-auxiliary problem with fixed $\alpha^l,X^l$, denote the corresponding states and aggregates by $\hat{X}^{*N,i},\hat{\upsilon}^{*N,i}$.

\begin{prp}\label{Prp:5.10}
	Assume that Assumptions \ref{Ass:POC1} and \ref{Ass:POC2} hold. Under the strategy constructed for the original $N$-problem, the cost functional and the value function of the $N$-auxiliary problem satisfy
	\begin{equation*}
		\max_{i=1,\cdots,N}|J^{N,i}(\hat{\alpha}^{N,i};\hat{\alpha}^{N,-i},\alpha^{N,0})-J^{*N,i}(\hat{\alpha}^{*N,i};\hat{\alpha}^{*N,-i},\alpha^l)|\le KK_N.
	\end{equation*}
\end{prp}

\begin{proof}
	Substituting $\hat{\alpha}^N,\hat{\alpha}^{*N}$ into the state equations (\ref{Eq:N_follower_sys}) and (\ref{eq38}) of the two problems, the classical estimate for SDE solutions yields
	\begin{equation*}
    \begin{aligned}
		&\max_{i=1,\cdots,N}\mathbb{E}\biggl[ \sup_{0\le t\le T}|\hat{X}_t^{N,i}-\hat{X}_t^{*N,i}|^2 \biggr]\\
		&\le K\max_{i=1,\cdots,N}\mathbb{E}\biggl[ \sup_{0\le t\le T}\big\{ |\hat{\upsilon}_t^{N,i}-\bar{z}_t^{N,i}|^2+|\hat{X}_t^{N,0}-X_t^l|^2 \big\} \biggr]\le KK_N^2.
	\end{aligned}
    \end{equation*}
	By the definitions of $\hat{\alpha}^N,\hat{\alpha}^{*N}$ and Proposition \ref{Prp:5.8}, we have
	\begin{equation*}
		\max_{i=1,\cdots,N}\mathbb{E}\biggl[ \sup_{0\le t\le T}|\hat{\alpha}_t^{N,i}-\hat{\alpha}_t^{*N,i}|^2 \biggr]\le KK_N^2.
	\end{equation*}
	By the quadratic form of the functionals, Proposition \ref{Prp:5.8}, and Cauchy's inequality, we obtain
	\begin{gather*}
		|h_t^f(\hat{X}_t^{N,i},\hat{\alpha}_t^{N,i},\hat{\upsilon}_t^{N,i},\hat{X}_t^{N,0},\alpha_t^{N,0})-h_t^f(\hat{X}_t^{*N,i},\hat{\alpha}_t^{*N,i},\bar{z}_t^{N,i},X_t^l,\alpha_t^l)|\le KK_N,\\
		|g^f(\hat{X}_T^{N,i},\hat{\upsilon}_T^{N,i},\hat{X}_T^{N,0})-g^f(\hat{X}_T^{*N,i},\bar{z}_T^{N,i},X_T^l)|\le KK_N.
	\end{gather*}
	Substituting into the definitions of the cost functionals (\ref{eq40}) and (\ref{eq39}) gives the estimate.
\end{proof}

Set
\begin{equation*}
	\bar{\mathcal{U}}_d^{N,i}=\biggl\{ \beta^i\in\mathcal{U}_d^{N,i}:\mathbb{E}\biggl[ \int_0^T |\beta_t^i|^2\mathrm{d}t \biggr]\le\frac{1}{K}J^{N,i}(\hat{\alpha}^{N,i};\hat{\alpha}^{N,-i},\alpha^0)+1 \biggr\}.
\end{equation*}

\begin{prp}\label{Prp:5.11}
	Assume that Assumptions \ref{Ass:POC1}, \ref{Ass:POC2}, and \ref{Ass:POC3} hold. For any $k=1,\cdots,N$ and perturbation $\beta^k\in\bar{\mathcal{U}}_d^{N,k}$, the estimate
	\begin{equation*}
		|J^{N,k}(\beta^k;\hat{\alpha}^{N,-k},\alpha^{N,0})-J^{*N,k}(\beta^k;\hat{\alpha}^{*N,-k},\alpha^l)|\le KK_N
	\end{equation*}
	holds.
\end{prp}

\begin{proof}
	Let $X^{N},X^{N,0},\upsilon^N$ be the states and aggregates of the $N$-problem under the strategy profile $(\beta^i;\hat{\alpha}^{N,-i})$, and let $X^{*N}$ be the state of the $N$-auxiliary problem under the strategy profile $(\beta^i;\hat{\alpha}^{*N,-i})$.
	
	By the definitions of the aggregate term $\upsilon^{N,i}$ and the weighted average term $\upsilon^{N,0}$,
	\begin{multline*}
		\max_{i=1,\cdots,N}\mathbb{E}\biggl[ \sup_{0\le s\le t}\big\{ |\upsilon_s^{N,i}-\hat{\upsilon}_s^{N,i}|^2+|\upsilon_s^{N,0}-\hat{\upsilon}_s^{N,0}|^2 \big\} \biggr]\\
		\le K\max_{i\neq k}\mathbb{E}\biggl[ \sup_{0\le s\le t}|X_s^{N,i}-\hat{X}_s^{N,i}|^2 \biggr]+\frac{K}{N}\mathbb{E}\biggl[\sup_{0\le s\le t}|X_s^{N,k}-\hat{X}_s^{N,k}|^2\biggr].
	\end{multline*}
	Applying the classical estimate for SDE solutions to the state equation (\ref{Eq:N_follower_sys}) for $i\neq k$, and then using Gronwall inequality, yields
	\begin{equation*}
		\max_{i\neq k}\mathbb{E}\biggl[ \sup_{0\le s\le t}|X_s^{N,i}-\hat{X}_s^{N,i}|^2 \biggr]\le \frac{K}{N}\mathbb{E}\biggl[ \sup_{0\le s\le t}|X_s^{N,k}-\hat{X}_s^{N,k}|^2 \biggr]+KK_N^2.
	\end{equation*}
	Applying the classical SDE estimate to the state equation (\ref{Eq:N_follower_sys}) for $X^{N,0}$ gives
	\begin{equation*}
		\mathbb{E}\biggl[ \sup_{0\le s\le t}|X_s^{N,0}-\hat{X}_s^{N,0}|^2 \biggr]\le \frac{K}{N}\mathbb{E}\biggl[ \sup_{0\le s\le t}|X_s^{N,k}-\hat{X}_s^{N,k}|^2 \biggr]+KK_N^2.
	\end{equation*}
	Applying the classical SDE estimate to the state equation (\ref{Eq:N_follower_sys}) for $i=k$, and then using Gronwall inequality and Assumption \ref{Ass:POC3}, gives
	\begin{equation*}
		\mathbb{E}\biggl[ \sup_{0\le s\le t}|X_s^{N,k}-\hat{X}_s^{N,k}|^2 \biggr]\le K\mathbb{E}\biggl[ \int_0^t |\beta^k_s-\hat{\alpha}_s^{N,k}|^2\mathrm{d}s \biggr]+KK_N^2\le K+KK_N^2.
	\end{equation*}
	Combining the preceding conclusions gives
	\begin{equation*}
		\max_{i=1,\cdots,N}\mathbb{E}\biggl[\sup_{0\le s\le t}\big\{ |\upsilon_s^{N,i}-\hat{\upsilon}_s^{N,i}|^2+|\upsilon_s^{N,0}-\hat{\upsilon}_s^{N,0}|^2+|X_s^{N,0}-\hat{X}_s^{N,0}|^2 \big\}\biggr]\le KK_N^2.
	\end{equation*}
	By Proposition \ref{Prp:5.8},
	\begin{equation*}
		\max_{i=1,\cdots,N}\mathbb{E}\biggl[\sup_{0\le s\le t}\big\{ |\upsilon_s^{N,i}-\bar{z}_s^{N,i}|^2+|\upsilon_s^{N,0}-\hat{M}_s^{f}|^2+|X_s^{N,0}-X_s^l|^2 \big\}\biggr]\le KK_N^2.
	\end{equation*}
	By the classical SDE estimate,
	\begin{equation*}
		\mathbb{E}\biggl[ \sup_{0\le t\le T}|X_t^{N,k}-X_t^{*N,k}|^2 \biggr]\le KK_N^2.
	\end{equation*}
	Again by the quadratic form of the functionals, we obtain
	\begin{equation*}
		|J^{N,k}(\beta^k;\hat{\alpha}^{N,-k},\alpha^{N,0})-J^{*N,k}(\beta^k;\hat{\alpha}^{*N,-k},\alpha^l)|\le KK_N. \qedhere
	\end{equation*}
\end{proof}

\begin{prp}\label{Prp:5.12}
	Assume that Assumptions \ref{Ass:POC1}, \ref{Ass:POC2}, and \ref{Ass:POC3} hold. For fixed $\alpha^l$, there exists a positive integer $N_0^f$ such that, for any $N>N_0^f$, the strategy profile constructed for the $N$-problem is an $\varepsilon_2$-Nash equilibrium.
\end{prp}

\begin{proof}
	For any $i=1,\cdots,N$ and perturbation $\beta^i\in\mathcal{U}_d^{N,i}\backslash\bar{\mathcal{U}}_d^{N,i}$,
	\begin{equation*}
		J^{N,i}(\beta^i;\hat{\alpha}^{N,-i},\alpha^0)\ge K\mathbb{E}\biggl[ \int_0^T |\beta_t^i|^2\mathrm{d}t \biggr]-K\ge J^{N,i}(\hat{\alpha}^{N,i};\hat{\alpha}^{N,-i},\alpha^0).
	\end{equation*}
	For any $\beta^i\in\bar{\mathcal{U}}_d^{N,i}$, Propositions \ref{Prp:5.10} and \ref{Prp:5.11} give
	\begin{multline*}
		J^{N,i}(\beta^i;\hat{\alpha}^{N,-i},\alpha^{N,0})\ge J^{*N,i}(\beta^i;\hat{\alpha}^{*N,-i},\alpha^l)-KK_N\\
		\ge J^{*N,i}(\hat{\alpha}^{*N,i};\hat{\alpha}^{*N,i},\alpha^l)-KK_N\ge J^{N,i}(\hat{\alpha}^{N,i};\hat{\alpha}^{N,-i},\alpha^{N,0})-KK_N.
	\end{multline*}
	By Assumption \ref{Ass:POC2}, for any $\varepsilon_2>0$, there exists $N_0^f$ such that, for any $N>N_0^f$,
	\begin{equation*}
		J^{N,i}(\beta^i;\hat{\alpha}^{N,-i},\alpha^{N,0})\ge J^{N,i}(\hat{\alpha}^{N,i};\hat{\alpha}^{N,-i},\alpha^{N,0})-\varepsilon_2.\qedhere
	\end{equation*}
\end{proof}

\begin{prp}
	Assume that Assumptions \ref{Ass:POC1}, \ref{Ass:POC2}, and \ref{Ass:POC3} hold. In the $N$-problem, for any fixed control $\alpha^l$,
	\begin{equation*}
		|J^{N,0}(\alpha^l)-J^l(\alpha^l)|\le KK_N.
	\end{equation*}
\end{prp}

\begin{proof}
	By the quadratic form of the functional,
	\begin{multline*}
		h_t^l(X_t^{N,0},\alpha_t^{N,0},\hat{\upsilon}_t^{N,0})-h_t^l(X_t^l,\alpha_t^l,\hat{M}_t^f)=\frac{1}{2}\big( \|X_t^{N,0}-h_t^{l,1}\hat{\upsilon}_t^{N,0}\|_{Q_t^l}-\|X_t^l-h_t^{l,1}\hat{M}_t^f\|_{Q_t^l} \big)\\
		=\frac{1}{2}\big[ (X_t^{N,0}+X_t^l)-h_t^{l,1}(\hat{\upsilon}_t^{N,0}+\hat{M}_t^f) \big]^\top Q_t^l\big[ (X_t^{N,0}-X_t^l)-h_t^{l,1}(\hat{\upsilon}_t^{N,0}-\hat{M}_t^f) \big].
	\end{multline*}
	By Proposition \ref{Prp:5.8},
	\begin{equation*}
    \begin{aligned} 
		&\mathbb{E}\biggl[\sup_{0\le t\le T}|h_t^l(X_t^{N,0},\alpha_t^{N,0},\hat{\upsilon}_t^{N,0})-h_t^l(X_t^l,\alpha_t^l,\hat{M}_t^f)|^2\biggr]\\
		&\le K\mathbb{E}\biggl[ \sup_{0\le t\le T}\big\{|X_t^{N,0}|^2+|\hat{\upsilon}_t^{N,0}|^2+|X_t^l|^2+|\hat{M}_t^f|^2\big\} \biggr]\\
        &\qquad \times\mathbb{E}\biggl[\sup_{0\le t\le T}\big\{ |X_t^{N,0}-X_t^l|^2+|\hat{\upsilon}_t^{N,0}-\hat{M}_t^f|^2 \big\}\biggr]\le KK_N^2.
	\end{aligned}
    \end{equation*}
	Similarly,
	\begin{equation*}
		\mathbb{E}\biggl[ |g^l(X_T^{N,0},\hat{\upsilon}_T^{N,0})-g^l(X_T^l,\hat{M}_T^l)|^2 \biggr]\le KK_N^2.
	\end{equation*}
	Therefore,
	\begin{equation*}
		|J^{N,0}(\alpha^{N,0})-J^l(\alpha^l)|\le KK_N.\qedhere
	\end{equation*}
\end{proof}

\begin{prp}\label{Prp:5.14}
	Assume that Assumptions \ref{Ass:POC1}, \ref{Ass:POC2}, and \ref{Ass:POC3} hold. There exists a positive integer $N_0^l$ such that, for any $N>N_0^l$, the leader control constructed for the $N$-problem constitutes an $\varepsilon_1$-optimal control.
\end{prp}

\begin{proof}
	For any perturbation $\beta\in\mathcal{U}_d^{N,0}$,
	\begin{equation*}
		J^{N,0}(\beta)\ge J^l(\beta)-KK_N\ge J^l(\hat{\alpha}^l)-KK_N\ge J^{N,0}(\hat{\alpha}^l)-KK_N.
	\end{equation*}
	For any $\varepsilon_1>0$, there exists $N_0^l$ such that, for any $N>N_0^l$,
	\begin{equation*}
		J^{N,0}(\beta)\ge J^{N,0}(\hat{\alpha}^l)-\varepsilon_1.\qedhere
	\end{equation*}
\end{proof}

\begin{proof}[Proof of Theorem \ref{Thm:POC}]
	The conclusion follows from Propositions \ref{Prp:5.12} and \ref{Prp:5.14}.
\end{proof}

\section{Special cases}\label{Sec:06}

\subsection{Indefinite LQ GMFG}

When $\alpha^l=0$, $b^l\equiv 0$, and $\sigma^l\equiv 0$ are fixed, we have $X^l=0$, and the problem no longer has a leader-follower structure. It degenerates into an indefinite LQ GMFG, where the diffusion term of the system equation is highly general and contains the player's state, control, and graphon aggregate:
\begin{equation*}
	\begin{cases}
		\mathrm{d}X_t^{f,u}=b_t^{f,u}(X_t^{f,u},\alpha_t^{f,u},GX_t^{f,u})\mathrm{d}t\\
		\qquad\qquad+\sigma_t^{f,u}(X_t^{f,u},\alpha_t^{f,u},GX_t^{f,u})\mathrm{d}W_t^{f,u},\\
		X_0^{f,u}=x_0^{f,u},\quad u\in I.
	\end{cases}
\end{equation*}
The quadratic cost functional contains the player's state, control, and graphon aggregate:
\begin{equation*}
	J^{f,u}(\alpha^{f,u};\alpha^{f,-u})=\mathbb{E}\biggl[ \int_0^T h_t^{f,u}(X_t^{f,u},\alpha_t^{f,u},GX_t^{f,u})\mathrm{d}t+g^{f,u}(X_T^{f,u},GX_T^{f,u}) \biggr].
\end{equation*}

Theorem \ref{Thm:3.1} can provide a Nash equilibrium for the above problem, and Theorem \ref{Thm:POC} gives the propagation of chaos result for the above problem.

\subsection{Stackelberg MFG}

When the graphon $G\equiv 1$ and $w\equiv 1$, the problem degenerates into an LQ indefinite Stackelberg MFG. The system equation is
\begin{equation*}
	\begin{cases}
		\mathrm{d}X_t^{f,u}=b_t^{f,u}(X_t^{f,u},\alpha_t^{f,u},\bar{M}_t^f,X_t^l,\alpha_t^l)\mathrm{d}t\\
		\qquad\qquad+\sigma_t^{f,u}(X_t^{f,u},\alpha_t^{f,u},\bar{M}_t^f,X_t^l,\alpha_t^l)\mathrm{d}W_t^{f,u},\\
		\mathrm{d}X_t^l=b_t^l(X_t^l,\alpha_t^l,\bar{M}_t^f)\mathrm{d}t+\sigma_t^l(X_t^l,\alpha_t^l,\bar{M}_t^f)\mathrm{d}W_t^l,\\
		X_0^{f,u}=x_0^{f,u},\quad u\in I,\quad X_0^l=x_0^l,
	\end{cases}
\end{equation*}
and the cost functionals are
\begin{gather*}
	J^{f,u}(\alpha^{f,u};\alpha^{f,-u},\alpha^l)=\mathbb{E}\biggl[ \int_0^T h_t^{f,u}(X_t^{f,u},\alpha_t^{f,u},\bar{M}_t^f,X_t^l,\alpha_t^l)\mathrm{d}t+g^{f,u}(X_T^{f,u},\bar{M}_T^f,X_T^l) \biggr],\\
	J^l(\alpha^l;\alpha^f)=\mathbb{E}\biggl[ \int_0^T h_t^l(X_t^l,\alpha_t^l,\bar{M}_t^f)\mathrm{d}t+g^l(X_T^l,\bar{M}_T^f) \biggr],
\end{gather*}
where
\begin{equation*}
	\bar{M}_t^f:=\int_I X_t^{f,u}\lambda(\mathrm{d}u)=\int_I \mathbb{E}[X_t^{f,u}|\mathcal{F}_T^l]\lambda(\mathrm{d}u).
\end{equation*}

Theorem \ref{Thm:3.1} gives the Stackelberg-Nash equilibrium for the above problem, and Theorem \ref{Thm:POC} gives the propagation of chaos result for the above problem.

\subsection{LQ GMFG with random effects}

When $b^l\equiv 0$ and $\sigma^l\equiv 0$, we have $X^l\equiv 0$, and $\alpha^l$ can be regarded as a random effect term. The problem degenerates into an LQ GMFG with random effects. Denote the random effect $\alpha^l$ by $\eta$. The system equation is
\begin{equation*}
	\begin{cases}
		\mathrm{d}X_t^{f,u}=b_t^{f,u}(X_t^{f,u},\alpha_t^{f,u},GX_t^{f,u},\eta_t)\mathrm{d}t\\
		\qquad\qquad+\sigma_t^{f,u}(X_t^{f,u},\alpha_t^{f,u},GX_t^{f,u},\eta_t)\mathrm{d}W_t^{f,u},\\
		X_0^{f,u}=x_0^{f,u},\quad u\in I,
	\end{cases}
\end{equation*}
and the cost functional is
\begin{equation*}
	J^{f,u}(\alpha^{f,u};\alpha^{f,-u})\\
	=\mathbb{E}\biggl[ \int_0^T h_t^{f,u}(X_t^{f,u},\alpha_t^{f,u},GX_t^{f,u},\eta_t)\mathrm{d}t+g^{f,u}(X_T^{f,u},GX_T^{f,u}) \biggr].
\end{equation*}

Theorem \ref{Thm:3.1} can provide a Nash equilibrium for the above problem, and Theorem \ref{Thm:POC} gives the propagation of chaos result for the above problem.

\section{Numerical simulation}
\label{Sec:07}

We consider the correction of cross-track errors in a UAV formation after a lateral-wind disturbance. A leader guides a large population of followers back to the centerline of the flight corridor. The scalar state $X$ denotes the cross-track displacement ($m$), and the control is the lateral guidance command. According to their communication strengths, the followers are divided into a high-connectivity block $I_H=[0,0.4]$ and a standard block $I_S=(0.4,1]$. The communication network is described by the two-block step graphon
\begin{equation}
	G(u,v)=
	\begin{cases}
		0.90,&u,v\in I_H,\\
		0.10,&(u,v)\in(I_H\times I_S)\cup(I_S\times I_H),\\
		0.30,&u,v\in I_S.
	\end{cases}
	\label{eq:uav-graphon}
\end{equation}
We take the graphon-operator eigenvalue $\lambda=0.37247$ and its normalized nonnegative eigenfunction
\begin{equation}
	w=1.90587\cdot\mathbf{1}_{I_H}+0.39609\cdot\mathbf{1}_{I_S}.
\end{equation}

Let $z_t^u=GX_t^{f,u}$ and $M_t=\int_0^1w^uX_t^{f,u}\lambda(\mathrm{d}u)$. The dynamics are
\begin{equation}\left\{\begin{aligned}\label{eq:uav-dynamics}
	\mathrm{d}X_t^{f,u}&=(a_fX_t^{f,u}+b_f\alpha_t^{f,u})\mathrm{d}t+\sigma_f \mathrm{d}W_t^{f,u},\\
	\mathrm{d}X_t^l&=(a_lX_t^l+b_l\alpha_t^l+c_lM_t)\mathrm{d}t+\sigma_l\mathrm{d}W_t^l. 
\end{aligned}\right.\end{equation}
The followers and leader's cost functionals are, respectively,
\begin{align}
	J^{f,u}={}&\frac12\mathbb E\!\left[\int_0^T\!\left\{
	q_f|X_t^{f,u}+\rho z_t^u-\kappa X_t^l|^2
	+r_f|\alpha_t^{f,u}|^2\right\}dt
	+g_f|X_T^{f,u}+\rho z_T^u-\kappa X_T^l|^2\right],\nonumber\\
	J^l={}&\frac12\mathbb E\!\left[\int_0^T\!\left\{
	q_l|X_t^l-\kappa_lM_t|^2+r_l|\alpha_t^l|^2\right\}dt
	+g_l|X_T^l-\kappa_lM_T|^2\right]. \label{eq:uav-costs}
\end{align}
Thus, each follower tracks the leader while suppressing the communication-weighted cross-track bias. We use
\[
\begin{gathered}
	T=12,\quad (a_f,b_f,\sigma_f)=(-0.35,1,0.22),\quad
	(a_l,b_l,c_l,\sigma_l)=(-0.20,1,0.08,0.06),\\
	(q_f,r_f,\rho,\kappa,g_f)=(2,0.5,0.45,0.85,4),\quad
	(q_l,r_l,\kappa_l,g_l)=(4,1.2,0.25,8),\\
	(X_0^l,X_0^H,X_0^S)=(5,7,-2).
\end{gathered}
\]

It is readily verified that these parameters satisfy the conditions of the corresponding results in the paper.

\begin{figure}[H]
	\centering
	\includegraphics[width=.98\linewidth]{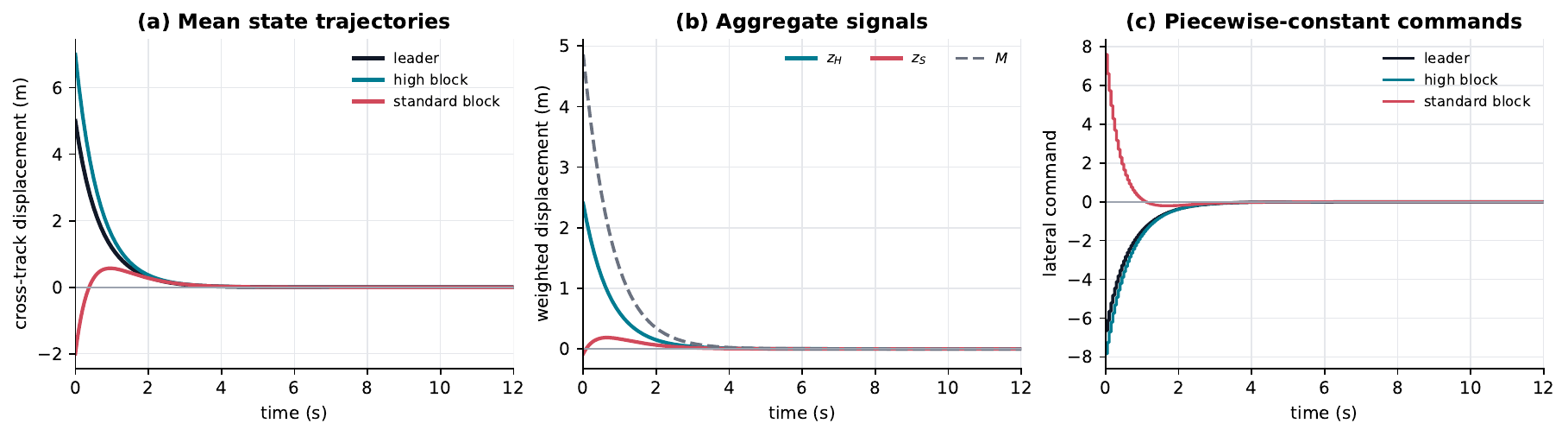}
	\caption{Equilibrium results: (a) mean state trajectories; (b) aggregate signals; (c) lateral guidance commands.}
	\label{fig:uav-equilibrium}
\end{figure}

In Figure~\ref{fig:uav-equilibrium} (a), all three mean state trajectories converge to zero. In panel (b), $z_H$, $z_S$, and $M$ decay simultaneously, and the relatively large $z_H(0)$ agrees with the stronger within-block connectivity of the high-connectivity population. In panel (c), each initial command has the opposite sign to the corresponding initial displacement and converges to zero as the tracking error vanishes. The staircase appearance results from the piecewise-constant implementation with numerical step size $h=0.05$. These results show that the equilibrium uses the state errors and graphon aggregates as feedback signals to return the formation to the corridor centerline, consistently with the mechanism predicted by the model.

\section{Conclusion}\label{Sec:08}

This paper studied a general leader-follower LQ stochastic graphon game with indefinite control weights and general diffusion terms, where the coupling between the leader and the followers occurs through their state equations and cost functionals. A rigorous formulation was established under a rich Fubini extension and the CELLN, which ensures the well-posedness of the controlled systems and the required adaptedness and integrability of the graphon aggregate and weighted-average terms.

For the limiting problem, stochastic maximum principles were used to derive graphon aggregated FBSDEs for the followers' Nash response and the leader's optimal control, which were further analyzed and simplified by Riccati equations. In the resulting Stackelberg structure, the leader anticipates the followers' Nash response under a given leader control and optimizes its own cost accordingly. Under suitable solvability and convexity conditions, a Stackelberg-Nash equilibrium was constructed.

The limiting equilibrium was further used to construct decentralized strategies for finite-player network games. When the finite graphs, initial states, and leader weights converge to their graphon limiting counterparts, the induced graphon aggregate terms, weighted-average terms, and leader state converge to the corresponding limiting quantities. Based on these convergence results and cost comparisons, the decentralized strategies were shown to form an approximate Stackelberg-Nash equilibrium, yielding a propagation of chaos result.

The results also cover several existing models as special cases and extend related LQ-GMFG results by allowing more general diffusion couplings, indefinite control weights, and graphon operators beyond the finite-rank setting. Several problems remain worthy of further study, including index-dependent coefficients, partially observed control problems, and applications to finance and other areas.

\end{document}